\documentclass[12pt,twoside]{amsart}

\usepackage{amssymb,latexsym,bbm,graphicx,epsfig,epic,eepic,oldgerm,psfrag,tikz,wasysym}
\usepackage{lipsum}

\usepackage{tikz-cd}
\usepackage{marvosym} %wieder dazu!                     
\usepackage{a4wide}     %       
\usepackage{latexsym}       
\usepackage{amscd}
\usepackage{graphics, psfrag}
\usepackage{amsmath}
\usepackage{amssymb}                                                                      
\usepackage{soul,xcolor}
\usepackage{esint}
\usepackage{mathrsfs}
\usepackage{amsthm}
\usepackage{bm}                    
\usepackage{bbm}
\usepackage{color}
\usepackage{accents}
\usepackage{enumerate,enumitem}
\usepackage{textgreek}
\usepackage{stmaryrd,wasysym}
\usepackage{comment}

\usepackage{xypic} 
\usepackage{xfrac}
\input xy
\xyoption{all}

\newcommand{\N}{\mathbb{N}}                 
\newcommand{\Z}{\mathbb{Z}}                  
                  
\newcommand{\R}{\mathbb{R}}                    
                    
\newcommand{\D}{\mathbb{D}}      

\newcommand{\HH}{\mathbb{H}}
            
\newcommand{\T}{\mathbb{T}}                     
                     
\newcommand{\finedim}{\hfill $\Box$}            % end of proof
\newcommand{\area}{\mathrm{area}}    
\newcommand{\supp}{\mathrm{supp\,}}             % support
\newcommand{\conv}{\mathrm{conv \2}}            % convex hull       
            
\newcommand{\rank}{\mathrm{rank\,}}             % rank
\newcommand{\RP}{\mathbb{RP}}                   % real projective space
\newcommand{\pr}{{\mathrm{pr}}}     
\newcommand{\SL}{\mathrm{SL}}
\newcommand{\SO}{\mathrm{SO}}

\newcommand{\id}{\mathrm{id}}         
            
\newcommand{\can}{\mathrm{can}\,}                
           
\newcommand{\proofend}{\hspace*{\fill} $\Box$\\}

\newenvironment{brsm}{
  \bigl( \begin{smallmatrix} }{
  \end{smallmatrix} \bigr)}

\def\rev{\operatorname{rev}}
\def\top{\operatorname{top}}

\def\Vol{\operatorname{Vol \1}}
\def\vol{\operatorname{vol}}
\def\sym{\operatorname{sym}}

\def\length{\operatorname{length}\1}
\def\card{\operatorname{card}}
\def\HT{\operatorname{HT}}
\def\BH{\operatorname{BH}}
\def\MT{\operatorname{MT}}
\def\OB{\operatorname{OB}}
\def\SH{\operatorname{SH}}
\def\WH{\operatorname{WH}}
\def\Ad{\operatorname{Ad}}

\def\ga{\alpha}

\def\gve{\varepsilon}

\def\gl{\lambda}

\def\1{\:\!}
\def\2{\;\!}
\def\m2{\!\!\:}

\def\cc{{\mathcal C}}

\def\cf{{\mathcal F}}
\def\cg{{\mathcal G}}
\def\ch{{\mathcal H}}

\def\cl{{\mathcal L}}
\def\cm{{\mathcal M}}

\def\cR{{\mathcal R}}

\def\ct{{\mathcal T}}

\def\ovr{{\bold r}}
\def\ox{{\bold x}}
\def\otheta{\textnormal{\fontfamily{phv}\selectfont\straighttheta}}

\newtheorem{theorem}{Theorem}[section]           % numbered absolutely
\newtheorem*{theorem*}{Theorem}              
\newtheorem{corollary}[theorem]{Corollary}        % numbered along with Theorem
\newtheorem*{corollary*}{Corollary}       
\newtheorem{lem}[theorem]{Lemma}            
\newtheorem{prop}[theorem]{Proposition}     
\newtheorem{defn}[theorem]{Definition}      
\newtheorem{rem}[theorem]{Remark}      
\newtheorem{rems}[theorem]{Remarks}     
         
\newtheorem{exs}[theorem]{Example}
    
\newtheorem{question}[theorem]{Question}

\begin{document}

\numberwithin{equation}{section}

\title[]{Entropy collapse versus entropy rigidity for Reeb and Finsler flows}

\author[]{Alberto Abbondandolo}
%\thanks{}
\address{Alberto Abbondandolo,
  Fakult\"at f\"ur Mathematik, 
  Ruhr-Universit\"at Bochum}
\email{alberto.abbondandolo@rub.de}

\author[]{Marcelo R.R.\ Alves}
%\thanks{}
\address{Marcelo Alves,
Faculty of Science, University of Antwerp}
\email{marcelorralves@gmail.com}

\author[]{Murat Sa\u glam}
%\thanks{}
\address{Murat Sa\u glam,
  Mathematisches Institut, 
	Universit\"at zu K\"oln}
\email{msaglam@math.uni-koeln.de}                          
								
\author[]{Felix Schlenk}  
\address{Felix Schlenk,
Institut de Math\'ematiques,
Universit\'e de Neuch\^atel}
\email{schlenk@unine.ch}


\date{\today}

\begin{abstract}
On every closed contact manifold there exist contact forms with volume one whose Reeb flows have arbitrarily small topological entropy.
In contrast, for many closed manifolds there is a uniform positive lower bound for the topological entropy
of (not necessarily reversible) normalized Finsler geodesic flows.
\end{abstract}

\maketitle

\begin{center}
Dedicated to the memory of Will Merry, 1984--2022
\end{center}

\tableofcontents

\section{Introduction}

\subsection{Main results}
The main results of this paper are the following two theorems.

\begin{theorem} \label{t:mainintro}
Let $(M,\xi)$ be a closed co-orientable contact manifold. 
For every $\gve >0$ there exists a contact form $\alpha$ on~$(M,\xi)$ with volume one
such that the topological entropy $h_{\top}(\alpha)$ of its Reeb flow is smaller than~$\gve$.
\end{theorem}

Given a closed manifold $Q$ let $h_{\vol} (Q)$ be the infimum of the volume entropies
of Riemannian metrics on~$Q$ that have volume one.
This number is equal to $2 \sqrt{\pi (k-1)}$ for a closed orientable surface of genus~$k \geq 2$, 
and it is positive for instance if $Q$ admits a Riemannian metric of negative curvature.
Given a Finsler metric~$F$ on~$Q$ we denote by $h_{\top} (F)$ the topological entropy of the time-one map of the
geodesic flow of~$F$.
Define the dimension constants 
$$
c_n :=  \frac{1}{(n!\, \omega_n)^{1/n}} 
$$
where $\omega_n$ is the volume of the Euclidean unit ball in~$\R^n$. 
For instance $c_2 = \frac{1}{\sqrt{2\pi}}$, and asymptotically $c_n \sim \sqrt{\frac{e}{2\pi}} \frac{1}{\sqrt n}$.

\begin{theorem} \label{t:Finslerintro}
Let $Q$ be a closed connected $n$-dimensional manifold. 
Then for every Finsler metric~$F$ on~$Q$ of Holmes--Thompson volume one
it holds that
$$
h_{\top} (F) \,\geq\, c_n \, h_{\vol} (Q) , \;\;
$$
and if $F$ is symmetric that
$$
h_{\top} (F) \,\geq\,  2 \2 c_n \, h_{\vol} (Q) .
$$
\end{theorem}

\medskip
In the rest of this introduction, we recall the notions appearing in these theorems, 
describe in more detail the results proved in this paper, 
put them into context, and formulate a few open problems they give rise to. 
We first tell our story for the special case of unit circle bundles over closed orientable surfaces
of higher genus. 
Most ideas are present already for these simple spaces.
We keep the presentation informal, referring to the subsequent sections for the precise definitions 
and arguments.

\subsection{The case of unit circle bundles over higher genus surfaces}
\label{ss:circle}

Let $Q_k$ be the closed orientable surface of genus $k \geq 2$.
For every Riemannian metric~$g$ on~$Q_k$ we consider the geodesic flow~$\phi_g^t$
on the unit circle bundle  
$$
\left\{ (q,v) \in TQ_k \mid g_q(v,v) = 1 \right\} .
$$
A good numerical measure for the complexity of the flow $\phi_g^t$ is the topological entropy
$h_{\top} (g) := h_{\top} (\phi_g^1)$. 
A definition can be found in Appendix~\ref{a:Manning}.
This is an interesting invariant because it is related to many other complexity measurements 
of~$\phi^t_g$, see~\cite{HK95}.

For which Riemannian metrics~$g$ is $h_{\top} (g)$ minimal?
Such a $g$ could then rightly be considered as a best Riemannian metric from a dynamical point of view. 
Since the topological entropy scales like 
\begin{equation} \label{e:scal}
h_{\top} (c \1 g) \,=\, \frac 1c \, h_{\top} (g) ,
\end{equation}
the problem is meaningful only if one imposes a normalization. 
We normalize by the Riemannian area and consider the scale invariant quantity 
\begin{equation} \label{e:hatarea}
\widehat h_{\top}(g) \,=\, \sqrt{\area_g (Q_k)} \, h_{\top} (g) .
\end{equation}
It is a classical theorem of Dinaburg~\cite{Din71} and Manning~\cite{Man79}
that the geodesic flow of any Riemannian metric on~$Q_k$
has positive topological entropy (cf.\ Appendix~\ref{a:Manning} below).
Their results do not give a uniform positive lower bound on~$\widehat h_{\top}(g)$
nor do they say anything about the minimizers, however.  
This was achieved in the following remarkable result of Katok~\cite{Kat83}.

\begin{theorem} \label{t:katok}
{\rm (Katok 1983)}
For every Riemannian metric $g$ on $Q_k$ it holds that
$$
\widehat h_{\top}(g) \,\geq\, 2 \sqrt{\pi (k-1)} .
$$
Moreover, equality holds if and only if $g$ has constant curvature.
\end{theorem}

Geodesic flows are very special Reeb flows. 
For our unit circle bundle over~$Q_k$, Reeb flows can be described as follows.
We look at the cotangent bundle $T^*Q_k$ instead of the tangent bundle, 
endowed with its usual symplectic form $\omega = d \lambda$, where $\lambda = \sum_{j=1}^2 p_j\2 dq_j$.
Let $H \colon T^*Q_k \to \R$ be a continuous function that is smooth and positive away from the zero-section and fiberwise positively homogenous of degree one: 
$H(q, r p) = r H(q,p)$ for all $r \geq 0$.
Then $S^*(H) := H^{-1} (1)$ is a smooth hypersurface of~$T^*Q_k$
with the property that 
for each point $q \in Q_k$ the intersection $S_q^* (H) := S^*(H) \cap T_q^*Q_k$
with the cotangent plane at~$q$ 
is the smooth boundary of a domain which is starshaped with respect to the origin~$0_q$,
see the left drawing in Figure~\ref{star.fig}.
Denote by $\phi_H^t$ the restriction of the Hamiltonian flow of~$H$ to~$S^*(H)$.
The class of these flows are the Reeb flows on our unit circle bundle.
This flow is a co-geodesic flow exactly if $H$ restricts on each fiber to the square root 
of a positive quadratic form.
Special shapes of the fibers $S_q^* (H)$ in $T_q^*Q_k$ correspond to special Reeb flows:

\smallskip
\begin{itemize}
\item[($\varbigcirc$)] 
$\phi_H^t$ is a Riemannian geodesic flow if and only if each $S_q^* (H)$ is a centrally symmetric ellipse. 

\item[(\1$\square$\1)]
$\phi_H^t$ is a reversible Finsler geodesic flow if and only if each $S_q^* (H)$ 
is a centrally symmetric closed smooth curve with strictly positive curvature.

\medskip 
\item[($\triangle$)]
$\phi_H^t$ is a Finsler geodesic flow if and only if each $S_q^* (H)$ is a closed smooth curve with strictly positive curvature.
\end{itemize}

\smallskip
\noindent
Here we identified co-Finsler geodesic flows with Finsler geodesic flows via the Legendre transform.

\begin{figure}[h]   
 \begin{center}
  \psfrag{a}{(a)} \psfrag{b}{(b)} \psfrag{c}{(c)} \psfrag{d}{(d)}       
  \leavevmode\includegraphics{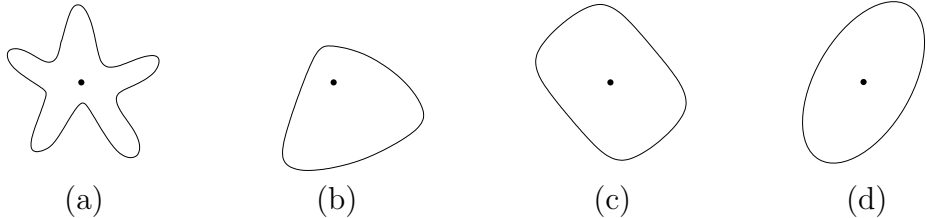}
 \end{center}
 \caption{Spheres $S_q^* (H)$ in $T_q^*Q_k$ defining (a) a Reeb flow, 
(b) a Finsler geodesic flow, (c) a reversible Finsler geodesic flow, 
(d) a Riemannian geodesic flow.}   \label{star.fig}
\end{figure}

Based on \cite{FS} it was shown in~\cite{MacSch11} that the above result of Dinaburg and Manning 
about Riemannian geodesic flows extends to all Reeb flows:

\begin{theorem} \label{t:MacSch}
Every Reeb flow $\phi_H^t$ on $S^*(H)\subset T^* Q_k$, $k\geq 2$, has positive topological entropy. 
\end{theorem}

Does Katok's theorem also extend to Reeb flows?
To make the question meaningful, we again need to normalize. 
We do this by the symplectic volume of the bounded domain~$D^*(H)$
in~$T^*Q_k$ with boundary~$S^* (H)$, and define the Holmes--Thompson volume of~$Q_k$ 
associated with~$H$ by
\begin{equation} \label{e:hatDH}
\vol_H^{\HT} (Q_k) \,=\, \frac{1}{2\pi}  \int_{D^*(H)} \omega\wedge \omega.
\end{equation}
Then the normalized topological entropy
$$
\widehat h_{\top}^{\HT} (H) \,:=\, \sqrt{\vol_H^{\HT}(Q_k)} \; h_{\top} (\phi_H^1)
$$
is invariant under scalings of~$H$.
In the Riemannian case, this definition agrees with~\eqref{e:hatarea},
since then
$\vol_H^{\HT} (Q_k) = \area_g (Q_k)$.
The following question was asked in~\cite[\S 7.2]{FLS13}.

\begin{question} \label{q:FLS}
Is there a positive constant $c(k)$ such that $\widehat h_{\top}^{\HT} (H) \geq c(k)$
for every Reeb flow on the co-circle bundle over~$Q_k$\2?
\end{question}

Let us first try to answer this question in the affirmative for Finsler geodesic flows.
Given a Finsler metric~$F$ on~$Q_k$,
an obvious idea is to find a lower bound for $\widehat h_{\top}^{\HT}(F)$
by choosing a larger Riemannian metric $\sqrt{g} \geq F$, cf.~\eqref{e:scal}.
In general, the topological entropy of geodesic flows is not monotone with respect 
to the order relation on metrics, however. 
We therefore pass to a more geometric version of entropy, which is indeed monotone: 
The {\it volume entropy}\/ of~$F$ is defined as the exponential growth rate of balls in the universal cover~$\widetilde Q_k$
(which is the plane),
\begin{equation} \label{e:hvoldef}
h_{\vol} (F) \,:=\, \lim_{R \to \infty} \frac 1R \1 \log \Vol (B_{\tilde q}(R))
\end{equation}
where $\tilde q$ is any fixed point in $\widetilde Q_k$, $B_{\tilde q} (R)$ is the ball of radius~$R$
about~$\tilde q$ with respect to the lifted Finsler metric, and $\Vol$ is the volume with respect to 
the lift of any smooth area form on~$Q_k$ (see Appendix~\ref{a:Manning} for details).
It is clear that $F_1 \geq F_2$ implies
\begin{equation}  \label{e:hvolFg}
h_{\vol} (F_2) \, \geq \, h_{\vol} (F_1) .
\end{equation}
In the case of a Riemmannian metric $g$, 
denoting $h_{\vol}(\sqrt{g})$ simply by~$h_{\vol}(g)$, we have that
\begin{equation}  \label{e:Man}
h_{\top} (g) \, \geq \, h_{\vol} (g) 
\end{equation}
%
%FFF hier ???
with equality if $g$ has non-positive curvature, as proven by Manning in~\cite{Man79}.
His proof of~\eqref{e:Man} readily generalizes to all Finsler metrics, 
see Appendix~\ref{a:Manning}:
\begin{equation}  \label{e:ManF}
h_{\top} (F) \, \geq \, h_{\vol} (F) . 
\end{equation}
Let $g$ be a Riemannian metric such that $\sqrt{g}\geq F$. 
Using \eqref{e:ManF} and~\eqref{e:hvolFg} we can now estimate
\begin{eqnarray*}
\widehat h_{\top}^{\HT} (F) &:=& \sqrt{\vol_F^{\HT}(Q_k)} \,h_{\top} (F)  \\
&\geq&
\sqrt{\vol_F^{\HT}(Q_k)} \, h_{\vol} (g)  \\
&=&
\sqrt{\frac{\vol_F^{\HT} (Q_k)}{\vol_g^{\HT} (Q_k)}} \,  \, \widehat{h}_{\vol} (g) .
\end{eqnarray*}
In~\cite{Kat83}, Katok actually proved Theorem~\ref{t:katok} for the normalized volume entropy~$\widehat h_{\vol}$ 
(which by Manning's theorem implies Theorem~\ref{t:katok}).
Hence we obtain
\begin{eqnarray} \label{e:hFg}
\widehat h_{\top}^{\HT} (F) 
\, \geq \, \sqrt{\frac{\vol_F^{\HT} (Q_k)}{\vol_g^{\HT} (Q_k)}} \,  \, 2 \sqrt{\pi (k-1)} . 
\end{eqnarray}
To get a uniform lower bound for $\widehat h_{\top}^{\HT} (F)$
we therefore look for a Riemannian metric~$g$ with $\sqrt{g} \geq F$ 
that is as close as possible to~$F$ in the sense of the Holmes--Thompson volume.
We best do this directly in the cotangent bundle~$T^*Q_k$.  
We thus look at each $q \in Q_k$ for a centrally symmetric ellipse in~$T_q^* Q_k$
such that, denoting by~$E_q$ the region bounded by it, we have
$E_q \supset D_q^* (F)$ and $E_q$ is as close to~$D_q^*(F)$ in volume as possible.

If $D_q^*(F)$ is centrally symmetric, the best choice is Loewner's outer ellipse.
This is the unique centrally symmetric ellipse enclosing $D_q^*(F)$ which minimizes the value of the area of the region bounded by it, which we denote by $E(D_q^*(F))$. Here the area $| \; |$ is taken with respect to any translation invariant measure on the plane~$T_q^*Q_k$.
Loewner's ellipse depends continuously on~$q$, and the largest area ratio
$$
\frac{| E(D_q^*(F)) | } { |D_q^*(F)|} 
$$
is~$\frac \pi 2$, which is attained exactly  when $D_q^*(F)$ is a parallelogram.
If we take the Riemannian metric~$g$ on~$Q_k$ that has the sets~$E(D_q^*(F))$  
as unit co-disks, 
we therefore obtain 
$$
\vol_F^{\HT} (Q_k)   \,\geq\, \frac{2}{\pi} \vol_g^{\HT} (Q_k) .
$$
Together with~\eqref{e:hFg} this yields
$$
\widehat h_{\top}^{\HT} (F) \,\geq\, 
\sqrt{\frac{2}{\pi}} \, 2 \sqrt{\pi (k-1)} \,=\, 2 \sqrt{2 (k-1)}. 
$$

If $D_q^* (F)$ is not centrally symmetric, we observe that the convex hull 
$$
\conv \bigl( D_q^*(F) \cup - D_q^*(F) \bigr)
$$
is centrally symmetric. 
It is not hard to see that for every convex body $K \subset \R^2$ that contains the origin,
$$
\left| \conv (K \cup -K) \right| \,\leq\, 4 |K| 
$$
with equality attained exactly by the triangles with one vertex at the origin.
Therefore, 
$$
\frac{|E (\conv (K \cup -K))|}{|K|} \,\leq\, 4 \cdot \frac{\pi}{2}
\ \,=\, 2\pi .
$$
Note that the constant $2\pi$ is sharp and is attained exactly by the triangles with one vertex at the origin,
see Figure~\ref{E.fig}~(b).
\begin{figure}[h] 
 \begin{center}
  \psfrag{a}{(a)} \psfrag{b}{(b)} \psfrag{c}{(c)}  
 \psfrag{K}{$K$}  \psfrag{-K}{$-K$}   
  \leavevmode\includegraphics{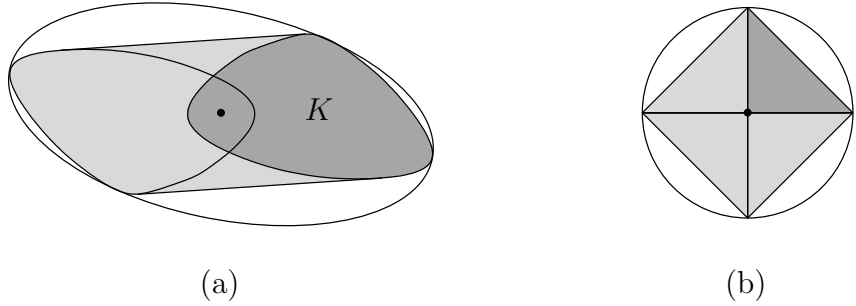}
 \end{center}
 \caption{The symmetrization $\conv (K \cup -K)$}   \label{E.fig}
\end{figure}
Since the two maps 
$$
K \,\mapsto\, \conv (K \cup -K) \,\mapsto\, E(K \cup -K)
$$ 
are continuous, 
we can take as $g$ the Riemannian metric with unit co-disks  
$$
E \bigl( \conv \bigl( D_q^*(F) \cup - D_q^*(F) \bigr)  \bigr)
$$
and obtain
$$
\widehat h_{\top}^{\HT} (F) \,\geq\, 
\frac{1}{\sqrt{2\pi}}\, 2 \sqrt{\pi (k-1)} \,=\, \sqrt{2(k-1)}. 
$$

Summarizing, we obtain Theorem~\ref{t:Finslerintro} for orientable surfaces:
\begin{equation} \label{e:FQk}
\widehat h_{\top}^{\HT} (F) \,\geq\, \frac{1}{\sqrt{2\pi}} \, h_{\vol} (Q_k) , 
\quad \mbox{and} \quad 
\widehat h_{\top}^{\HT} (F) \geq \sqrt{\frac 2 \pi} \, h_{\vol} (Q_k)  \, \mbox{ if $F$ is symmetric.}
\end{equation}

How sharp are these lower bounds?
It is still unknown whether the constants
$\frac{1}{\sqrt{2\pi}}$ and~$\sqrt{\frac 2 \pi}$ 
can be replaced by~$1$.
In other words, it is unknown whether there exist Finsler metrics~$F$ on~$Q_k$
such that $\widehat h_{\vol}^{\HT} (F) < h_{\vol} (Q_k)$.
We shall say more on this ``minimal entropy problem" in \S \ref{ss:entprob}.

Recall that for the closed orientable surfaces $Q_k$ of genus~$k$ one has
$$
h_{\vol} (Q_k) = 2 \sqrt{\pi (k-1)} .
$$
For the non-orientable surface $P_k$ whose orientation cover is $Q_k$ this implies
$$
h_{\vol} (P_k) = \sqrt{2 \pi (k-1)} .
$$
For the other four closed surfaces 
(the sphere, the torus, the real projective plane and the Klein bottle) 
Theorem~\ref{t:Finslerintro} is not useful since $h_{\vol}$ vanishes, 
and in fact there exist geodesic flows on these surfaces with vanishing topological entropy.

\medskip
We now look at general Reeb flows on the co-circle bundle over~$Q_k$.
As said earlier, these flows correspond to Hamiltonian flows on $S^*(H) = H^{-1}(1)$
of Hamiltonian functions $H \colon T^*Q_k \to \R$ that are fiberwise homogeneous of degree~one.
Looking for a lower bound for $\widehat h_{\top}^{\HT} (\phi_H)$,
we proceed as in the case of Finsler geodesic flows,
but knowing already~\eqref{e:FQk} we now compare~$H$ with any Finsler metric.
Choose a Finsler Hamiltonian $F \colon T^*Q_k \to \R$
such that $D^*(H) \subset D^*(F)$, i.e., $F \leq H$.

Definition~\eqref{e:hvoldef} can be extended to Reeb flows:
Fix a point $q \in Q_k$, take a lift $\tilde q \in \widetilde Q_k$ of~$q$
and the lift $\widetilde H \colon T^*\widetilde Q_k \to \R$ of~$H$,
and then define $h_{\vol}(H,q)$ as the exponential growth rate
of the volume of the set $B_{\tilde q}(\widetilde H, T)$ 
of those points $z \in \widetilde Q_k$ for which the fiber $S^*_z(\widetilde H)$ can be reached
in time~$\leq T$ by a flow line of~$\phi_{\widetilde H}^t$ that starts at the fiber~$S^*_{\tilde q}(\widetilde H)$.
As we shall show in Appendix~\ref{a:Manning} one then still has Manning's inequality, 
$$
h_{\top}(H) \,\geq\, h_{\vol}(H,q) .
$$
%We now wish to compare $h_{\vol}(H,q)$ with $h_{\vol}(G)$.
%While it is easy to see that there exists a constant $c>0$ depending only on $H$ and~$G$
%such that $h_{\vol}(H,q) \leq h_{\vol}(G)$,  
%the opposite inequality is not clear at all, because for non-convex~$H$
We now wish to show that there is a constant $c>0$ depending only on $H$ and~$F$
such that $h_{\vol}(H,q) \geq c\, h_{\vol}(F)$.  
The existence of such a constant for non-convex~$H$ does not follow from geometric considerations, 
since it is not true in general that $F \leq H$ implies the inclusion of balls 
$B_{\tilde q} (\widetilde F,T) \subset B_{\tilde q}(\widetilde H, T)$.
However, using Lagrangian Floer homology in $T^*Q_k$ one can avoid passing through~$h_{\vol}(H, q)$
and prove directly that 
\begin{equation} \label{e:hsigmaG}
\widehat{h}_{\top}^{\HT} (H) \,\geq\, \frac{1}{\sigma(H;F)} \, \widehat h_{\vol} (F) 
\end{equation}
where $\sigma (H;F)$ is the smallest real number such that 
$\frac{1}{\sigma (H;F)} D^*(F) \subset D^*(H)$, cf.\ Figure~\ref{fig.GH}.
This is explained in \S \ref{s:Floer}, using the proof of 
the above Theorem~\ref{t:MacSch} from~\cite{MacSch11}.

\begin{figure}[h]   
\begin{center}
 \leavevmode\includegraphics{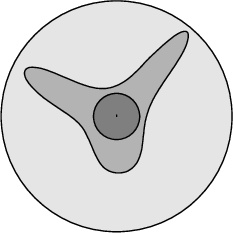}
\end{center}
 \caption{The co-disks $\frac{1}{\sigma(H;F)} D^*_q (F) \subset D^*_q (H) \subset D^*_q (F)$
          in $T_q^*Q_k$}    
\label{fig.GH}
\end{figure}

\noindent
The number $\sigma (H) := \inf \{ \sigma (H;F) \mid F \leq H \}$
is a measure for how far the fibers of $D^*(H)$ are from being convex.
Inequalities~\eqref{e:hsigmaG} and~\eqref{e:FQk} and Katok's inequality 
imply 
%the following result.

\begin{prop} \label{p:Floer2}
For every Reeb flow~$\phi_H$ on the co-circle bundle over $Q_k$,
\begin{equation} \label{e:hHs}
\widehat{h}_{\top}^{\HT} (H) \,\geq\, \frac{1}{\sigma (H)} \, \sqrt{2 (k-1)} .
\end{equation}
\end{prop}

The following special case of Theorem~\ref{t:mainintro} shows that 
the lower bound in~\eqref{e:hHs} cannot be made uniform, 
that the answer to Question~\ref{q:FLS} is `no', and that
there is no way to extend Katok's rigidity theorem to Reeb flows.

\begin{theorem} \label{t:collapseS}
For every $\gve >0$ there exists a Reeb flow $\phi_H^t$ on~$S^*(H)$ with 
$\widehat h_{\top}^{\HT} (H) \leq \gve$.
\end{theorem}

Proposition \ref{p:Floer2} shows that this entropy collapse cannot happen 
unless at least some of the co-disks $D_q^*(H) = D^*(H) \cap T_q^*Q_k$ are 
very far from convex.
Writing down explicitely such star fields on~$T^*Q_k$ that lead to small $\widehat h_{\top}^{\HT}$ 
seems difficult, however. 
In fact, our proof of Theorem~\ref{t:collapseS} does not use the special fibration structure of~$S^*(H)$,
but uses the existence of open book decompositions valid for all closed 3-manifolds,
see the beginning of \S \ref{ss:proof3} for an outline and \S \ref{s:collapse3} for the proof.

\subsection{Entropy rigidity for Finsler geodesic flows} \label{ss:entrig}

Proceeding as in the previous section, one readily arrives at Theorem~\ref{t:Finslerintro}
for Finsler geodesic flows on closed manifolds~$Q$ of arbitrary dimension~$n$,
\begin{equation} \label{e:FQ}
\widehat h_{\top}^{\HT} (F) \,\geq\, c_n \, h_{\vol} (Q) , 
\quad \mbox{and} \quad 
\widehat h_{\top}^{\HT} (F) \geq  2 \2 c_n \, h_{\vol} (Q) \, \mbox{ if $F$ is symmetric.}
\end{equation}
Here the normalization $\widehat h_{\top}^{\HT} (F) = \left(\vol_F^{\HT}(Q)\right)^{1/n} \, h_{\top}(F)$
is done in terms of the Holmes--Thompson volume
\begin{equation} \label{def:HTF}
\vol_F^{\HT} (Q) \,:=\, \frac{1}{n!\, \omega_n} \int_{D^*(F)} \omega^n ,
\end{equation}
which extends definition~\eqref{e:hatDH}. 
The proof of~\eqref{e:FQ} uses Loewner's outer ellipsoids and the Roger--Shephard volume bounds for symmetrized convex bodies. Similar arguments appear in~\cite{apbt16}, 
where they are used to derive systolic inequalities for Finsler metrics from the 
analogous inequalities for Riemannian metrics.

For Finsler metrics there is another natural volume,
the Busemann--Hausdorff volume. 
For reversible Finsler metrics, this volume is at least
the Holmes--Thompson volume, see~\S \ref{s:ent}.
The second inequality in~\eqref{e:FQ} thus also holds true if
we normalize by the Busemann--Hausdorff volume.
For irreversible Finsler metrics, however, we do not know
whether the first inequality in~\eqref{e:FQ} holds true for 
the Busemann--Hausdorff volume.

\smallskip
For manifolds of dimension $n \geq 3$, it is more difficult to understand 
the volume entropy~$h_{\vol}(Q)$ than for surfaces.
The only sharp result is the following extension of Katok's theorem.

\begin{theorem} \label{t:BCG}
{\rm (Besson--Courtois--Gallot\cite{BCG, BCG1})}
If $Q$ is a closed manifold of dimension at least~$3$ 
that admits a locally symmetric Riemannian metric~$g_0$ of negative curvature, 
then
$$
\widehat h_{\vol}(g) \,\geq\, \widehat h_{\vol} (g_0) 
$$
for every Riemannian metric $g$ on $Q$,
and equality holds if and only if $g$ is also locally symmetric.
In particular, $h_{\vol}(Q) = \widehat h_{\vol} (g_0) >0$.
\end{theorem}

Note that the space of minimizers up to isometry in Katok's theorem is the $6k-6$ dimensional 
Teichm\"uller space, 
while the minimizers in~Theorem~\ref{t:BCG} are all isometric up to scaling, by Mostow's theorem.  

In the context of Theorem~\ref{t:Finslerintro} we wish to know when $h_{\vol} (Q) >0$.
%Even though $h_{\vol} (Q)$ is a homotopy invariant~\cite{Bab92}, it is hard to estimate directly.
%The main tool for proving positivity of the minimal volume is another homotopy invariant of~$Q$, 
%the {\it simplicial volume}~$\|Q\|$. 
The main tool for proving $h_{\vol} (Q) >0$ is the {\it simplicial volume}~$\|Q\|$. 
If $Q$ is orientable, 
it is defined as $\inf \sum_i |r_i|$ where the infimum is taken over 
those sums $\sum_i r_i \sigma_i$ that represent the fundamental class $[Q] \in H_n (Q;\R)$ with real coefficients.
If $Q$ is not orientable, pass to the orientation double covering $\widehat Q$ and put $\|Q\| = \frac 12 \|\widehat Q\|$.
Gromov proved in~\cite{Gro82} that
$$
h_{\vol} (Q) \,\geq\, C_n^{-1} \, \|Q\|^{1/n}  
$$
for an explicit dimension constant $C_n$.

There are many more manifolds $Q$ with positive simplicial volume~$\|Q\|$ than those in Theorem~\ref{t:BCG}.
Indeed, $\|Q\| >0$ for all manifolds that admit a Riemannian metric of negative curvature, 
and positivity of the simplicial volume is preserved under taking 
the product with any other closed manifold of positive simplicial volume and under
taking the connected sum with any other closed manifold of the same dimension.
We refer to~\cite{Gro82} and~\cite{Loe} for more examples and information on simplicial volume.

\subsection{Entropy collapse for Reeb flows} \label{ss:collapse-intro}

Reeb flows are flows naturally associated to contact manifolds. 
A contact structure~$\xi$ on a ($2n-1$)-dimensional manifold~$M$ is a maximally non-integrable hyperplane field of the tangent bundle~$TM$.
We assume throughout that $\xi$ is co-orientable, i.e., 
$\xi = \ker \alpha$ for a 1-form~$\alpha$ on~$M$.
In terms of such a form~$\alpha$, called a contact form for~$\xi$, 
the maximal non-integrability means that $\alpha \wedge (d\alpha)^{n-1}$ is a volume form on~$M$.
For any non-vanishing function $f$ on~$M$ the $1$-form $f \alpha$ is also a contact form on~$(M,\xi)$.
Each contact form~$\alpha$ gives rise to the Reeb flow~$\phi_\alpha^t$, which is generated by
the Reeb vector field~$R_\alpha$ implicitly defined by the two conditions
$$
d\alpha (R_\alpha, \cdot ) =0, \qquad \alpha (R_\alpha) = 1.
$$ 

For every closed manifold $Q$ the so-called spherization $(S^*Q,\xi_{\can})$
is a contact manifold whose Reeb flows are exactly the flows~$\phi_H^t$ on~$S^*(H)$
described in \S \ref{ss:circle} in the case of closed surfaces~$Q_k$,
see Appendix~\ref{ss:RSG}.
Every closed 3-manifold admits infinitely many non-isotopic contact structures,
and an odd-dimensional closed manifold~$M$ admits a contact structure
if and only if its stabilized tangent bundle $TM \oplus \R$ admits a complex structure~\cite{BEM15}.

Theorem~\ref{t:MacSch} has been extended to many contact manifolds:
First, for many closed manifolds~$Q$ every Reeb flow on~$(S^*Q,\xi_{\can})$
has positive topological entropy, \cite{MacSch11}.
Second, there are many closed 3-dimensional manifolds~$M$ such that for every contact structure~$\xi$ 
on~$M$ every Reeb flow has positive topological entropy, 
\cite{A1,A3,A2,ACH,Meiwes}. 
For a recent result for {\it non-degenerate}\/ Reeb flows see~\cite{CoDeRe20}.

While in these results the underlying manifolds have rich loop space topology, 
there are also examples where the positivity of topological entropy of all Reeb flows
does not come from the topological complexity of the loop space.
For instance, it is shown in~\cite{AM19} that the standard smooth sphere of dimension $2n-1 \geq 5$
admits a contact structure for which every Reeb flow has positive topological entropy. 

Nevertheless, for none of these contact manifolds there can be a uniform bound for the normalized 
topological entropy:
The contact volume of the co-oriented contact manifold~$(M,\alpha)$ of dimension~$2n-1$ is defined as
$$
\vol_{\alpha}(M) \,:=\, \frac{1}{n!\,\omega_n} \2 \int_M \alpha \wedge (d\alpha)^{n-1} .
$$
Now define the normalized topological entropy of the Reeb flow $\phi_\alpha^t$ by
\begin{equation} \label{e:han}
\widehat h_{\top}(\alpha) \,:=\,  (\vol_{\alpha}(M))^{1/n} \: h_{\top} (\phi_{\alpha}^1) .
\end{equation}
This normalization extends the normalizations~\eqref{e:hatDH} and~\eqref{def:HTF}
%up to an irrelevant dimension constant 
to all contact manifolds,
see Appendix~\ref{ss:RSG}.
The following result implies Theorem~\ref{t:mainintro}.

\begin{theorem} \label{t:h=c}
Let $(M,\xi)$ be a closed co-orientable contact manifold of dimension at least three.
Then for every real number $c >0$ there exists 
a contact form $\alpha$ for~$\xi$
such that $\widehat h_{\top}(\alpha) = c$.
\end{theorem}

We shall in fact prove the flexibility expressed in Theorem~\ref{t:h=c} for a larger 
growth rate:
Given a $C^1$-diffeomorphism $\phi$ of a compact manifold~$M$,
we define the two real numbers
\begin{eqnarray*}
\Gamma_+(\phi) &:=& \lim_{n \to +\infty} \frac 1n \log \| d \phi^n \|_{\infty} \,, \\
\Gamma(\phi)   &:=& \max \left\{ \Gamma_+(\phi), \Gamma_+(\phi^{-1}) \right\}.
\end{eqnarray*}
Here $\| \cdot \|_{\infty}$ denotes the supremum norm induced by a Riemannian metric 
on~$M$, but the above limit, whose existence follows from the subadditivity of the sequence $\log \|d\phi^n\|_{\infty}$, is clearly independent of the choice of the metric. 

The quantity $\Gamma_+$ was used by Yomdin~\cite{Yom87}
to measure the difference between topological entropy and volume growth, 
and the study of the growth type of the sequence $\| d \phi^n \|_\infty$
for various classes of diffeomorphisms was proposed in~\cite[\S 7.10]{DG90}.
The more symmetric invariant~$\Gamma$ and its polynomial version were investigated, for instance, 
in~\cite{Po02, PoSo04}.
For Hamiltonian flows and Reeb flows, where uniform measurements (like the Hofer metric) 
turned out to capture symplectic rigidity, 
it is particularly natural to look at these two growth rates.

The norm growths $\Gamma_+$ and $\Gamma$ are related to the topological entropy by
\begin{equation} \label{e:hRR}
h_{\top} (\phi) \,\leq\, (\dim M) \: \Gamma_+(\phi) \,\leq\, (\dim M) \: \Gamma (\phi) ,
\end{equation}
see \cite[Corollary 3.2.10]{HK95} for the first inequality.
The numbers $\Gamma_+ (\phi)$ and $\Gamma(\phi)$ are upper bounds for several other invariants of~$\phi$,
and hence the collapsibility of~$\Gamma$ for Reeb flows also implies the collapsibility of these other invariants.
For instance, $\Gamma_+ (\phi)$ is not less than the largest Lyapunov exponent 
$\chi_{\max} (p)$ at every point $p \in M$.
With 
$$
\Sigma (p) = \sum_{\chi_i^+ (p)>0} k_i^+(p) \, \chi_i^+ (p)
$$
the sum of the positive Lyapunov exponents $\chi_i^+(p)$ at~$p$ counted with their multiplicities~$k_i^+(p)$, we then also have 
$\Sigma (p) \leq (\dim M) \, \Gamma_+ (\phi)$.
Together with the Margulis--Ruelle inequality  
(see \cite[Theorem S.2.13]{HK95}) 
we obtain that the metric entropy~$h_{\mu}(\phi)$ 
with respect to any invariant Borel probability measure~$\mu$ has the upper bound
$$
h_\mu (\phi) \,\leq\, \int_M \Sigma (p)\, d\mu (p) \,\leq\, (\dim M) \: \Gamma_+ (\phi).
$$
Applying the variational principle for the topological entropy, we obtain again~\eqref{e:hRR}.

From now on we focus on~$\Gamma$.
For a flow $\phi^t$ we set $\Gamma(\phi) = \Gamma(\phi^1)$, and for a Reeb flow $\phi_\alpha^t$ we set
$\Gamma(\alpha) = \Gamma (\phi_\alpha)$.
For $c>0$ we have $\phi_{c \1 \alpha}^t = \phi_\alpha^{t/c}$ and hence $\Gamma(c \1 \alpha) = \frac 1c \Gamma(\alpha)$.
Like for the topological entropy, the invariant
$$
\widehat \Gamma (\alpha) \,=\, \vol_{\alpha} (M)^{1/n} \: \Gamma(\alpha) ,
$$
where $\dim M = 2n-1$,
is therefore invariant under scaling.
In view of~\eqref{e:hRR}, the following result improves Theorem~\ref{t:mainintro}.

\begin{theorem}\label{t:Gamma=c}
Let $(M,\xi)$ be a closed co-orientable contact manifold of dimension at least three. 
Then for every real number $c>0$ there exists 
a contact form $\alpha$ for~$\xi$ such that $\widehat \Gamma(\alpha) = c$.
\end{theorem}

We shall prove Theorems~\ref{t:h=c} and \ref{t:Gamma=c}
along the following lines.
The main step is to show that 
for every $\gve >0$ there exists a contact form $\alpha_{\gve}$ for~$\xi$ 
such that $\widehat \Gamma(\alpha_{\gve}) \leq \gve$.
We do this with the help of an open book decomposition of~$M$ and an inductive construction,
in which the induction step $\dim 2n-1 \leadsto \dim 2n+1$ is carried out by 
applying the induction hypothesis to the binding of the open book decomposition of~$M$.
We can start the induction in dimension~$1$ at the circle, 
for which $\widehat \Gamma (d\theta) =0$. 
We nevertheless present the $3$-dimensional case separately in \S \ref{s:collapse3}
because we believe that after understanding the geometric ideas in this particular situation 
it is easier 
%for the reader 
to follow the general argument.
The induction step is done in \S \ref{s:collapsegeq3}.
It uses results of Giroux on the correspondence between contact structures
and supporting open books, that we recollect in \S \ref{s:Giroux}.

Given contact forms $\alpha_{\gve}$ as above,
Theorems~\ref{t:h=c} and~\ref{t:Gamma=c}
follow from~\eqref{e:hRR} and from a simple modification of~$\alpha_{\gve}$ 
that increases $\widehat h_{\top}$ and $\widehat \Gamma$, see~\S \ref{s:spectrum}.

\subsection{Collapsing the growth rate of symplectic invariants} \label{ss:collapseSH}
In the works \cite{A2, AM19, Meiwes} it is shown that the exponential growth rate of certain symplectic topological invariants provides a lower bound for the topological entropy of Reeb flows. 
These invariants are linearised Legendrian contact homology~\cite{A2}, 
wrapped Floer homology~\cite{AM19}, and Rabinowitz--Floer homology~\cite{Meiwes}. 
Combining these results with Theorem~\ref{t:mainintro} we obtain that 
the growth rate of these invariants can be made arbitrarily small.
Details are given in \S \ref{s:collapseSH}.

\subsection{Relations to systolic inequalities} \label{ss:systolic}

Consider a closed co-orientable contact manifold $(M,\xi)$ of dimension~$2n-1$.
Given a contact form~$\alpha$ for~$\xi$ that has at least one periodic Reeb orbit, 
take the smallest period $T_{\min}(\alpha)$. The so-called systolic ratio 
$$
\rho_{\mathrm{sys}} (\alpha) \,=\, \vol_{\ga} (M)^{- 1/n} \: T_{\min} (\alpha)
$$ 
is then invariant under scalings of $\alpha$.

While for spherizations $S^*Q$ of many closed manifolds~$Q$
there are famous uniform upper bounds on the systolic ratios of Riemannian Reeb flows, 
in the full class of Reeb flows one has the following flexibility result.

\begin{theorem} \label{t:sys}
For any closed co-orientable contact manifold $(M,\xi)$ and every positive number~$c$ there exists a contact form $\alpha$ such that $\rho_{\mathrm{sys}} (\alpha) > c$.
\end{theorem}

This result was shown for the tight 3-sphere in~\cite{ABHS}
and for all contact 3-manifolds in~\cite{ABHS1}
by a plug construction in open book decompositions.
The idea in this paper to use open book decompositions for proving entropy collapse
came from these works. 
Theorem~\ref{t:sys} in dimension $\geq 5$ was proved in~\cite{Sag18}.
That proof was later on much simplified~\cite{Sag20} by using our inductive construction 
in~\S \ref{s:collapsegeq3}.
Interestingly, our construction in dimension~3 does not yield a proof of  
Theorem~\ref{t:sys}. 
This suggests that at least in the smallest interesting dimension,
one has more flexibility to collapse
the topological entropy of Reeb flows than to increase their systolic ratio.

%%%%%%%%%%%%%%%%%%%%%%%%%%%%%%%%%%%%%%%%%%%%%%%%%%%%%%%%%%%%

\subsection{Minimal entropy problems for Finsler and Reeb flows} 
\label{ss:entprob}

Given a class $\cc$ of maps on a compact manifold~$M$,
it is interesting to understand which maps in class~$\cc$
minimize the (normalized) topological entropy.
Since topological entropy is a measure for the complexity, these maps
can then be considered as the simplest, or the best, maps on~$M$ in class~$\cc$.

For the class of Riemannian geodesic flows on the spherization $SQ$ of 
a compact manifold~$Q$,
the minimal entropy problem consists of three parts.

\smallskip
\begin{itemize}
\item[(P1)] 
Compute the minimal entropy
\begin{equation*} %\label{e:defhtopg}
h_{\top}(Q, \cg) \,:=\, 
   \inf \left\{ \widehat h_{\top} ( \phi_g) \mid \text{$g$ a Riemannian metric on $Q$} \right\} .
\end{equation*}

\item[(P2)] 
Decide whether the infimum is attained or not.

\smallskip
\item[(P3)] 
If the infimum is attained, describe the minimizers~$g$.
\end{itemize}

\smallskip
For manifolds admitting a locally symmetric Riemannian metric of negative curvature, 
Theorems~\ref{t:katok} and~\ref{t:BCG} completely solve the minimal entropy problem.
Among the many further interesting works on the minimal entropy problem are~\cite{KKW, PP}.

The minimal entropy problem can also be formulated  
for the larger classes of Finsler and Reeb flows.
Define three more numbers 
$$
h_{\top}^{\HT} (Q,\cR) \,\leq\,  h_{\top}^{\HT} (Q,\cf) \,\leq\, h_{\top}^{\HT} (Q, \cf_{\rev}) 
\,\leq\, h_{\top} (Q,\cg) 
$$
by taking the infimum in the definition of $h_{\top} (Q,\cg)$ 
over all contact forms on $(S^*Q,\xi_{\can})$  for $h_{\top}^{\HT} (Q,\cR)$,
over all Finsler metrics for $h_{\top}^{\HT} (Q,\cf)$,
and over all reversible Finsler metrics for $h_{\top}^{\HT} (Q,\cf_{\rev})$, 
respectively,
where as in~\eqref{def:HTF} and ~\eqref{e:han} we normalize by the Holmes--Thompson volume.

Theorem~\ref{t:mainintro} shows that $h_{\top}^{\HT} (Q,\cR) =0$ 
for all compact manifolds~$Q$.
This settles~(P1) for the class~$\cR$.
Furthermore, 
for many manifolds, like those with fundamental group of exponential growth, the answer
to~(P2) is `no' by the general version of Theorem~\ref{t:MacSch} from~\cite{MacSch11}. 

We now turn to the invariants $h_{\top}^{\HT} (Q,\cf)$ and $h_{\top}^{\HT} (Q,\cf_{\rev})$.
By Theorems~\ref{t:Finslerintro} and \ref{t:manning} we have
$$
\begin{array}{rclcl}
  c_n\, h_{\vol} (Q) &\leq& h_{\top}^{\HT} (Q,\cf) &\leq&  h_{\top} (Q, \cg) ,   \\ [0.8em]
2 c_n\, h_{\vol} (Q) &\leq& h_{\top}^{\HT} (Q,\cf_{\rev}) &\leq&  h_{\top} (Q, \cg) .
\end{array}
$$
For manifolds admitting a locally symmetric Riemannian metric of negative curvature
(for which $h_{\top} (Q, \cg) = h_{\vol}(Q)$)
nothing more seems to be known about the values of $h_{\top}^{\HT} (Q,\cf)$ and
$h_{\top}^{\HT} (Q,\cf_{\rev})$,
so already (P1) in the entropy problem is wide open for the classes~$\cf$ and~$\cf_{\rev}$. 

Addressing (P3) we note
that in the Finsler setting one cannot expect metrics of minimal normalized topological entropy to be unique, or even to be characterised in terms of curvature-like invariants. Indeed, any exact symplectomorphism of~$T^*Q$ that is $C^2$-close to the identity maps the unit cotangent sphere bundle~$S^*(F)$ of the Finsler metric~$F$ 
to the unit cotangent sphere bundle~$S^*(F')$ of some Finsler metric~$F'$ whose geodesic flow is conjugated to the one of~$F$ by a smooth time-preserving conjugacy. 
In particular, the new Finsler metric~$F'$ has the same normalized topological entropy 
as~$F$, but need not be isometric to it. 
See Appendix~\ref{app:Finslerisotop} for a discussion of this.

\medskip \noindent
{\bf Higher rank.}
More can be said in higher rank. 
The following result is proved in~\S \ref{ss:ver} using Verovic's work~\cite{Ver99}.

\begin{prop} \label{prop:ver}
Let $(Q,g)$ be a compact locally symmetric space of non-compact type and
of rank~$\geq 2$.
Then there exists a constant $c<1$ such that
\begin{equation} \label{e:Qcf}
h_{\top}^{\HT} (Q,\cf_{\rev}) \,\leq\, c\, \widehat h_{\vol} (g) . 
\end{equation} 
\end{prop}

The constant $c$ only depends on the globally symmetric space $(\widetilde Q, \widetilde g)$,
and it can be computed from its Weyl data. 
See Proposition~\ref{p:ver} below for a more precise statement.

Let $h_{\vol}^{\sym}(Q)$ be the minimum of $\widehat h_{\vol}(g)$ taken over all 
locally symmetric Riemannian metrics~$g$ on~$Q$.
This number is easy to compute, see \cite[\S 2]{CF03}.
Unfortunately it is still not known whether Theorem~\ref{t:BCG} also holds in higher rank, 
that is, whether $h_{\vol} (Q) = h_{\vol}^{\sym}(Q)$. 
However, this is known if $(Q,g)$
is locally isometric to a product of negatively curved symmetric spaces 
of dimension~$\geq 3$, \cite{CF03},
and for quotients of the $k$-fold product 
$(\HH^2)^k = \HH_2 \times \dots \times \HH_2$
of the real hyperbolic plane, \cite{Mer16}.
For these spaces, \eqref{e:Qcf} can thus be written as
$$
h_{\top}^{\HT} (Q,\cf_{\rev}) \,\leq\, c\, h_{\vol}(Q) . 
$$
We shall compute the constant $c$ for quotients of $(\HH^2)^k$ in~\S \ref{ss:ver}.
For instance, $c (\HH^2 \times \HH^2) = \sqrt[4]{2} \approx 0.841$.
This should be compared with the constant $2 c_4 = \sqrt{\frac{2}{\pi}} \approx 0.61$
for the lower bound in Theorem~\ref{t:mainintro}.

\begin{comment}

\medskip \noindent
{\bf Slow entropy.}
For base manifolds~$Q$ with poor loop space topology, like those with fundamental 
group of polynomial growth, an interesting minimal entropy problem can be formulated for all the 
four classes $\cg$, $\cf_{\rev}$, $\cf$, $\cR$ by looking at the slow topological entropy.
It is defined by replacing in the definition~\eqref{def:htop} of topological entropy
the term $\frac 1T \log$  by $\frac{1}{\log T} \log$.
The slow entropy does not change if one multiplies the contact form 
by a non-zero constant, whence we do not normalize it.
By Theorem~\ref{t:Finslerintro}, the resulting slow minimal entropies agree for the three
classes $\cg$, $\cf_{\rev}$, $\cf$, 
$$
{\mbox{slow\2-\,}}h_{\top} (Q,\cg) \,=\, 
{\mbox{slow\2-\,}}h_{\top} (Q,\cf_{\rev}) \,=\, 
{\mbox{slow\2-\,}}h_{\top} (Q,\cf) .
$$
We refer to~\cite{FLS13} for related results and to \S 7.2 therein for the formulation of several 
open problems on slow entropies of Reeb flows and Finsler geodesic flows.

\end{comment}

\medskip
The minimal entropy problem can also be studied for the volume entropies $h_{\vol}$
instead of $h_{\top}$, and by normalizing either entropy by the Busemann--Hausdorff volume. 
Much of the above discussion applies also to these minimal entropies.

%%%%%%%%%%%%%%%%%%%%%%%%%%%%%%%%%%%%%%%%%%%%%%%%%%%%%%%%%%%%%

\subsection{Topological pressure}

In view of Theorem~\ref{t:mainintro}, there is no minimal entropy program for Reeb flows.
Furthermore, the situation cannot be salvaged by looking at subexponential growth rates, 
since replacing $\lim_{n \to \infty} \frac 1n \log \dots$ in the definition of topological entropy
by $\lim_{n \to \infty} \frac{1}{n^c} \log \dots$ for some $c \in (0,1)$ yields 
$+ \infty$ for all Reeb flows on many contact
manifolds by Theorem~\ref{t:MacSch}.

However, increasing topological entropy in terms of topological pressure 
leads to a meaningful problem.
Given a closed contact manifold $(M,\xi)$ associate with 
every continuous function $f \in C^0(M,\R)$ and every contact form $\alpha$ 
for~$\xi$ the topological pressure $P(\alpha, f) = P (\phi_{\alpha},f)$,
see \cite[Chapter~9]{Wal82} for the definition and basic results on topological 
pressure. 
We recall that $P(\alpha,0) = h_{\top}(\alpha)$ and that the variational principle for topological pressure says
\begin{equation} \label{e:pressure}
P(\alpha,f) \,=\, 
\sup_{\mu \in \cm (\alpha)} 
\left\{ h_{\mu}(\phi_{\alpha}) + \int_M f \,d\mu \right\}
\end{equation}
where $\cm (\alpha)$ denotes the set of $\phi_{\alpha}^t$-invariant Borel probability 
measures on~$M$ 
and $h_\mu \geq 0$ is the entropy of the measure~$\mu$.
Define
$$
P (M,f) \,:=\, \inf \left\{ P(\alpha, f) \mid \mbox{$\alpha$ a normalized contact form on $(M,\xi)$} \right\} .
$$
Since $P(\alpha, f+c) = P(\alpha, f) +c$ for all $c \in \R$, we can assume that $\min f = 0$.
Together with Theorem~\ref{t:mainintro} we then obtain
$$
0 \,\leq\, P(M,f) \,\leq\, \max f .
$$
It would be interesting to see if these bounds can be sharpened 
for functions~$f$ that do not identically vanish.
Our proof of Theorem~\ref{t:mainintro} does not help with this problem, 
since the maximal measures in~\eqref{e:pressure} (the so-called equilibrium states)
may not be related in any way to the open book decomposition in our proof.

%\medskip \noindent
%{\bf Remark about smoothness.}
%The volume entropy $h_{\vol}(F)$ is a purely metrical quantity, that can be defined
%for Finsler metrics~$F$ of class~$C^0$.
%To have a Finsler flow, we need to assume that $F$ is of class $C^{1,1}$,
%-- mmh Achtung Legendre trafo, symmetrisch Bedingung = C^2 und strict positiv gekrümmte Sphären --
%and we prove Theorem~\ref{t:Finslerintro} and Manning's inequality for Finsler flows of such~$F$.
%For Manning's inequality for Reeb flows and for inequality~\eqref{e:hsigmaG} 
%we need to assume that $\alpha$ is $C^\infty$-smooth,
%since in the proofs (in Appendix \ref{a:Manning} and \S \ref{s:Floer})
%we appeal to Yomdin's theorem.
%The collapsing results are all proven for $C^\infty$ data.

\bigskip
\noindent
{\bf Acknowledgment.}
We wish to thank Viktor Ginzburg for several useful discussions, 
that in particular led to the results in Section~\ref{s:spectrum}, 
Leonid Polterovich who asked us whether the collapse of
topological entropy can be extended to the collapse of norm growth,
and Patrick Verovic for explaining to us his work~\cite{Ver99}.
We also thank the referee for many useful comments and the careful reading.

A.A., M.A.\ and M.S.\ are partially supported by the SFB/TRR 191 
`Symplectic Structures in Geometry, Algebra and Dynamics', 
funded by the DFG (Projektnummer 281071066 – TRR 191).
M.A. is also supported by the ERC consolidator grant 646649 `SymplecticEinstein' 
and by the Senior Postdoctoral fellowship of the Research Foundation - Flanders (FWO) 
in fundamental research 1286921N.
F.S.\ is partially supported by the SNF grant 200020-144432/1.

\section{Volume entropy for Finsler geodesic flows} \label{s:ent}

\subsection{Finsler metrics and their volumes}  \label{ss:vol}
By a Finsler metric on an $n$-dimensional manifold~$Q$ we mean in this paper a continuous 
function $F \colon TQ \rightarrow [0,+\infty)$ which is fiberwise convex, 
fiberwise positively homogeneous of degree~1, and positive outside of the zero section. 
The Finsler metric~$F$ is said to be reversible if $F(v)=F(-v)$ for all $v \in TQ$. 

For $q \in Q$ the unit disk in~$T_q Q$ determined by the Finsler metric~$F$ is the set
\[
D_q(F) := \left\{ v \in T_qQ \mid F(v) \leq 1 \right\} .
\]
This is a convex compact neighborhood of the origin in $T_q Q$. 
The function $F|_{T_q Q}$ is precisely the Minkowski gauge of~$D_q(F)$. 
The unit co-disk in $T_q^* Q$ is the polar set of~$D_q(F)$:
\[
D_q^*(F) := \left\{ p \in T_q^* Q \mid \langle p,v \rangle \leq 1 \; \forall \2 v \in D_q(F) \right\},
\]
where $\langle\cdot,\cdot \rangle$ denotes the duality pairing between tangent vectors and co-vectors. This is a compact convex neighborhood of the origin in~$T_q^* Q$. 

On compact Finsler manifolds there are two different notions of volume that are used in the literature. From the point of view of this paper, the most natural one is the Holmes--Thompson volume, which can be defined as
\[
\vol_F^{\HT}(Q) := 
\frac{1}{n! \, \omega_n} \, \int_{D^*(F)} \omega^n ,
\]
where 
\[
D^*(F) := \bigcup_{q \in Q} D_q^*(F) \,\subset\, T^* Q
\]
is the unit co-disk bundle of $Q$,
where $\omega^n$ denotes the standard volume form on~$T^*Q$ 
induced by integrating the $n$-fold exterior power of  
the canonical symplectic form~$\omega = \sum_j dq_j \wedge dp_j$, 
and where $\omega_n$ is the volume of the Euclidean unit ball in~$\R^n$, 
$n=\dim Q$. 
The normalization factor $n! \,\omega_n$  makes $\vol_F^{\HT}(Q)$ coincide with 
the Riemannian volume of~$Q$ when $F(v) = \sqrt{g(v,v)}$ is a Riemannian metric on~$Q$. 

Alternatively, the Holmes--Thompson volume can be defined as the integral over~$Q$ of a suitable volume density~$\rho_F^*$. Here by volume density we mean a norm on the line bundle~$\Lambda^n(TQ)$, whose fiber at $q \in Q$ is the  top degree component of the exterior algebra of $T_q Q$, that is, the 1-dimensional space spanned by $v_1\wedge \dots \wedge v_n$, where $v_1,\dots,v_n$ is a basis of $T_q Q$. 
When $Q$ is orientable, a volume density is just the absolute value of a nowhere vanishing differential $n$-form. A volume density can be integrated over any non-empty open subset of $Q$, producing a positive number.
The volume density $\rho_F^*$ is defined as follows: Given any volume density $\rho$ on~$Q$ set
\[
\rho_F^*(q) := \frac{|D_q^*(F)|_{\rho}^*}{\omega_n} \,\rho(q),
\]
where $|\! \cdot \!|^*_{\rho}$ denotes the Lebesgue measure on $T_q^* Q$ 
that is normalized to~1 on the $n$-dimensional parallelogram spanned 
by the covectors that are dual to basis vectors $v_1,\dots,v_n$ in~$T_q Q$ 
such that $\rho(q)[v_1 \wedge \dots\wedge v_n]=1$. 
We then have
\[
\vol_F^{\HT}(Q) = \int_Q \rho_F^*.
\]
Another common choice is to consider the Busemann--Hausdorff volume, which is defined as
\[
\vol_F^{\BH}(Q) := \int_Q \rho_F,
\]
where the volume density $\rho_F$ is given by
\[
\rho_F(q) := \frac{\omega_n}{|D_q(F)|_{\rho}} \, \rho(q).
\]
Here $\rho$ is again an arbitrary volume density on $Q$ and $| \! \cdot \! |_{\rho}$ is 
the Lebesgue measure on~$T_q Q$ normalized to~1 on the parallelogram spanned by vectors 
$v_1,\dots,v_n$ in~$T_q Q$ with $\rho(q)[v_1 \wedge \dots \wedge v_n]=1$. 
When $F$ is reversible, the Busemann--Hausdorff volume of~$Q$ coincides with the $n$-dimensional Hausdorff measure of~$Q$ with respect to the distance induced by~$F$.

Both volumes reduce to the standard Riemannian volume when the Finsler metric $F$ is Riemannian. If $F$ is reversible, then 
\begin{equation}
\label{reversible}
\vol_F^{\HT}(Q) \leq \vol_F^{\BH}(Q), 
\end{equation}
with equality holding if and only if $F$ is Riemannian. 
This follows from the Blaschke--Santal\'o inequality, see e.g.\ \cite{Dur98}. 
In the non-reversible case, the Holmes--Thompson volume can be much larger than the Busemann--Hausdorff volume. 
Note that both the Holmes--Thompson and 
the Busemann--Hausdorff volume depend monotonically on the Finsler metric, meaning that
\begin{equation}
\label{monotonicity}
F_1 \leq F_2 \qquad \Rightarrow \qquad \vol_{F_1}^{\HT}(Q) \leq \vol_{F_2}^{\HT}(Q), \quad  \vol_{F_1}^{\BH}(Q) \leq \vol_{F_2}^{\BH}(Q),
\end{equation}
and rescale as
\begin{equation}
\label{rescaling}
\vol_{cF}^{\HT}(Q) = c^n \vol_F^{\HT}(Q), \qquad \vol_{cF}^{\BH}(Q) = c^n \vol_F^{\BH}(Q),
\end{equation}
when the Finsler metric $F$ is multiplied by a positive constant $c$.

\subsection{Volume entropy} \label{ss:volent}
Let $F$ be a Finsler metric on a compact $n$-dimensional manifold~$Q$. 
This Finsler metric lifts to a Finsler metric on the universal cover~$\widetilde{Q}$ of~$Q$, 
and we denote the lifted metric by the same symbol~$F$.  The $R$-ball centered at $q \in \widetilde{Q}$  
that is induced by~$F$ is the following compact subset of~$\widetilde{Q}$:
\begin{equation} \label{e:Lip}
B_q(F,R) := \bigl\{ \gamma(R) \mid \gamma \colon [0,R] \rightarrow \widetilde{Q} 
        \mbox{ Lipschitz curve, } \gamma(0)=q \mbox{ and } F \circ \dot\gamma \leq 1 \mbox{ a.e.} \bigr\}.
\end{equation}
When $F$ is reversible, $B_q(F,R)$ is the ball of the distance on~$\widetilde{Q}$ 
that is induced by~$F$; in general, it is the forward ball of an asymmetric distance. 

The volume entropy of $F$ is the non-negative number
\begin{equation} \label{def:hvol}
h_{\vol}(F) := \lim_{R\rightarrow \infty} \frac{1}{R} \log \Vol (B_q(F,R)).
\end{equation}
Here $\Vol$ denotes the volume of Borel subsets of~$\widetilde{Q}$ with respect to the lift to~$\widetilde{Q}$ of an arbitrary Riemannian metric on~$Q$. 
A minor modification of Manning's argument from~\cite{Man79} shows that 
the above limit exists and is independent of the choice of the point~$q \in \widetilde{Q}$ and of the Riemannian metric on~$Q$,
see Proposition~\ref{p:manning}. In the case of the Finsler metric $G=\sqrt{g(\cdot,\cdot)}$ that is induced by a Riemannian metric~$g$, we use interchangeably the notations
\[
h_{\vol}(g) = h_{\vol}(G).
\]
The volume entropy is monotonically decreasing in $F$, meaning that
\begin{equation}
\label{monotonicity-h}
F_1 \leq F_2 \qquad \Rightarrow \qquad h_{\vol}(F_1) \geq h_{\vol}(F_2).
\end{equation}
Indeed, if $F_1 \leq F_2$ on $Q$ then the same inequality holds on $\widetilde{Q}$ and hence \eqref{e:Lip} implies
\[
B_q(F_1,R) \supset B_q(F_2,R) \qquad \forall \, q \in \widetilde{Q}, \; \forall \2 R \geq 0,
\]
from which \eqref{monotonicity-h} follows. Let $c$ be a positive number. 
From the identity
\[
B_q(cF,R) = B_q(F,c^{-1} R)
\]
we deduce that the volume entropy rescales as
\begin{equation}
\label{rescaling-h}
h_{\vol}(cF) = c^{-1} \,h_{\vol}(F).
\end{equation}
Together with \eqref{rescaling}, this suggests to consider the normalized volume entropies
\[
\widehat{h}_{\vol}^{\HT}(F) := \vol_F^{\HT}(Q)^{1/n}\, h_{\vol}(F), \qquad \widehat{h}_{\vol}^{\BH}(F) := \vol_F^{\BH}(Q)^{1/n}\, h_{\vol}(F).
\]
These quantities are now invariant under scaling: 
\[
\widehat{h}_{\vol}^{\HT}(cF) = \widehat{h}_{\vol}^{\HT}(F), \qquad 
\widehat{h}_{\vol}^{\BH}(cF) = \widehat{h}_{\vol}^{\BH}(F).
\] 
Since the Holmes--Thompson and the Busemann--Hausdorff volumes coincide when $F=G=\sqrt{g}$ is Riemannian, there is just one normalized volume entropy in the Riemannian case, and we denote it by
\[
\widehat{h}_{\vol}(g) = \widehat{h}_{\vol}(G).
\]
In the next two subsections, we study how the two different normalized volume entropies 
of an arbitrary Finsler metric can be bounded from below and from above 
in terms of the normalized volume entropy of suitable Riemannian metrics. 
Our arguments follow~\cite{apbt16}, 
where similar techniques are used in order to derive bounds for the systolic ratio. 

\subsection{From reversible Finsler to Riemannian} 
Let $F$ be a reversible Finsler metric on the compact $n$-dimensional manifold~$Q$. 
Denote by~$E_q$ the inner Loewner ellipsoid of the symmetric convex body~$D_q(F)$, 
i.e.\ the ellipsoid centered at the origin which is contained in~$D_q(F)$ 
and has maximal volume among all ellipsoids with this property. 
Here by volume we mean any translation invariant measure on~$T_q Q$ 
(which is unique up to multiplication by a positive constant). 
It is well known that the inner Loewner ellipsoid is unique, and John proved that it satisfies
\begin{equation}
\label{john}
E_q \subset D_q(F) \subset \sqrt{n}\, E_q.
\end{equation}
See \cite{Joh48}, or \cite{Bal97} for a modern proof of these results.
Denote by $G \colon TQ \rightarrow [0,+\infty)$ the function which in each tangent space~$T_q Q$ 
is the Minkowski gauge of~$E_q$. The function~$G$ is the square root of a 
Riemannian metric: $G(v) = \sqrt{g(v,v)}$ for some continuous
Riemannian metric~$g$ on~$Q$. Indeed, the continuity of~$G$ easily follows from 
the uniqueness of the inner Loewner ellipsoid. From the inclusions~\eqref{john} we deduce 
the inequalities
\begin{equation}
\label{relmetri}
n^{-1/2}\,  G \leq F \leq G,
\end{equation}
which thanks to \eqref{monotonicity-h} and \eqref{rescaling-h} imply the bounds
\begin{equation}
\label{relhvol}
h_{\vol}(G) \leq h_{\vol}(F) \leq \sqrt{n}\, h_{\vol}(G).
\end{equation}
By \eqref{monotonicity} and the second inequality in \eqref{relmetri} the Busemann--Hausdorff volume of~$(Q,F)$ 
has the upper bound
\begin{equation} \label{UB}
\vol_F^{\BH} (Q) \leq \vol_G (Q).
\end{equation} 
In order to get a lower bound for the Holmes--Thompson volume of $(Q,F)$ we can use \eqref{monotonicity}, \eqref{rescaling} and the first inequality in \eqref{relmetri} and obtain
\begin{equation}
\label{cattivo}
\vol_G (Q) \leq n^{n/2} \vol_F^{\HT} (Q).
\end{equation}
However, we get a better bound by the following argument. 
The polar set $E_q^*=D_q^*(G)$ of $E_q=D_q(G)$ satisfies
\[
D_q^*(F) \subset D_q^*(G)
\]
and is the outer Loewner ellipsoid of $D_q^*(F)$, i.e.\ the centrally symmetric ellipsoid of minimal volume among those containing $D_q^*(F)$. Then we have
\[
|D_q^* (G)|_{\rho}^* \,\leq\, \frac{n! \, \omega_n}{2^n} \, |D_q^*(F)|_{\rho}^* .
\]
This follows from the fact that the ratio between the volume of the outer 
Loewner ellipsoid of a symmetric convex body~$K$ and the volume of~$K$ is maximal 
for the cross-polytope, 
a result that Ball deduced from the inverse Brascamp--Lieb inequality of Barthe in~\cite[Theorem 5]{bal01}. 
Therefore, we obtain
\begin{equation}
\label{buono}
\vol_G(Q) \,\leq\, \frac{n! \,\omega_n}{2^n} \vol_F^{\HT}(Q),
\end{equation}
which is a  better bound than \eqref{cattivo} for every $n \geq 2$, and also asymptotically because
\[
%\log \frac{n!\, \omega_n}{2^n} = \frac{n}{2} \log n - \frac{n}{2} \left( 1 - \log \frac{\pi}{2} \right) + O(1)
\lim_{n \to \infty} \left( \frac{n!\, \omega_n}{2^n \, n^{n/2}} \right)^{1/n} \,=\, \sqrt{\frac{\pi}{2e}} 
\]
by the Stirling formula. 
By putting together 
\eqref{reversible}, \eqref{relhvol}, \eqref{UB} and~\eqref{buono} 
we obtain the following result.

\begin{prop}
\label{bound1}
Let $F$ be a reversible Finsler metric on the compact $n$-dimensional manifold~$Q$ 
and let $G = \sqrt{g}$ be the Riemannian metric on~$Q$ whose unit disks are 
the inner Loewner ellipsoids of the unit disks of~$F$. Then
$$
\frac{2}{(n!\, \omega_n)^{1/n}} \, \widehat{h}_{\vol}(g) \,\leq\, 
\widehat{h}_{\vol}^{\HT}(F) \,\leq\, \widehat{h}_{\vol}^{\BH}(F) \,\leq\, 
                 \sqrt{n} \; \widehat{h}_{\vol}(g).
$$
\end{prop}

\subsection{From irreversible to reversible Finsler} 
Let $F$ be an arbitrary Finsler metric on the compact $n$-dimensional manifold~$Q$. We symmetrize the metric~$F$ 
by the following procedure: We define $S \colon TQ \rightarrow [0,+\infty)$ to be the reversible Finsler metric on~$Q$ whose unit co-disk at each $q \in Q$ is the reflection body of~$D_q^*(F)$, 
i.e.\ the centrally symmetric convex body
\[
D^*_q(S) := \mathrm{conv} \bigl( D_q^*(F) \cup ( - D_q^*(F)) \bigr).
\]
Note that 
\begin{equation} \label{symmetrization}
D_q^*(F) \subset D^*_q(S) \subset \theta\,  D_q^*(F),
\end{equation}
where $\theta$ is the irreversibility ratio of $F$, i.e.\ the number
\begin{equation} \label{e:irr}
\theta := \max_{\substack{v\in TQ \\ F(v)=1}} F(-v),
\end{equation}
which is at least 1, and equal to 1 if and only if $F$ is reversible. 
Indeed, the second inclusion in \eqref{symmetrization} follows from the fact that $\theta$ is an upper bound for the norm of minus the identity on~$T_q Q$ with the asymmetric norm~$F$, and hence also for the norm of minus the identity 
on~$T_q^* Q$ with the asymmetric norm that is dual to~$F$. Moreover, the volume of $D^*_q(S)$ has the upper bound
\begin{equation} \label{Rogers-Shephard}
|D^*_q(S)|_{\rho}^* \,\leq\, 2^n \, |D_q^*(F)|_{\rho}^*,
\end{equation}
as proven by Rogers and Shephard in \cite[Theorem 3]{rs58}. From \eqref{symmetrization} we deduce
\begin{equation}
\label{N3}
\theta^{-1} S \leq F \leq S,
\end{equation}
and hence \eqref{monotonicity-h} and \eqref{rescaling-h} imply
\begin{equation}
\label{N4}
h_{\vol}(S) \leq h_{\vol} (F) \leq \theta\, h_{\vol}(S).
\end{equation}

On the other hand, from \eqref{Rogers-Shephard} and the second inequality in \eqref{N3} we obtain the following inequalities for the Holmes--Thompson volume
\begin{equation}
\label{N5}
2^{-n} \, \vol_S^{\HT} (Q) \leq \vol_F^{\HT} (Q) \leq \vol_S^{\HT} (Q).
\end{equation}
The bounds \eqref{N4} and \eqref{N5} imply the following result.

\begin{prop}
\label{bound2}
Let $F$ be a Finsler metric on the compact $n$-dimensional manifold~$Q$ 
with irreversibility ratio~$\theta$ and let $S$ be the reversible Finsler metric whose dual disks are the reflection bodies of the dual disks of~$F$:
\[
D^*_q(S) = \mathrm{conv} \bigl( D_q^*(F) \cup ( - D_q^*(F)) \bigr) \qquad \forall \, q \in Q.
\]
Then
\[
\frac{1}{2} \, \widehat{h}^{\HT}_{\vol}(S) \,\leq\,  \widehat{h}_{\vol}^{\HT} (F) 
             \,\leq\, \theta \, \widehat{h}_{\vol}^{\HT}(S).
\]
\end{prop}

The lower bounds of Propositions \ref{bound1} and~\ref{bound2}, 
together with Stirling's formula, have the following consequence:

\begin{corollary}
\label{minhvol}
Let $Q$ be a compact $n$-dimensional manifold and
denote by~$h_{\vol}(Q)$ the infimum of~$\widehat{h}_{\vol}(g)$ over all Riemannian metrics~$g$ 
on~$Q$. Then the Holmes--Thompson normalized volume entropy of an arbitrary Finsler metric~$F$ on~$Q$ 
has the lower bound
\[
\widehat{h}_{\vol}^{\HT}(F) \,\geq\, c_n \, h_{\vol}(Q) ,
\]
where
\[
c_n :=  \frac{1}{(n!\, \omega_n)^{1/n}} \sim \sqrt{\frac{e}{2\pi}} \frac{1}{\sqrt n}.
\]
Moreover, if the Finsler metric $F$ is reversible, we have
\[
\widehat{h}_{\vol}^{\BH}(F) \geq  \widehat{h}_{\vol}^{\HT}(F) \geq 2 \2 c_n \, h_{\vol}(Q).
\]
\end{corollary}

\begin{rem}
{\rm 
If we symmetrize $D_q^*(F)$ by considering the difference body $D_q^*(F)-D_q^*(F)$ instead of the reflection body, 
then we get a worse bound, because in this case the factor $2^n$ in~\eqref{Rogers-Shephard} must be replaced by 
the middle binomial coefficient~$\binom{2n}{n}$, in view of the Rogers--Shephard inequality for the volume 
of the difference body, see~\cite{RS57}. 
By using the reflection body instead of the difference body, the systolic upper bounds of 
Theorem~4.13 and Corollary~4.14 in~\cite{apbt16} can be improved by replacing the dimension dependent 
quantity $\sqrt[n]{(2n)!/(n!)^2}$ by the constant factor~2.
}
\end{rem}

If the volume entropy is normalized by the Busemann--Hausdorff volume, we do not get a lower bound that is independent of the irreversibility ratio. 
From~\eqref{monotonicity}, \eqref{rescaling}, \eqref{N3} and~\eqref{N4} we obtain
\[
\frac{1}{\theta}\,  \widehat{h}_{\vol}^{\BH}(S) \,\leq\, \widehat{h}_{\vol}^{\BH}(F)  \,\leq\, \theta\,  \widehat{h}_{\vol}^{\BH}(S).
\]
We do not have a lower bound that is independent of $\theta$ because, unlike the volume ratio 
$|D_q^*(S)|_{\rho}^*/  |D_q^*(F)|_{\rho}^*$, the ratio
$|D_q(F)|_{\rho}/|D_q(S)|_{\rho}$ can be arbitrarily large.

On the other hand, the upper bound can be made independent of the irreversibility ratio~$\theta$ by symmetrizing,
this time, directly in~$TQ$: 
We consider the reversible Finsler metric~$T$ whose unit ball at~$q$ is the set
\[
D_q(T) := \mathrm{conv} \bigl( D_q(F)  \cup (- D_q(F)) \bigr).
\]
For this metric, we have
\[
T \leq F \leq \theta\, T,
\]
from which we obtain
\[
\frac{1}{\theta} \, h_{\vol}(T) \leq h_{\vol}(F)\leq  h_{\vol} (T).
\]
Moreover, the Rogers--Shephard inequality for the reflection body gives
\[
\vol_T^{\BH} (Q) \leq \vol_F^{\BH} (Q) \leq 2^n  \vol_T^{\BH} (Q),
\]
and we deduce the following result.

\begin{prop}
Let $F$ be a Finsler metric on the compact $n$-dimensional manifold~$Q$ 
with irreversibility ratio~$\theta$ and let $T$ be the 
reversible Finsler metric whose unit disk at each~$q$ is the reflection body 
\[
D_q(T) =  \mathrm{conv} \bigl( D_q(F)  \cup (- D_q(F)) \bigr)
\]
of the disk of $F$ at $q$. Then
\[
\frac{1}{\theta} \, \widehat{h}^{\BH}_{\vol}(T) \,\leq\,  \widehat{h}_{\vol}^{\BH} (F) 
             \,\leq\, 2 \, \widehat{h}_{\vol}^{\BH}(T).
\]
\end{prop}

\subsection{Lower bounds on the normalized topological entropy}  \label{ss:lbne}
We now assume that the (possibly irreversible) Finsler metric~$F$ on~$Q$ has better regularity and convexity properties: 
Outside of the zero section, $F \colon TQ \rightarrow \R$ is of class~$C^{2}$ 
and the fiberwise second differential~of $F^2$ is positive definite. 
We will refer to such an~$F$ as to a regular Finsler metric. Under these assumptions, 
the geodesic flow of~$F$ is well defined. We denote by $h_{\top}(F)$ the topological entropy of this flow, and by
\[
\widehat{h}_{\top}^{\HT}(F) = \vol_F^{\HT}(Q)^{1/n} \, h_{\top}(F), \qquad 
\widehat{h}_{\top}^{\BH}(F) = \vol_F^{\BH}(Q)^{1/n} \, h_{\top}(F)
\]
the Holmes--Thompson and Busemann--Hausdorff normalizations of this entropy. 
Manning's inequality
\[
h_{\top}(F) \geq h_{\vol}(F)
\]
from \cite{Man79} holds also in the Finsler setting, as shown in Theorem~\ref{t:manning}. 
Then Corollary~\ref{minhvol} has the following immediate consequence. 

\begin{corollary}
\label{minhtop}
Let $Q$ be a compact $n$-dimensional manifold and
denote by~$h_{\vol}(Q)$ the infimum of~$\widehat{h}_{\vol}(g)$ over all Riemannian metrics~$g$ on~$Q$. 
Then the Holmes--Thompson normalized topological entropy of any regular Finsler metric~$F$ on~$Q$ has the lower bound
\[
\widehat{h}_{\top}^{\HT}(F) \,\geq\, c_n \, h_{\vol}(Q) ,
\]
where
\[
c_n :=  \frac{1}{(n! \, \omega_n)^{1/n}} \sim \sqrt{\frac{e}{2\pi}} \frac{1}{\sqrt n}.
\]
Moreover, if the Finsler metric $F$ is reversible, we have
\[
\widehat{h}_{\top}^{\BH}(F) \geq  \widehat{h}_{\top}^{\HT}(F) \geq 2 \2 c_n \, h_{\vol}(Q).
\]
\end{corollary}

\subsection{Finsler metrics with small topological entropy}  \label{ss:ver}
The following result is more precise than Proposition~\ref{prop:ver}.

\begin{prop} \label{p:ver}
Let $(\widetilde Q, \widetilde g)$ be a Riemannian globally symmetric space of non-compact type
and of rank~$\geq 2$.
Let $G$ be the connected component of the identity of the isometry group of $(\widetilde Q, \widetilde g)$.
Then there exist computable constants $c^{\HT} < c^{\BH} < 1$ that depend only on 
$(\widetilde Q, \widetilde g)$
with the following property:
For every discrete co-compact subgroup~$\Gamma$ of~$G$ that acts without fixed points on~$\widetilde Q$
and for every $\gve >0$ 
there exists a smooth reversible $G$-invariant Finsler metric~$F$ on $Q = \widetilde Q / \Gamma$
such that
\begin{eqnarray} 
\widehat h_{\top}^{\HT}(F) 
  &=& \widehat h_{\vol}^{\HT}(F)  
	\;\leq\; (1+\gve)\; c^{\HT} \; \widehat h_{\vol} (\widetilde g) , \label{e:HT3} \\ 
\widehat h_{\top}^{\BH}(F) 
  &=& \widehat h_{\vol}^{\BH}(F)  
	\;\leq\; (1+\gve)\; c^{\BH} \; \widehat h_{\vol} (\widetilde g).  \label{e:BH3}
\end{eqnarray}
In particular, $\widehat h_{\top}^{\HT}(Q, \cf_{\rev}) \leq c^{\HT} \; \widehat h_{\vol} (\widetilde g)$
and $\widehat h_{\top}^{\BH}(Q, \cf_{\rev}) \leq c^{\BH} \; \widehat h_{\vol} (\widetilde g)$.
\end{prop}

\proof
Fix a point $x_0 \in G$, let $K \subset G$ be the stabilizer of~$x_0$,
let $\frak{g}$ and~$\frak k$ be the Lie algebras of $G$ and~$K$, 
and let $\frak{g} = \frak{k} \oplus \frak{p}$ be the Cartan decomposition associated with~$x_0$.
(Then $\frak{p} \cong T_{x_0} \widetilde Q$.)
Choose a maximal abelian subalgebra $\frak a \subset \frak p$,
and let $W_{\frak a}$ be its Weyl group.

The set of $G$-invariant Finsler metrics on $\widetilde Q$ is in bijection with 
the set
$$
\cc \,=\, \left\{ C \subset \frak a \mid C \mbox{ centrally symmetric $W_{\frak a}$-invariant convex body} \right\} .
$$
In general, 
the Finsler metric assoiated with $C \in \cc$ is only continuous,
and it is smooth if and only if the boundary of~$C$ is smooth.
Let $C_0$ be ``the least convex" body in~$\cc$ of $\widetilde g$-volume one.
(For details we refer to \cite{Ver99}, but it should become clear from Examples~\ref{ex:ver} below
how to construct $C_0$.)
Since $\dim \frak a = \rank (G/K) \geq 2$, $C_0$ is not just a segment, 
and hence not an ellipsoid, i.e., the Finsler metric~$F_0$ associated with $C_0$
is not Riemannian. 
In fact, $F_0$ is not smooth.
Verovic shows that $F_0$ is the unique minimizer of $\widehat h_{\vol} (F)$ 
among all $G$-invariant continuous Finsler metrics on~$\widetilde Q$.
In particular, the constant~$c^{\BH}$ defined by
\begin{equation} \label{e:ver0}
\widehat h_{\vol}^{\BH}(F_0) \,=\, c^{\BH} \: \widehat h_{\vol}(\widetilde g)
\end{equation}
is strictly less than $1$. 
There is a simple formula computing this constant in terms of the Weyl data of~$\frak a$.

Fix $\gve >0$, choose a smooth body $C$ from $\cc$ such that
$$
C_0 \,\subset\, C \,\subset\, (1+\gve)\, C_0 ,
$$
and let $F$ be the associated Finsler metric.
Then
\begin{equation} \label{e:ver1}
\widehat h_{\vol}^{\BH} (F) \,\leq\, (1+\gve) \, \widehat h_{\vol}^{\BH} (F_0) ,
\end{equation}
cf.\ \S \ref{ss:volent}.
Since $\widetilde Q$ is of non-compact type, $G$ is semi-simple,
see for instance \cite[Proposition 6.38 (d)]{Zil10}.
It thus follows from \cite[Theorem 6.3 (2)]{DH07} that $F$ has negative flag curvature. 
Therefore, the extension of Manning's equality to reversible Finsler metrics  
in~\cite[Theorem~6.1]{Egl97} implies that
\begin{equation} \label{e:ver2}
\widehat h_{\top}^{\BH} (F) \,=\,  \widehat h_{\vol}^{\BH} (F) .
\end{equation}
The line~\eqref{e:BH3} follows from \eqref{e:ver2}, \eqref{e:ver1}, and \eqref{e:ver0}.

Define the constant $c^{\HT}$ by
$$
\widehat h_{\vol}^{\HT}(F_0) \,=\, c^{\HT} \; \widehat h_{\vol}(\widetilde g) .
$$
By \eqref{reversible} and the Santal\'o inequality, $c^{\HT} < c^{\BH}$.
Repeating the above arguments we obtain~\eqref{e:HT3}.
\proofend

\begin{exs}  \label{ex:ver}
{\rm
{\bf 1.}
Let $Q$ be a compact quotient of the symmetric space~$(\HH^2)^k$ of rank~$k$.
The maximal abelian subalgebra~$\frak a$ is~$\R^k$,
with Weyl chamber $\frak a^+ = \R_{>0}^k$.
The set of positive roots is given by the dual basis $\gve_1, \dots, \gve_k$
of the standard basis $e_1, \dots, e_k$ of~$\R^k$.

The standard Riemannian metric~$g$ on~$Q$ corresponds, up to to scaling, to the closed unit ball~$B$ 
in~$\frak a = \R^k$,
and we take $F_0$ to be the non-smooth Finsler metric corresponding to the cross-polytope~$C_0$
in~$\frak a$ with vertices $\pm e_1, \dots, \pm e_k$.
Now Proposition~2.2 in~\cite{Ver99} shows that
$$
\begin{array}{rclcl}
\widehat h_{\vol}^{\BH} (F_0) &=& 
\left( \vol_{F_0}^{\BH}(Q) \right)^{1/(2k)} 
     \displaystyle \max_{v \in C_0 \cap \overline{\frak a^+}} \bigl( \gve_1(v) + \dots + \gve_k(v) \bigr) &=&  
		\left( \vol_{F_0}^{\BH}(Q) \right)^{1/(2k)}  
		\\ [1em]
\widehat h_{\vol}^{\BH} (g) &=& 
\left( \vol_g^{\BH}(Q) \right)^{1/(2k)} 
     \displaystyle \max_{v \in B \cap \overline{\frak a^+}} \bigl( \gve_1(v) + \dots + \gve_k(v) \bigr) &=&
		\left( \vol_g^{\BH}(Q) \right)^{1/(2k)} \sqrt{k} .
\end{array}
$$
Let $D_{x_0}(F_0)$ resp.\ $D_{x_0}(g)$ be the unit ball of $F_0$ resp.\ $g$
in $T_{x_0}\widetilde Q \sim \frak p$.
By $G$-invariance of~$F_0$ and~$g$, and in view of the definition of the
Busemann--Hausforff volume in~\S \ref{ss:vol}, 
\begin{equation} \label{e:volq}
\frac{\vol_{F_0}^{\BH}(Q)}{\vol_g^{\BH}(Q)} \,=\,
\frac{\vol_{g}(D_{x_0}(g))}{\vol_g(D_{x_0}(F_0))} .
\end{equation}

\begin{lem}
The quotient on the right of \eqref{e:volq} is equal to $\frac{(2k)!}{2^k \2 k!}$.
% $(2k-1)!!} \,=\, 1 \cdot 3 \cdots (2k-1)$.
\end{lem}

\proof
We have
$D_{x_0}(F_0) = \Ad (K) (C_0)$ and $D_{x_0}(g) = \Ad (K) (B)$.
For $k=1$, when $G = \SL(2;\R)$ and $K = \SO(2;\R)$,
a computation in the orthogonal basis 
$\begin{brsm} 1&0\\0&-1\end{brsm}$,
$\begin{brsm} 0&1\\1&0\end{brsm}$ of~$\frak p$
shows that the orbit $\Ad (K) p$ of a point $p \in \frak p$
is the circle through~$p$.
For general~$k \geq 1$, the $\Ad(K)$-orbit of $p = (p_1, \dots, p_k) \in \frak p$
is the $k$-torus made of circles of radius~$|p_i|$.
Since the restrictions of $F_0$ and $g$ to $\frak p$ are $\Ad (K)$-invariant,
it follows that the quotient on the right of~\eqref{e:volq} is equal to the quotient of 
the two integrals
$$
\int_{B \cap \overline{\frak a^+}} 
   \left( x_1 x_2 \cdots x_k \right)  dx_1 dx_2 \cdots dx_k,
\qquad
\int_{C_0 \cap \overline{\frak a^+}} 
   \left( x_1 x_2 \cdots x_k \right)  dx_1 dx_2 \cdots dx_k  .
$$
The first integral equals $\frac{1}{2^k \2 k!}$ and the second equals $\frac{1}{(2k)!}$
as one finds
using Fubini's theorem and induction.
\proofend

Together with the lemma we conclude that
$$
c^{\BH}_k \,:=\, c^{\BH} \left((\HH^2)^k \right) \,=\, \frac{\widehat h_{\vol}^{\BH} (F_0)}{\widehat h_{\vol}^{\BH} (g)}
 \,=\,   \left( \frac{(2k)!}{k!} \right)^{1/(2k)} \frac{1}{\sqrt{2k}} .
$$

We next compute the Holmes--Thompson volumes $\vol_{F_0}^{\HT}(Q)$ and $\vol_{g}^{\HT}(Q)$.
Denote by $C_0^*$ and~$B^*$ the polar sets of $C_0$ and~$B$
in~$\frak p^*$, respectively,
and by $g^*$ the dual Riemannian metric on~$\frak p^*$.
By $G$-invariance of~$F_0$ and~$g$, and in view of the definition of the
Holmes--Thompson volume in~\S \ref{ss:vol}, 
\begin{equation} \label{e:volqHT}
\frac{\vol_{F_0}^{\HT}(Q)}{\vol_g^{\HT}(Q)} \,=\,
\frac{\vol_{g^*}(D_{x_0}^*(F_0))}{\vol_{g^*}(D_{x_0}^*(g))} .
\end{equation}
Using that $D_{x_0}^*(F_0) = \Ad^* (K) (C_0^*)$ and $D_{x_0}^*(g) = \Ad^* (K) (B^*)$,
that the polar set $C_0^*$ of the cross-polytope~$C_0$ is the unit cube, 
and the computations
$$
\int_{B^*   \cap \R_{\geq 0}^k} \left( x_1 \cdots x_k \right) \, dx_1 \cdots dx_k \,=\, \frac{1}{2^k\, k!} ,
\qquad
\int_{C_0^* \cap \R_{\geq 0}^k} \left( x_1 \cdots x_k \right) dx_1 \cdots dx_k \,=\, \frac{1}{2^k} ,
$$
we find that the right quotient in \eqref{e:volqHT} is $k!$.
Therefore,
$$
c_k^{\HT} \,:=\, c^{\HT} \left((\HH^2)^k \right) \,=\, 
\frac{\widehat h_{\vol}^{\HT} (F_0)}{\widehat h_{\vol}^{\HT} (g)}
 \,=\,   \left( k! \right)^{1/(2k)} \frac{1}{\sqrt{k}} .
$$

It is shown in \cite{Mer16} that
$$
h_{\vol}^{\sym} (Q) \,=\, h_{\vol}(Q) .
$$
Together with Theorem~\ref{t:Finslerintro} and Proposition~\ref{prop:ver} we obtain
\begin{eqnarray*}
2 \2 c_{2k} \,h_{\vol} (Q) 
&\leq& \widehat h_{\top}^{\HT} (Q,\cf_{\rev}) \;\leq\; 
c_k^{\HT} \,h_{\vol} (Q) , \\
2 \2 c_{2k} \,h_{\vol} (Q) 
&\leq& \widehat h_{\top}^{\BH} (Q,\cf_{\rev}) \;\leq\; 
c_k^{\BH} \,h_{\vol} (Q) .
\end{eqnarray*}

The sequence $c_k^{\BH}$, $k \geq 2$, is monotone decreasing to $\sqrt{\frac 2 e} \approx 0.858$, 
starting with
$$
c_2^{\BH} \approx 0.931
\quad \mbox{ and } \quad 
c_3^{\BH} \approx 0.907 .
$$
The sequence $c_k^{\HT}$ is monotone decreasing to $\sqrt{\frac 1 e} \approx 0.616$, 
starting with
$$
c_2^{\HT} \approx 0.841
\quad \mbox{ and } \quad 
c_3^{\HT} \approx 0.778 .
$$
In contrast, the sequence $2 \2 c_{2k} =  \frac{2}{\bigl( (2k)! \, \omega_{2k} \bigr)^{1/(2k)}}$ 
is monotone decreasing like
$\sqrt{\frac{e}{2\pi}} \frac{1}{\sqrt{2k}}$, starting with
$$
2 \2 c_4 \approx 0.606
\quad \mbox{ and } \quad 
2 \2 c_6 \approx 0.508 .
$$

The constant $h_{\vol}(Q)$ can be computed as follows.
On $(\HH^2)^k$ the minimum of the volume entropies among symmetric metrics
is attained exactly by multiples of $g^k := g \times \dots \times g$,
where $g$ is the metric on~$\HH$ of constant curvature~$-1$,
see~\cite[\S 2]{CF03}.
Since $h_{\vol}(g) =1$, we have 
$h_{\vol}(g^k) = \sqrt{\sum (h_{\vol}(g))^2}    
= \sqrt{k}$.
Hence
$$
h_{\vol}(Q) \,=\, h_{\vol}^{\sym} (Q) \,=\, 
\left( \vol_{g^k}(Q)\right)^{1/2k} \sqrt{k} . 
$$
For instance, if $Q$ is the product of orientable surfaces of genus~$k_j$,
then
$$
\vol_{g^k}(Q) \,=\, \prod_{j=1}^k \area_g(Q_{k_j}) \,=\, 
2^k \prod_{j=1}^k \sqrt{\pi (k_j-1)} .
$$

\medskip
{\bf 2.}
Take the $5$-dimensional symmetric space $\SL (3;\R) / \SO (3;\R)$ of rank~$2$.
The ``least convex" body $C_0$ from~$\cc$ is a regular hexagon.
We scale this hexagon such that it is the hexagon~$H_{{\mbox{\tiny in}}}$ 
inscribed the unit disc~$B$ of~$\R^2 = \frak a$.
Verovic computed in \cite[p.\ 1644]{Ver99}
that for the Finsler metric corresponding to $\sqrt{\frac 23} \, H_{{\mbox{\tiny in}}}$, 
the volume growth is~$2$. 
Hence the volume growth of the Finsler metric corresponding to~$H_{{\mbox{\tiny in}}}$ 
is $\sqrt{\frac 32} \,2$.
Further, the volume growth of the Riemannian metric corresponding to~$B$  
is $2 \sqrt{2}$.

To compute the volumes,
since $\SO(3;\R)$ is 3-dimensional
we now have to take $r^3 \2 dx dy$ as density on~$\frak a$.
The integral of~$r^3$ over~$B$ and over $H_{{\mbox{\tiny in}}}$ are, respectively,
$\frac{2\pi}{5}$ and 
$$
I_{{\mbox{\tiny in}}} \,=\, \frac{3 \sqrt{3}}{640}  \left( 27 \ln 3 + 68 \right) .
$$
With this 
we find along the lines of the previous examples that 
$$
c^{\BH} (\SL (3;\R) / \SO (3;\R)) \,=\,  \left( \frac{2\pi}{5 I_{{\mbox{\tiny in}}}} \right)^{1/5} \frac{\sqrt{3}}{2}
 \,\approx\, 0.95 .
$$

The polar set of $H_{{\mbox{\tiny in}}}$ is a regular hexagon $H_{{\mbox{\tiny out}}}$ 
circumscribed the unit co-disc. 
After identifying $\frak a$ with $\frak a^*$ by the inner product, 
we have that $H_{{\mbox{\tiny out}}}$ is obtained from $H_{{\mbox{\tiny in}}}$
by dilation by $\frac{2}{\sqrt{3}}$ and rotation by~$\frac{\pi}{6}$.
Hence the integral of $r^3$ over $H_{{\mbox{\tiny out}}}$ is
$$
I_{{\mbox{\tiny out}}} \,=\, \left( \frac{2}{\sqrt{3}} \right)^{5} I_{{\mbox{\tiny in}}}.
$$
Therefore, 
$$
c^{\HT} (\SL (3;\R) / \SO (3;\R)) \,=\,  
\left( \frac{5 I_{{\mbox{\tiny out}}} }{2\pi} \right)^{1/5} \frac{\sqrt{3}}{2} \,=\, 
\frac{\sqrt{3}}{2 \: c^{\BH}}
 \,\approx\, 0.912  .
$$

These two constants
should be compared with the constant $2 c_5 \approx 0.551$
for the lower bound in Corollary~\ref{minhtop}.
}
\end{exs}

\begin{question}
{\rm 
Recall that the non-smooth Finsler metric~$F_0$ is the unique minimizer of
$\widehat h_{\vol}^{\BH}(F)$ among $G$-invariant continuous Finsler metrics on~$Q$.
%FFF auch für \HT  check Verovich p 1647-48
This Finsler metric has a high degree of symmetry: 
Its restriction to $\frak a$ is invariant under the Weyl group, and it is $G$-invariant.
Since in Theorems~\ref{t:katok} and~\ref{t:BCG} the Riemannian minimizers are the locally symmetric
metrics, 
one may expect that $F_0$ minimizes $\widehat h_{\vol}^{\BH}(F)$ and
$\widehat h_{\vol}^{\HT}(F)$
among {\it all} continuous Finsler metrics on~$Q$.
Would this imply that there are no smooth minimizers?
}
\end{question}

%%%%%%%%%%%%%%%%%%%%%%%%%%%%%%%%%%%%%%%%%%%%%%%%%%%%%%%%%%%%%%%%%%%%%%%%%%%%%%%%%%%%%

\section{A lower entropy bound for Reeb flows on spherizations} \label{s:Floer}    

Recall from Theorem~\ref{t:mainintro} that there cannot be a uniform 
lower bound for the normalized topological entropy of Reeb flows.
In this section we show that for many base manifolds~$Q$,
one nevertheless has a control on the entropy collapse of Reeb flows 
on the spherization~$S^*Q$ in terms of the geometry of their defining star fields:
Entropy collapse can only happen if some fibers are far from convex.
The proof relies on Floer homology.

We consider a closed manifold $Q$ and two Reeb flows on~$S^*Q$,
one arbitrary and one Finsler. 
As in the previous section and as in Appendix~\ref{ss:RSG} we work in~$T^*Q$.
We then have two Hamiltonian functions $H,F \colon T^*Q \to \R$
that are fiberwise positively homogeneous of degree one 
and smooth and positive away from the zero section.
Again we denote by $\phi_H^t$ the flow of~$H$ on $S^*(H) = H^{-1}(1)$,
and similarly for~$F$.
Let $\sigma_-$ and $\sigma_+$ be the smallest positive numbers such that
$$
\frac{1}{\sigma_-}\, F \,\leq\, H \,\leq\, \sigma_+\, F  \quad \mbox{ on } T^*Q .
$$
For the co-disk bundles we then have
$$
\frac{1}{\sigma_+}\, D^*(F) \,\subset\, D^*(H) \,\subset\, \sigma_- \, D^*(F),
$$
see Figure~\ref{fig.GHG}.

\begin{figure}[h]   
 \begin{center}
  \psfrag{H}{$D^*_q (H)$}  \psfrag{F-}{$\frac{1}{\sigma_+} D^*_q (F)$} \psfrag{F}{$\sigma_- \, D^*_q (F)$} 
  \leavevmode\includegraphics{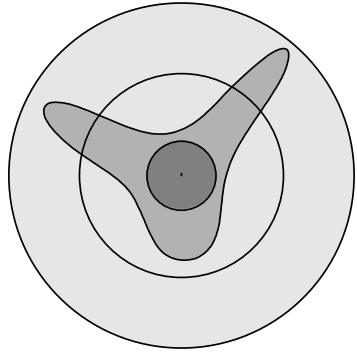}
 \end{center}
 \caption{The co-disks $\frac{1}{\sigma_+} D^*_q (F) \subset D^*_q (H) \subset \sigma_- \, D^*_q (F)$ in $T_q^*Q$}   
\label{fig.GHG} 
\end{figure}

The number
$$
\sigma (H;F) \,:=\, \sigma_- \, \sigma_+ .
$$
does not change under rescalings of~$H$ or~$F$. 
We have $\sigma (H;F) \geq  1$ with equality if and only if $H=cF$ for some positive number~$c$. 
%We are particularly interested in the module of starshapedness $\sigma(H;G)$ with respect to a Riemannian Hamiltonian $G$, i.e.\ the square root of a positive definite quadratic form.} 
Moreover, $\sigma (H;F) \leq  \sqrt n$ if $H$ is a reversible Finsler Hamiltonian and 
$F$ is chosen to be the Riemannian Hamiltonian associated with the outer Loewner ellipsoids of~$D^*(H)$, see~\eqref{john}.

\begin{prop} \label{p:control}
Let $F$ be a (possibly irreversible) $C^\infty$-regular Finsler metric on the closed manifold~$Q$.
Then for every $C^\infty$-smooth Reeb flow $\phi_H^t$ on~$S^*Q$ we have
$$
\widehat h_{\top}^{\HT}(\phi_H) \,\geq\, \frac{1}{\sigma (H;F)} \, \widehat h_{\vol}(F) .
$$
\end{prop}

\proof
After scaling $F$ we can assume that $\sigma_- =1$. 
We abbreviate $\sigma (H;F) = \sigma_+ =: \sigma$.

\begin{lem} \label{le:topvol}
$h_{\top}(\phi_H) \geq h_{\vol}(F)$.
\end{lem}

\proof
The lemma can be extracted from~\cite{MacSch11}. We briefly review the proof.
Instead of working with $\phi_H$ and~$\phi_F$, we work with the Hamiltonian flows 
$\Phi_H$ and $\Phi_F$ on~$T^*Q$ of the functions $H^2$ and~$F^2$.
Then $\Phi_H = \phi_H$ on~$S^*(H)$.
Using the variational principle for topological entropy and the homogeneity of~$H^2$ one finds
$$
h_{\top} (\phi_H) \,=\, h_{\top} (\Phi_H |_{S^*(H)}) \,=\, h_{\top} (\Phi_H |_{D^*(H)}) .
$$
Fixing a point $q \in Q$ we can further estimate, using Yomdin's theorem from~\cite{Yom87}
and the $C^\infty$-smoothness of~$\Phi_H$,
$$
h_{\top} (\Phi_H |_{D^*(H)}) \,\geq\, \lim_{n \to \infty} \frac 1n \log \mu_{g^*} \bigl( \Phi_H^n (D^*_q(H)) \bigr) .
$$ 
Here
$\mu_{g^*} (S)$ denotes the Riemannian volume of the submanifold~$S \subset T^*Q$ with respect to the restriction to~$S$ of the Riemannian metric on~$T^*Q$
induced by a Riemannian metric~$g$ on~$Q$.
In Theorem~4.6 and Section~5.1 of~\cite{MacSch11} it is shown by Lagrangian Floer homology 
that for every $\gve >0$ there exists $N(\gve)$ such that
$$
\mu_{g^*} \bigl( \Phi_H^n (D^*_q(H)) \bigr) \,\geq\, 
  \vol_g(Q) \, e^{(\gamma (F)-\gve) \2 n} \quad
\mbox{ for } n \geq N(\gve) 
$$
where $\gamma (F)$ is the exponential growth rate of the number of elements in the fundamental group of~$Q$ 
that can be represented by a loop of $F$-length $\leq R$.
It is easy to see that $\gamma (F) = h_{\vol} (F)$.
(The proof in \cite[Prop.\ 9.6.6]{HK95} given for a Riemannian~$F$ applies without
changes to a general Finsler metric.)
The lemma follows.
\proofend

In view of the inclusion $\frac 1 \sigma D^*(F) \subset D^*(H)$ we infer from Lemma~\ref{le:topvol} that
\begin{eqnarray*}
\widehat h_{\top}^{\HT} (\phi_H) &=& \left( \vol_H^{\HT}(Q)  \right)^{1/n} \, h_{\top} (\phi_H) \\
&\geq& \frac{1}{\sigma} \left( \vol_F^{\HT}(Q) \right)^{1/n} \, h_{\vol} (F) \\
&=& \frac{1}{\sigma} \, \widehat h_{\vol}(F) 
\end{eqnarray*}
as claimed.
\proofend

We now define the module of starshapedness of~$H$ by
$$
\sigma (H) \,:=\, \inf \left\{ \sigma (H;F) \mid 
          \mbox{$F$ is a Finsler metric} \right\} .
$$
While an individual $\sigma (H;F)$ can be large even for a Riemannian Hamiltonian~$H$,
the number~$\sigma (H)$ is a measure for the maximal starshapedness, 
or non-convexity, of the fibers of~$D^*(H)$. 
For instance, $\sigma (H) = 1$ if and only if $H$ is Finsler. 
%FFF while $\sigma (H)$ can be arbitrarily large on the class of Finsler structures.
%    Aber: one can improve the 1/\sigma near the end of the above proof by looking more carefully at the volume quotient, as in \S 7.4
%
From Proposition~\ref{p:control} and Corollary~\ref{minhvol} we obtain
the following result.

\begin{corollary} \label{c:RR}
Let $Q$ be a closed manifold. 
For every $C^\infty$-smooth Reeb flow $\phi_H^t$ on~$S^*Q$ we have
$$
\widehat h_{\top}^{\HT}(\phi_H) \,\geq\, \frac{c_n}{\sigma (H)} \, \widehat h_{\vol}(Q) .
$$
\end{corollary}

\begin{rems}
{\rm
(1)
In the special case that $H=F$ is a $C^\infty$-regular reversible Finsler Hamiltonian, 
Proposition~\ref{p:control} applied to Riemannian metrics and 
the Loewner bound~\eqref{john} yield the uniform lower bound
$$
\widehat h_{\top}^{\HT}(F) \,\geq\, \tfrac{1}{\sqrt n} \, \widehat h_{\vol} (Q) .
$$
Even in this special case, this lower bound for $\widehat h_{\top}^{\HT}(F)$ 
coming from Floer homology and from the Loewner bound is 
only slightly weaker than the lower bound
$$
\widehat h_{\top}^{\HT}(F) \,\geq\, 2 c_n \, \widehat h_{\vol} (Q)
$$
from Corollary~\ref{minhtop} that comes from Manning's inequality and the 
Loewner bound.
Indeed, recalling that $c_n = \frac{1}{ (n! \, \omega_n)^{1/n} }$,
the function $f(n) = 2 c_n \sqrt{n} \colon \N \to [1, \infty)$ is strictly monotone increasing, 
with 
$$
f(2) = \tfrac{2}{\sqrt{\pi}} \approx 1.13 
\quad \mbox{and} \quad 
\lim_{n \to + \infty} f(n) = \sqrt{\tfrac{2e}{\pi}} \approx 1.315 .
$$

%omega[n_] := (Pi)^{n/2} / Gamma[n/2+1]
%c[n_] := 1/((Factorial[n] omega[n])^{1/n})
%f[n_] := 2 c[n] Sqrt[n]

\medskip
(2)
In the case that $H$ is a Finsler Hamiltonian and $F$ is a Riemannian Hamiltonian, 
we have obtained the inequality in Lemma~\ref{le:topvol} 
in \S \ref{ss:lbne} by estimating
$$
h_{\top}(\phi_H) \geq h_{\vol}(\phi_H)  \geq h_{\vol}(F) .
$$
The first inequality, which is Manning's inequality, also holds for $C^\infty$-smooth Reeb flows, 
see Theorem~\ref{t:manninggen}.
The second inequality holds in the Finsler case in view of the inclusion of balls~\eqref{monotonicity-h},
which follows from the triangle inequality. 
But in the Reeb case there is no triangle inequality.
Floer homology 
(or, more precisely: properties of Floer continuation maps that stem from the Floer--Gromov compactness theorem for $J$-holomorphic strips)
makes up for this.

\medskip
(3)
Proposition~\ref{p:control} and Corollary \ref{c:RR} are  
interesting only if $\widehat h_{\vol}(F)$ and $h_{\vol}(Q)$ are positive, 
which is possible only if the fundamental group of~$Q$ has exponential growth. 
The results in~\cite{MacSch11} imply meaningful variations of 
Proposition~\ref{p:control} and Corollary~\ref{c:RR}
for many other manifolds. For instance, 
assume that $Q$ is a simply connected manifold such that the exponential growth rate
$\gamma (\Omega Q)$ of the dimension of the $\Z_2$-homology of degree $\leq k$ of the based loop space~$\Omega Q$
is positive.
Then 
$$
\widehat h_{\top}^{\HT} (\phi_H) \,\geq\, \frac{1}{\sigma (H;F)} \; C(F) \: \gamma (\Omega Q) 
$$
with a positive constant $C(F)$ that does not change under rescalings of~$F$.  

\medskip
(4)
We refer to \cite{Dah20} 
for a thorough study of continuity properties of topological entropy implied
by Floer homological techniques.
}
\end{rems}

\section{Entropy collapse for Reeb flows in dimension $3$}  \label{s:collapse3}

In this section we prove Theorem~\ref{t:mainintro} in dimension 3.

\subsection{Recollections on open books}
\label{ss:recollections}
In this paragraph we collect results on open books needed in our proof.
For more information and details we refer to~\cite{Etn06} and \cite[\S 4.4]{Gei08}.

Let $M$ be a closed connected orientable 3-manifold.
An {\it open book} for~$M$ is a triple $(\Sigma,\psi,\Psi)$,
where $\Sigma$ is a compact oriented surface with non-empty boundary~$\partial \Sigma$
and $\psi$ is a diffeomorphism of~$\Sigma$ that is the identity near the boundary
such that there is a diffeomorphism~$\Psi$ from
$$ 
M(\psi) := \Sigma(\psi) \cup_{\id} \left( \overline \D \times \partial \Sigma \right) 
$$
to $M$.
Here $\Sigma(\psi)$ denotes the mapping torus
$$
\Sigma(\psi) \,=\, \left( [0,2\pi] \times \Sigma \right) / \sim 
$$
where $(2\pi,p) \sim (0,\psi(p))$ for each $p \in \Sigma$,
and $\overline \D$ is the closed unit disk. 
Viewing $S^1$ as the interval~$[0,2\pi]$ with endpoints identified, 
we write $\partial (\Sigma(\psi))$ as $S^1 \times \partial \Sigma$.
The manifold~$M$ is thus presented as the union of the mapping torus~$\Sigma(\psi)$
and finitely many full tori, one for each boundary component of~$\Sigma$,
glued along their boundaries by the identity map
$$
\partial (\Sigma(\psi)) \,=\, S^1 \times \partial \Sigma
\,\stackrel{\id\,}{\longrightarrow}\, 
\partial ( \overline \D \times \partial \Sigma ) .
$$
We remark that the diffeomorphism 
$$
\Psi \colon M(\psi) \to M
$$ 
is part of the definition of the open book. 
If $\psi'$ is another diffeomorphism of~$\Sigma$ that is the identity near the boundary and is 
isotopic to~$\psi$ via an isotopy that fixes each point of~$\partial \Sigma$, 
then $M(\psi')$ is diffeomorphic to~$M(\psi)$.
We also remark that what we call an open book is usually called 
an abstract open book decomposition
in the literature.

For each $\theta \in S^1$ denote by $\Sigma_{\theta}^\circ$ the image under the diffeomorphism~$\Psi$ 
of the union of $\{\theta\} \times \Sigma$ with the union of half-open annuli 
$$
A_{\theta} \,=\, \bigl\{ (\theta,r) \in \overline \D \setminus \{0\} \bigr\} 
          \times \partial \Sigma  .
$$
The closure $\Sigma_{\theta}$ of $\Sigma_{\theta}^\circ$, called a page, is diffeomorphic to~$\Sigma$, 
and the common boundary of the~pages~$\Sigma_{\theta}$, called the binding of the open book, 
is the image under~$\Psi$ of $ \{0\} \times \partial \Sigma \subset \D \times \partial \Sigma$.  
The orientation of~$\Sigma$ induces orientations on the pages and the binding.

There are several different beautiful constructions proving the existence
of an open book for every 3-manifold~$M$ as above.
The first of these constructions was given by 
J.~W.~Alexander~\cite{Ale20} as early as~1920,
who used his findings that every such~$M$ is a branched covering of the 3-sphere 
branching along a link
and that every link in~$\R^3$ can be obtained as the closure of a braid,
see also \cite[p.\ 340]{Rol76}. 
Alexander's construction in fact provides an open book such that $\Sigma$ has just one boundary component. 

\medskip \noindent
{\bf Contact structures.}
Let $M$ be a closed connected oriented 3-manifold and $(\Sigma, \psi,\Psi)$ be an open book for~$M$. 
%We will assume for simplicity that $\partial \Sigma$ is connected. 

\begin{defn}
{\rm
A contact form $\alpha$ on $M$ is said to be {\it adapted to the open book} $(\Sigma,\psi,\Psi)$ if 
\begin{itemize}
\item[$\bullet$]
$\alpha$ is positive on the binding,
\item[$\bullet$]
$d\alpha$ is a positive area form on the interior of every page.
\end{itemize}
}
\end{defn}

It is not hard to see that a contact form $\alpha$ is adapted to an open book if and only if
\begin{itemize}
  \item the Reeb vector field $R_\alpha$ is positively transverse to the interior of the pages,
  \item the Reeb vector field is tangent to the binding and induces the positive orientation on the binding.
\end{itemize}

\begin{defn}
{\rm 
A contact $3$-manifold $(M,\xi)$ is said to be supported by an open book 
$(\Sigma,\psi,\Psi)$ 
if there exists a contact form~$\alpha$ on~$(M,\xi)$ adapted to this open book.
}
\end{defn}

\begin{rem} \label{rem:isotopy}
{\rm If a contact $3$-manifold $(M,\xi)$ is supported by an open book 
$(\Sigma,\psi,\Psi)$ 
and if $\psi'$ is another diffeomorphism of~$\Sigma$ that is the identity near $\partial \Sigma$ 
and that is isotopic to~$\psi$ via an isotopy that fixes~$\partial \Sigma$ 
pointwise, then $(M,\xi)$ is also supported by an open book~$(\Sigma,\psi',\Psi')$. 
This follows easily from the fact that for such a~$\psi'$ there exists 
a diffeomorphism from $M(\psi)$ to~$M(\psi')$ that takes pages to pages.
}
\end{rem}

The following result of Giroux shows the central role played by 
open books in $3$-dimen\-sional contact topology.

\begin{theorem}[Giroux] \label{t:Giroux}
Given a closed connected oriented contact $3$-manifold~$(M,\xi)$, 
there exists an open book for~$M$ supporting~$(M,\xi)$. 
Moreover, the open book can be chosen to have connected binding.
Two contact structures supported by the same open book are diffeomorphic.
\end{theorem}

For the proof of the first and the third assertion we
refer to~\cite[Theorem~3 and Proposition~2]{Gir02}
and to~\cite[Theorem~4.6 and Proposition 3.18]{Etn06}.
That the binding can be assumed to be connected is shown 
in~\cite[Corollary~4.25]{Etn06} and in~\cite{CoHo08}.

\subsection{Proof of entropy collapse in dimension $3$} \label{ss:proof3}

We now proceed with the proof of the main result of this section.
 
\begin{theorem}\label{t:3d}
Let $(M,\xi)$ be a closed co-orientable contact manifold of dimension~$3$. 
Then for every $\varepsilon > 0$ there exists a contact form~$\alpha$ on~$(M,\xi)$ such that 
$\vol_{\alpha} (M) =1$ and $h_{\top}(\alpha) \leq \gve$.
\end{theorem}

While our proof works verbatim when $\partial \Sigma$ is not connected, the geometry in our argument is easier to visualize for connected~$\partial \Sigma$, so we assume this property.

\smallskip
The structure of the proof is as follows.
\begin{itemize}
\item 
Given $(M,\xi)$ as in Theorem~\ref{t:3d} we use the first statement in Theorem~\ref{t:Giroux} 
to obtain an open book~$(\Sigma,\psi,\Psi)$ for~$M$ that supports~$(M,\xi)$.

\item 
We then apply a classical recipe due to Thurston--Winkelnkemper to construct for each $\gve >0$ 
a contact form~$\widetilde \alpha_\gve$ adapted to~$(\Sigma,\psi,\Psi)$ 
with $\vol_{\widetilde \alpha_\gve}(M) = 1$ 
and $h_{\top}(\widetilde \alpha_\gve) \leq \gve$.
    
\item 
By the second statement of Theorem \ref{t:Giroux}, 
$\ker \Psi_* \widetilde\alpha_\gve$ is diffeomorphic to~$\xi$ by a diffeomorphism~$\rho_\gve$. 
Hence $(\rho_\gve \circ \Psi)_* (\widetilde \alpha_\gve)$ is a contact form on~$(M,\xi)$ 
with the properties asserted in Theorem~\ref{t:3d}.
\end{itemize}

For the construction of $\widetilde \alpha_{\gve}$, we first construct on the 
mapping torus~$\Sigma(\psi)$ for all small $s>0$ contact forms~$\alpha_s$
with $\vol_{\alpha_s}(\Sigma (\psi)) = O(s)$ and $h_{\top}(\alpha_s) = O(1)$.
Crucially, near the boundary of~$\Sigma(\psi)$ these contact forms are such that 
they extend to contact forms (also denoted~$\alpha_s$) 
on the full torus $\overline \D \times \partial \Sigma$
in such a way that the Reeb flows are linear on each torus $S^1(r) \times \partial \Sigma$.
Therefore, even though the Reeb vector fields ``explode" in the interior of the full torus 
as $s \to 0$ (see Figure~\ref{XA.fig}),
the topological entropy on the full torus vanishes for all~$s$.
Since also $\vol_{\alpha_s}(\overline \D \times \partial \Sigma) = O(s)$,
we find that $\vol_{\alpha_s}(M) = O(s)$ and $h_{\top}(\alpha_s) = O(1)$
for all small $s >0$.
The form $\widetilde \alpha_\gve$ is now obtained by taking $s$ small and rescaling~$\alpha_s$.

\medskip
\noindent
{\it Proof of Theorem \ref{t:3d}.} 

\medskip \noindent
{\bf Step 1: A family of contact forms $\alpha_s$ on $\Sigma(\psi)$.}
By Theorem~\ref{t:Giroux} there exists an open book $(\Sigma,\psi,\Psi)$ for~$M$ that supports~$(M,\xi)$. 
We first choose a collar neighbourhood $N \subset \Sigma$ of~$\partial \Sigma$ on which 
$\psi$ is the identity. Thus $N$ is diffeomorphic to $[1,1+\delta] \times \partial \Sigma$, 
and we have polar coordinates $(r,x)$ for~$N$,
where $x$ is the angular coordinate for $\partial \Sigma$,
such that the boundary of~$\partial \Sigma$ corresponds to~$r=1$. 

Choose an area form~$\omega$ on~$\Sigma$ such that $\omega = dx \wedge dr$ on~$N$. 
%??? with $\int_\Sigma \omega = 1$. 
By Remark~\ref{rem:isotopy} and by Moser's isotopy theorem 
we can assume that the diffeomorphism~$\psi$ is a symplectomorphism 
of~$(\Sigma, \omega)$.
Since $H^2(\Sigma,\partial \Sigma;\R)$ vanishes, there exists a primitive~$\lambda$ of~$\omega$ 
that equals $(2-r) \2 dx$ on~$N$.

\begin{figure}[h]  
 \begin{center}
  \psfrag{r}{$r$} \psfrag{1}{$1$} \psfrag{1d}{$1 \! + \:\!\! \delta$}  \psfrag{N}{$N$}  \psfrag{th}{$x$}   \psfrag{S}{$\Sigma$}
  \leavevmode\includegraphics{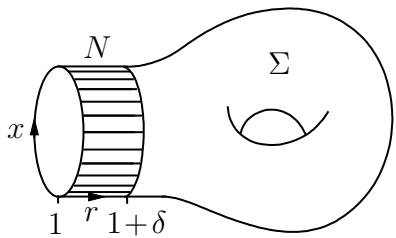}
 \end{center}
 \caption{The neighbourhood $N$ of $\partial \Sigma$}   \label{NSigma.fig}
\end{figure}

We now construct for each sufficiently small $s>0$ a contact form~$\alpha_s$ 
on the mapping torus~$\Sigma(\psi)$. For this, let $\chi \colon [0,2\pi] \to [0,1]$ 
be a smooth monotone function such that $\chi(0)=0$, $\chi(2\pi)=1$, and $\chi'$ 
has support in~$(0,2\pi)$. 
On $[0,2\pi] \times \Sigma$ define the $1$-form 
\begin{equation}
\alpha_s \,:=\, d\theta + 
       s \2 \bigl( (1-\chi(\theta)) \2 \lambda + \chi(\theta) \2 \psi^*\lambda \bigr).   
\end{equation}
By the properties of~$\chi$, each 1-form $\alpha_s$ descends to a 1-form 
on~$\Sigma(\psi)$, that we still denote by~$\alpha_s$.
Using that $\psi$ is a symplectomorphism of $(\Sigma, \omega)$
we compute that
\begin{equation} \label{e:s^2}
\alpha_s \wedge d \alpha_s \,=\, s\2 d\theta \wedge \omega + O(s^2) .
\end{equation}
Hence there exists $s_0>0$ such that $\alpha_s$ is a contact form 
for all $s \in (0,s_0]$. 

We now compute the Reeb vector field $R_{\alpha_s}$. 
With $\lambda_{\theta} := (1-\chi(\theta)) \2 \lambda + \chi(\theta) \2 \psi^*\lambda$
one checks that
\begin{equation} \label{e:Ralpha}
  R_{\alpha_s} \,=\, \frac{\partial_{\theta} + Y}{1 + s \lambda_{\theta} (Y)},
\end{equation}
where $Y$ is the vector field that is tangent to $\{\theta\} \times \Sigma$ 
for all $\theta \in S^1$ 
and satisfies 
$$
\iota_Y \omega \,=\, \chi'(\theta) (\psi^*\lambda - \lambda) .
$$
The formula~\eqref{e:Ralpha} shows that $R_{\alpha_s}$ is positively 
transverse to each surface $\{\theta\} \times \Sigma$.
 
The next lemma gives an upper bound for $h_{\top}(\phi_{\alpha_s})$ when $s$ 
is sufficiently small. 
Notice that it makes sense to talk about $h_{\top}(\phi_{\alpha_s})$, 
since $\Sigma(\psi)$ is compact and $R_{\alpha_s}$ is tangent to~$\partial \Sigma(\psi)$. 
 
\begin{lem} \label{returntimebounded}
There exists $s_1 \in (0,s_0)$ and a constant $E>0$ such that for every $s \in (0,s_1]$,
\begin{eqnarray}\label{returntimebound}
h_{\top} (\phi_{\alpha_s}) \leq E.
\end{eqnarray}
\end{lem}

\proof 
Choose $s_1 >0$ such that $\frac{1}{2} \leq \frac{1}{1+ s \lambda_\theta (Y)} \leq 2$ on~$\Sigma (\psi)$
for every $s \in (0,s_1]$. 
For each such $s$ define the function $f_s = \frac{1}{1+ s \lambda_\theta (Y)}$ 
on~$\Sigma (\psi)$.
Then for every $s \in (0,s_1]$ we have $R_{\alpha_s} = f_s \, (\partial_{\theta} + Y)$
with $\frac{1}{2} \leq f_s \leq 2$. 
By Ohno's result from~\cite{Ohno80}, for a non-vanishing vector field~$X$ and a positive function~$f$ 
on a compact manifold, 
$$
h_{\top} (\phi_{fX}) \,\leq\, \max f \cdot h_{\top} (\phi_{X}).
$$
It follows that $h_{\top} (\phi_{\alpha_s}) \leq 2 \2 h_{\top} (\phi_{\partial_{t} + Y }) = :E$. 
\qed

\medskip \noindent
{\bf Step 2: A family of contact forms $\sigma_s$ on $\overline{\mathbb{D}} \times \partial \Sigma$.}
We have constructed a family of contact forms~$\alpha_s$ on the mapping torus~$\Sigma(\psi)$. 
We now wish to extend these forms to contact forms on~$M(\psi)$. 
For this let $V$ be the collar neighbourhood of $\partial \Sigma(\psi)$ defined by 
$$
V := [0,2\pi]\times N / \sim 
$$
where $ (2\pi,p) \sim  (0,p)$ for each $p \in \Sigma$. On $V$ the contact form~$\alpha_s$ reads
\begin{equation} \label{eq:alpha}
    \alpha_s \,=\, d \theta + s \2 (2-r) \2 dx
\end{equation}
where $(r,x)$ are the coordinates on~$N$ introduced above and $\theta \in S^1$.

We proceed to construct for each $s \in (0,s_1)$ a contact form $\sigma_s$ 
on~$\overline \D \times \partial \Sigma$, where $\overline{\D}$ is again the closed unit disk in~$\R^2$.
Consider polar coordinates $(\otheta, \ovr) \in S^1 \times (0,1]$ on $\overline{\D} \setminus \{0\}$ 
and the coordinate $\ox$ on~$\partial \Sigma$.
We can then consider coordinates $(\otheta, \ovr, \ox)$ 
on $\overline{\D} \setminus \{0\} \times\partial \Sigma$.
We pick a smooth function $f \colon (0,1] \to \R$ such that 
\begin{itemize}
\item $f'<0$,
\item $f(\ovr) = 2-\ovr$ on a neighbourhood of $1$,
\item $f(\ovr) = 2 -\ovr^4$ on a neighbourhood of $0$,
\end{itemize}

\smallskip \noindent
and we pick another smooth function $g \colon (0,1]\to \R$ satisfying
\begin{itemize}
\item $g'>0$ on $(0,1)$, 
\item $g(1) =1$ and all derivatives of $g$ vanish at $1$,
\item $g(\ovr) = \frac{\ovr^2}{2}$ on a neighbourhood of $0$.
\end{itemize}

\begin{figure}[h]   
 \begin{center}
  \psfrag{r}{$\ovr$}  \psfrag{1}{$1$}  \psfrag{2}{$2$}  \psfrag{f}{$f$}  \psfrag{g}{$g$}
  \leavevmode\includegraphics{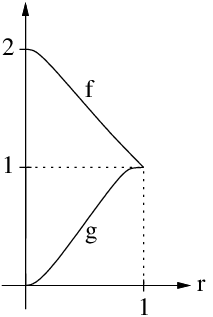}
 \end{center}
 \caption{The functions $f$ and $g$}   \label{fg.fig}
\end{figure}

Define the $1$-form 
\begin{equation} \label{e:la}
\sigma_s (\otheta, \ovr, \ox) \,=\, 
                 g(\ovr) \2 d\otheta + s f(\ovr) \2 d\ox 
\end{equation}
on $\overline{\D} \setminus \{0\} \times \partial \Sigma$.
Then 
\begin{equation} \label{e:dla}
\sigma_s \wedge d \sigma_s \,=\, 
s \2 h(\ovr) \2 d\ovr \wedge d \otheta \wedge d\ox
\end{equation}
where $h(\ovr) = (fg'-f'g)(\ovr)$.
It follows that $\sigma_s$ is a contact form on $\overline{\D} \setminus \{0\} \times \partial \Sigma$.
For $\ovr$ near~$0$ we have $h(\ovr) = \ovr (2+\ovr^4)$,
whence $\sigma_s$ extends to a smooth contact form on 
$\overline{\D} \times \partial \Sigma$, that we also denote by~$\sigma_s$.
The Reeb vector field of $\sigma_s$ is given by
\begin{equation} \label{e:Xle}
R_{\sigma_s} (\otheta, \ovr, \ox) \,=\, 
\frac{1}{h(\ovr)} \left( - f'(\ovr) \2 \partial_\otheta + \frac 1 s \2 g'(\ovr) \2 \partial_{\ox} \right).
\end{equation}
It follows that $R_{\sigma_s}$ is tangent to the tori $\T_{\ovr} := \{ \ovr = \text{const} \}$
and that for each $\ovr \in (0,1]$ the flow of $R_{\sigma_s}$ is linear:
\begin{equation} \label{e:linear}
\phi_{\sigma_s}^t (\otheta, \ovr, \ox) \,=\, 
\left( \otheta - \frac{f'(\ovr)}{h(\ovr)} \2 t, \, \ovr, \, \ox + \frac{g'(\ovr)}{s \2 h(\ovr)} \2 t \right) .
\end{equation}
In particular, using our choices of $f$ and $g$ we see that $R_{\sigma_s} = \partial_{\otheta}$
on the boundary torus~$\T_1$,
and that $R_{\sigma_s} = \frac{1}{2 s} \partial_{\ox}$ is tangent along the core circle $C= \{ \ovr=0 \}$ 
of the full torus, and gives the positive orientation to $\partial \Sigma$.
Furthermore, \eqref{e:Xle} shows that $R_{\sigma_s}$ is positively transverse to the half-open annuli 
$$
A_{\otheta} := \{\otheta\} \times (0,1] \times \partial \Sigma \,\subset\, \overline{\D} \setminus \{0\} \times \partial \Sigma.
$$

\begin{figure}[h]   
 \begin{center}
  \psfrag{C}{$C$} \psfrag{2e}{$\frac{1}{2 s} \partial_{\ox}$} \psfrag{dt}{$\partial_\otheta$} \psfrag{A}{$A_{\bm 0}$}
  \leavevmode\includegraphics{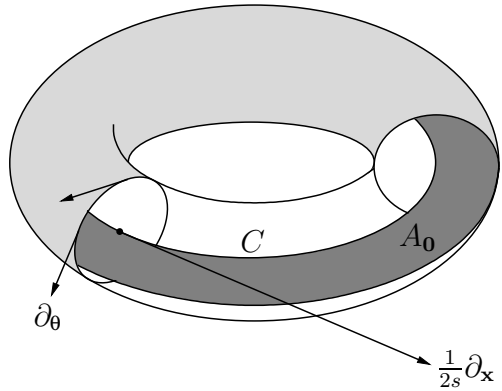}
 \end{center}
 \caption{Some vectors of $R_{\sigma_s}$ and the annulus $A_{\bm 0}$}   \label{XA.fig}
\end{figure}

The Reeb flow $\phi_{\sigma_s}^t$ on the full torus $\overline \D \times \partial \Sigma$ is integrable. 
More precisely, the core circle~$C$ of~$\overline \D \times \partial \Sigma $ 
is the trace of a periodic orbit of $\phi_{\sigma_s}^t$,  
and $(\overline \D \times \partial \Sigma ) \setminus C$ is foliated by the flow-invariant tori~$\T_\ovr$,
on which $\phi_{\sigma_s}^t$ is the linear flow~\eqref{e:linear}.
The topological entropy of these linear flows of course vanishes.
By the variational principle for topological entropy we therefore find that
\begin{equation} \label{e:h0}
h_{\top}(\phi_{\sigma_s}) \,=\, 
       \sup_{0 \leq r \le 1} h_{\top}(\phi_{\sigma_s} |_{\T_{\ovr}}) \,=\, 0.
\end{equation}

\medskip \noindent
{\bf Step 3: A family of contact forms on $M$.}
We first observe that the coordinates $\theta$ and $\otheta$, $r$ and $\ovr$, $x$ and $\ox$ are glued 
via the identification map used to glue $\Sigma(\psi)$ and $\overline \D \times \partial \Sigma$. 
It follows that they extend to coordinates on 
$$
V \cup_{\id} \left( \overline \D \times \partial \Sigma \right) .
$$
In view of the expressions \eqref{eq:alpha} and \eqref{e:la} for the contact forms~$\alpha_s$ 
on~$V$, and $\sigma_s$ on $\overline \D \times \partial \Sigma$, 
these two contact forms are glued to a smooth contact form~$\tau_s$ on~$M(\psi)$. 

As mentioned above, the Reeb vector field $R_{\tau_s}$ is positively transverse to 
the surfaces $\{\theta\} \times \Sigma$ in~$\Sigma(\psi)$ and to the annuli~$A_{\otheta}$ in 
$\overline \D \times \partial \Sigma$. 
Also, $R_{\tau_s}$ is tangent to the core circle~$C$ of $\overline \D \times \partial \Sigma$,
giving the positive orientation.
It follows that the Reeb vector field of $\Psi_* \tau_s$ is positively transverse to 
the interior of the pages of the open book $(\Sigma, \psi,\Psi)$, 
and positively tangent to the binding of the open book. 
By Theorem~\ref{t:Giroux} the contact structure $\ker \Psi_* \tau_s$ 
is diffeomorphic to~$\xi$, 
by some diffeomorphism $\rho_s \colon M \to M$.
Summarizing, there are diffeomorphisms $\Psi$ and~$\rho_s$ such that
\begin{equation} \label{e:MPP}
\left( M(\psi), \tau_s \right)  \stackrel{\Psi}{\longrightarrow} 
\left( M, \Psi_* \tau_s \right)  \stackrel{\rho_s}{\longrightarrow} 
\left( M, (\rho_s \circ \Psi)_*\tau_s \right)  
\end{equation}
with $\ker ( (\rho_s \circ \Psi)_* \tau_s) = \xi$.

\medskip \noindent
{\bf Step 4: Estimating the volume and the topological entropy of $\tau_s$.}
The Reeb flow $\phi_{\tau_s}^t$ of $\tau_s$ leaves the compact sets~$\Sigma(\psi)$ 
and~$\overline{\D} \times \partial \Sigma$ invariant. 
Since these compact sets cover~$M$ and since 
$\phi_{\tau_s}^t |_{\Sigma(\psi)} = \phi_{\alpha_s}^t$ and
$\phi_{\tau_s}^t |_{\overline{\D} \times \partial \Sigma} = \phi_{\sigma_s}^t$,
it follows from \cite[ Proposition~3.1.7 (2)]{HK95} 
and from Lemma~\ref{returntimebounded} and \eqref{e:h0} that 
\begin{equation} \label{entropypartition}
h_{\top}(\phi_{\tau_s}) = 
 \max \left\{ h_{\top}(\phi_{\alpha_s}) , h_{\top}(\phi_{\sigma_s}) \right\} \,\leq\, E.
\end{equation}

We decompose the integral of $\tau_s \wedge d \tau_s$ as
\begin{equation} \label{e:vol}
\int_{M(\psi)} \tau_s \wedge d \tau_s \,=\,
\int_{\Sigma(\psi)} \alpha_s \wedge d \alpha_s + \int_{\overline{\D} \times \partial \Sigma} \sigma_s \wedge d \sigma_s .
\end{equation}

For the first summand we have by \eqref{e:s^2} that
\begin{equation} \label{e:vol1}
\int_{\Sigma(\psi)} \alpha_s \wedge d \alpha_s \,=\, 
\int_{[0,2\pi] \times \Sigma} s \,d\theta \wedge \omega + O(s^2) 
\,=\, s \, 2\pi \int_\Sigma \omega + O(s^2) \,=\, O(s) .
\end{equation}

To estimate the second term in~\eqref{e:vol} we use~\eqref{e:dla}:
\begin{equation}
\int_{\overline{\D} \times \partial \Sigma} \sigma_s \wedge d \sigma_s \,=\, 
\int_{\overline{\D} \setminus \{0\} \times \partial \Sigma} \sigma_s \wedge d \sigma_s \,=\, 
s \int_{\overline{\D} \setminus \{0\} \times \partial \Sigma} h(\ovr) \,d\ovr \wedge d\otheta \wedge d\ox .
\end{equation}
Since the right integral is finite, it follows that also
\[
\int_{\overline{\D} \times \partial \Sigma} \sigma_s \wedge d \sigma_s = O(s).
\]
Together with \eqref{e:vol} and \eqref{e:vol1} we conclude that
\begin{equation} \label{e:sigma}
\int_{M(\psi)} \tau_s \wedge d \tau_s\,=\, O(s).
\end{equation}

\medskip \noindent
{\bf Step 5: End of proof.}
By \eqref{e:sigma} we know that given $\gve >0$ there exists $s \in (0,s_1]$ 
such that
\begin{equation} \label{volest1}
    \vol_{\tau_s} \bigl( M(\psi) \bigr)  \,=\, 
		\frac{1}{2\pi}\, \int_{M(\psi)} \tau_s \wedge d \tau_s  \,\leq\, 
		\frac{\gve^2}{E^2}.
\end{equation} 
Defining $\widetilde{\tau} := \left( \vol_{\tau_s} \bigl( M(\psi) \bigr) \right)^{-\frac{1}{2}} \, \tau_s$ 
we obtain 
$$
\vol_{\widetilde{\tau}} \bigl( M(\psi) \bigr)  \,=\, 1,
$$
and by \eqref{entropypartition} and \eqref{volest1} 
$$
h_{\top }(\phi_{\widetilde{\tau}}) 
\,=\, \left( \vol_{\tau_s} \bigl( M(\psi) \bigr) \right)^{1/2} \, h_{\top} (\phi_{\tau_s}) \,\leq\, 
\gve .
$$
Together with \eqref{e:MPP} and in view of the conjugacy invariance of topological entropy
if follows that $(\rho_s \circ \Psi)_* \widetilde \tau$ is a contact form on~$(M,\xi)$ 
of volume~$1$ and topological entropy~at most~$\gve$. 
\qed

\begin{question}
{\rm
It would be interesting to see how in the case of a spherization~$S^*Q_2$ 
over the closed orientable surface of genus~$2$ 
the open book decomposition used in the above proof looks like. 
Since our construction of the contact form~$\widetilde \tau$ is explicit, 
one could then maybe understand the star field $\{ D_q^*(H) \}$ 
corresponding to~$\widetilde \tau$.
In view of~\eqref{e:hHs} some of the stars must be very spiky.
How many spikes appear in these stars? }
\end{question}

%%%%%%%%%%%%%%%%%%%%%%%%%%%%%%%%%%%%%%%%

\section{Generalities on Giroux's correspondence in higher dimensions} 
\label{s:Giroux}
In this section, we summarize those concepts and results on the Giroux correspondence 
between contact structures and supporting open books in higher dimensions that we shall
use in the proof of Theorem~\ref{t:mainintro}. 
While we reprove those parts that we use in a somewhat different form, 
we refer to~\cite{Gir02} and~\cite{Gir17} for the parts that we can cite and for further results.

\subsection{Ideal Liouville domains}
Let $F$ be a $2n$-dimensional compact manifold with non-empty boundary~$K$, 
and denote by~$F^\circ$ the interior of~$F$. 
A symplectic form $\omega$ on~$F^\circ$ is called an \emph{ideal Liouville structure}\/ on~$F$ 
(abbreviated~ILS)
if $\omega$ admits a primitive $\lambda$ on~$F^\circ$ such that for some (and then any) smooth function
\begin{equation} \label{prop_of_u}
u \colon F \rightarrow [0,+\infty) \; \mbox{ for which } K=u^{-1}(0)  \mbox{ is a regular level set} 
\end{equation}
the 1-form $u \lambda$ on $F^\circ$ extends to a smooth 1-form~$\beta$ on~$F$ which is a contact form along~$K$. 

If such a 2-form $\omega$ exists, the pair $(F,\omega)$ is called an \emph{ideal Liouville domain}\/ (ILD), 
and any primitive~$\lambda$ with the above property is called an \emph{ideal Liouville form}\/ (ILF). 
Given an ILD $(F,\omega)$, the contact structure 
$$
\xi := \ker (\beta|_{TK})
$$
depends on the 2-form $\omega$ but neither on~$\lambda$ nor~$u$, see Proposition~2 in~\cite{Gir17}. 
Moreover, once $\lambda$ is chosen, one can recover every (positive) contact form on~$(K,\xi)$ 
as the restriction to~$K$ of the extension of $u\lambda$ for some function~$u$ 
with property~\eqref{prop_of_u}. 
This is why the pair $(K,\xi)$ is called the \emph{ideal contact boundary} of~$(F,\omega)$. 
We note that the orientation of~$K$ that is determined by the co-oriented contact structure~$\xi$ 
coincides with the orientation of~$K$ as the boundary of~$(F,\omega)$. 

A very useful feature of an ILD is that a neighborhood of its boundary admits an explicit parametrization 
in which any ILF has a very nice form.  

\begin{lem} \label{le:nearK} 
Let $(F,\omega)$ be an ILD and $\lambda$ be an ILF. Let $u$ be a function satisfying~\eqref{prop_of_u} 
and let $\beta$ be the extension of~$u\lambda$. Then for any  contact form $\alpha_0$ on~$(K,\xi)$, 
there exists an embedding 
\begin{eqnarray*}
\imath \colon [0,+\infty) \times K \rightarrow F
\end{eqnarray*}
such that 
$$
\imath^*\lambda=\frac{1}{r}\alpha_0 \: \textrm{ and } \: \imath(0,q)=q \: \textrm{ for all } \: q\in K ,
$$ 
where $r \in [0,+\infty)$. In particular, 
$$
\imath^*\beta = \frac{u \circ\imath}{r}\alpha_0 \: \textrm{ on }\: F^\circ$$ and for all $q\in K$,
$$
(\beta |_{TK})(q) = \left( \frac{\partial (u \circ \imath)}{\partial r}(0,q) \right) \alpha_0.
$$
\end{lem} 

\begin{proof} 
The above statement is a reformulation of Proposition~3 in~\cite{Gir17}. 
We give a similar but more explicit proof. 

Let $\dim F=2n$. Since $\beta$ is by assumption a positive contact form on~$K$, 
$\beta \wedge (d\beta)^{n-1}$ is a positive volume form on~$K$.
Using $\omega = d \lambda = d(\beta / u)$ on~$F^\circ$ we compute
\begin{equation} \label{e:obu}
\omega^n = \left( d(\beta/u) \right)^n = u^{-n-1} (u \, d\beta + n \beta\wedge du) \wedge (d\beta)^{n-1} = u^{-n-1}\mu
\end{equation}
where $\mu$ is the $2n$-form
$$
\mu := (u \, d\beta + n\beta \wedge du) \wedge (d\beta)^{n-1}.
$$
The above expression shows that $\mu$ is smooth on $F$ and, 
together with~\eqref{e:obu} and the fact that $0$ is a regular value of~$u$, 
that it is a positive volume form on~$F$.
Define the smooth vector field~$X$ on~$F$ by
\begin{equation} \label{e:Xmu}
\iota_X \mu = -n \beta \wedge (d\beta)^{n-1} .
\end{equation}

Recall that the Liouville vector field~$Y$ of~$\lambda$ is the vector field on~$F^\circ$
defined by $\iota_Y d\lambda = \lambda$.
Using $\beta = u \lambda$ on~$F^\circ$ we compute
$$
-n\beta \wedge (d\beta)^{n-1} = -n \2 u^n \2 \lambda \wedge (d\lambda)^{n-1} = -u^n \2 \iota_Y \omega^n = 
-u^{-1} \2 \iota_Y \mu.
$$
Comparing with~\eqref{e:Xmu} we find $Y = -u X$. Then
\[
\beta(X) = - \lambda(Y) = - d\lambda (Y,Y) = 0
\]
on $F^{\circ}$, and by continuity $\beta(X)=0$ on $F$. Hence on $F^\circ$,
\begin{equation} \label{e:Lb}
L_X \beta = \iota_X d\beta = -\frac{1}{u} \2 \iota_Y \left(du \wedge \lambda + u d\lambda \right) = 
      -\frac{1}{u} \left( du(Y) \2 \lambda + u \lambda \right) = \frac{1}{u} \left( du(X)-1\right) \beta .
\end{equation}
This shows that $du (X) =1$ along $K$ and that the function $\frac 1u (du(X)-1)$ is smooth on~$F$.
In particular, $X$ points inwards on $K=\partial F$ and hence the flow $\phi_X^t$ of~$X$ 
on the compact manifold~$F$ is well-defined for every $t\geq 0$. 
We define the smooth embedding 
$$
\Phi \colon [0,+\infty) \times K \rightarrow F, \quad (t,q) \mapsto \phi^t_{X}(q) .
$$ 
By construction we have $\Phi^*X=\partial_t$. Put $\hat{\beta} := \Phi^*\beta$, $\hat{u} := \Phi^*u$, 
and $\hat{\lambda} := \Phi^*\lambda$. 
The identities $\beta (X)=0$ and~\eqref{e:Lb} say that on $[0,+\infty) \times K$, 
\begin{equation} \label{e:bdot}
\hat{\beta}(\partial_t)=0, \quad \partial_t \1 \hat{\beta} = \frac{\partial_t \hat u -1}{\hat u} \2 \hat \beta .
\end{equation}
Here the function $\frac{\partial_t \hat u -1}{\hat u}$ is smooth and bounded 
on $[0,+\infty) \times K$ since by \eqref{e:Lb} the function $\frac 1u (du(X)-1)$ is smooth on~$F$.
Define the smooth function $v_0 \colon [0,+\infty) \times K \to \R$ by
$$
v_0(t,q) \,=\, \int_0^t \frac{\partial_t \hat{u} (\tau,q)-1}{\hat{u}(\tau,q)} \,d\tau .
$$
The solution of the problem \eqref{e:bdot} with initial condition $\beta_0(q) = \beta (0,q)$ is then
$$
\hat{\beta}(t,q) = \exp \left( v_0(t,q) \right) \beta_0(q), \qquad \forall \, (t,q)\in [0,+\infty) \times K,
$$
and therefore
$$
\hat{\lambda}(t,q) \,=\, 
\frac{1}{\hat{u}(t,q)} \exp \left( v_0(t,q) \right) \beta_0(q), \qquad \forall \, (t,q)\in (0,+\infty) \times K .
$$
Now let $\alpha_0$ be a positive contact form on $(K,\xi)$. 
Then there is a positive function~$\kappa$ on~$K$ 
such that $\beta_0 = \kappa \2 \alpha_0$.
On $(0,+\infty)\times K$ define the function  
\begin{equation} \label{e:Lambda}
\Lambda(t,q) \,=\, 
  \frac{\kappa(q)}{\hat{u}(t,q)} \exp \left( v_0(t,q) \right).
\end{equation}
Then $\hat{\lambda} = \Lambda \2 \alpha_0$. It is clear that $\Lambda >0$, 
and $\lim_{t \rightarrow 0} \Lambda(t,q) = +\infty$ for all $q \in K$. 
We note that 
$$
\frac{\partial \Lambda}{\partial t} =-\frac{\Lambda}{\hat{u}}<0
$$
and therefore
\begin{equation} \label{e:f}
\Lambda(t,q) \,=\, \Lambda(1,q) \exp \left( -\int_{1}^t\frac{1}{\hat{u}(\tau,q)} \,d\tau \right) .
\end{equation}
On $[0,+\infty) \times K$, $\hat{u}$ is bounded from above since $F$ is compact. 
Therefore $\lim_{t \rightarrow +\infty} \Lambda(t,q)=0$ for all $q \in K$. 
It follows that $\Lambda(\cdot,q)$ is an orientation reversing diffeomorphism from 
$(0,+\infty)$ onto $(0,+\infty)$ for all~$q$. 
Hence there exists a positive smooth function~$f$ on $(0,+\infty) \times K$ such that 
$$
\Lambda(f(r,q),q) \,=\, \frac{1}{r} \quad \forall \, (r,q) \in (0,+\infty) \times K,
$$
and for every $q\in K$ the function $f(\cdot,q)$ is an orientation preserving diffeomorphism  from $(0,+\infty)$ onto $(0,+\infty)$.
Define the embedding
$$
\Psi \colon (0,+\infty) \times K \rightarrow [0,+\infty) \times K, \quad (r,q) \mapsto (f(r,q),q) .
$$
By construction $\Psi^*\hat{\lambda}=\frac{1}{r}\alpha_0$. 
We claim that $\Psi$ extends to a smooth embedding 
$$
\Psi \colon [0,+\infty) \times K \rightarrow [0,+\infty) \times K \quad \mbox{ with $\Psi(0,q)=(0,q)$.}
$$ 
Postponing the proof of the claim, we note that $\imath = \Phi \circ \Psi$ is then the desired embedding. 
The rest of the statement of the lemma follows immediately from the identity $\imath^*\lambda = \frac 1r \2 \alpha_0$.

We now show that the extension of $\Psi$ given by $\Psi (0,q) = (0,q)$ is smooth.
Combining \eqref{e:Lambda} and~\eqref{e:f} we get
\begin{equation}\label{e:formulaforf}
\Lambda(1,q)^{-1}\exp \left( \int_{1}^{f(r,q)}\frac{1}{\hat{u}(\tau,q)} \,d\tau \right)=r.    
\end{equation}
We consider the function 
$$
g(t,q) \,:=\, \Lambda(1,q)^{-1} \exp \left( \int_{1}^{t} \frac{1}{\hat{u}(\tau,q)} \,d\tau \right)
$$
on $(0,+\infty) \times K$ and we define 
$$
\widetilde{\Psi} \colon (0,+\infty) \times K \rightarrow (0,+\infty) \times K, \quad \widetilde{\Psi}(t,q)=(g(t,q),q).$$
Then we have
$$
\widetilde{\Psi} \circ \Psi(r,q) \,=\, \widetilde{\Psi} (f(r,q),q) \,=\, \bigl( g(f(r,q),q),q \bigr) \,=\, (r,q)
$$
on $(0,+\infty)\times K$.
We claim that $\widetilde{\Psi}$ extends smoothly to $[0,+\infty)\times K$ by $\widetilde{\Psi}(0,q)=q$ 
for all $q\in K$. To see this we define the function $v_1 \colon [0,+\infty) \times K \to \R$ 
by  
$$
v_1(t,q) \,:=\, \int_1^t \frac{\partial_t \hat u(\tau,q)-1}{\hat u(\tau,q)} \,d\tau.
$$
We recall that the above integrand is smooth and bounded on $[0,+\infty) \times K$ and so is the function~$v_1$. 
For every $t \in (0,+\infty)$ and $q \in K$,
\begin{eqnarray*}
e^{v_1} \,=\, \exp \left( \int_1^t \frac{\partial_t \hat u-1}{\hat u} \, d\tau\right)
        &=& \exp \left(\log \hat u(t,q)-\log \hat u(1,q) -\int_1^t \frac{1}{\hat u} \, d\tau \right) \\
&=& \frac{\hat u(t,q)}{\hat u(1,q)} \exp \left( -\int_1^t \frac{1}{\hat u} \, d\tau \right)
\end{eqnarray*}
and so 
$$
g(t,q) \,=\, \Lambda(1,q)^{-1} \, e^{-v_1(t,q)} \, \frac{\hat u(t,q)}{\hat u(1,q)}.
$$
Note that $\Lambda(1,q)\neq 0$. The above expression says that $g$ extends smoothly to $[0,+\infty) \times K$ 
by $g(0,q)=0$ for $q \in K$. For $t>0$ we compute 
\begin{eqnarray*}
\partial_t g (t,q) &=& \Lambda(1,q)^{-1} \,e^{-v_1(t,q)} \left[- \partial_t v_1 (t,q) \, 
           \frac{\hat u(t,q)}{\hat u(1,q)} + \frac{\partial_t \hat u (t,q)}{\hat u(1,q)} \right] \\
&=&\Lambda(1,q)^{-1} \, e^{-v_1(t,q)} \left[-\frac{(\partial_t \hat u(t,q)-1)}{\hat u(t,q)} 
           \frac{\hat u(t,q)}{\hat u(1,q)} + \frac{\partial_t \hat u(t,q)}{\hat u(1,q)} \right] \\
&=&\Lambda(1,q)^{-1} \, e^{-v_1(t,q)} \,\frac{1}{\hat u(1,q)} .
\end{eqnarray*}
By the smoothness of $g$, this expression also holds true for $t=0$.
In particular, $\partial_t g(t,q) >0$ for all $(t,q) \in [0,+\infty) \times K$.
It follows that $D \widetilde{\Psi}(t,q)$ is invertible for all $(t,q) \in [0,+\infty) \times K$. 
By the inverse function theorem, the extension of $\Psi$ over $[0,+\infty)\times K$ is~$C^1$ 
and in fact $C^\infty$-smooth since $\widetilde{\Psi}$ is smooth. 
\end{proof}

\subsection{Ideal Liouville domains and contact structures}
Ideal Liouville domains are particularly useful for clarifying the existence and uniqueness of contact structures supported by open books in higher dimensions. We first recollect some facts about open books. 

An \textit{open book}\/ in a closed manifold~$M$ is a pair~$(K,\Theta)$ where
\begin{enumerate}[label=\mbox{(ob\arabic*)}]
\item \label{ob1}
$K\subset M$ is a closed submanifold of co-dimension two with trivial normal bundle;

\smallskip
\item \label{ob2}
$\Theta \colon M \setminus K \rightarrow S^1= \R /2 \pi \Z$ is a locally trivial smooth fibration 
that on a deleted neighbourhood~$(\D \setminus \{0\}) \times K$ of~$K$ reads
$\Theta (re^{i\theta},q) = \theta$. 
\end{enumerate}  
The submanifold $K$ is called the \textit{binding}\/ of the open book, and the closures of the
fibres of~$\Theta$ are called the~\textit{pages}. 
The pages are compact submanifolds with common boundary~$K$.  
The canonical orientation of~$S^1$ induces co-orientations of the pages. 
Hence if $M$ is oriented, then so are the pages, 
and then also the binding as the boundary of a page. 

Another way of defining an open book is as follows. Let $h \colon M \rightarrow \mathbb{C}$ 
be a smooth function such that $0$ is a regular value. 
Set $K := h^{-1}(0)$, and assume that
$\Theta :=h/|h| \colon M \setminus K \rightarrow S^1$ has no critical points.
Then the pair $(K,\Theta)$ is an open book in $M$. Moreover, any open book in~$M$ can be recovered 
via a \emph{defining function}~$h$ as above, and such a defining function is unique up to multiplication 
by a positive function on~$M$.

Given an open book $(K,\Theta)$ in a closed manifold~$M$, one finds a vector field~$X$ on~$M$, 
called a \emph{spinning vector field}, such that 
\begin{enumerate}[label=\mbox{(m\arabic*)}]
\item \label{m1}
$X$ lifts to a smooth vector field on the manifold with boundary obtained from~$M$ 
by real oriented blow-up along~$K$, in which each disk $\D \times \{q\}$ 
of the neighbourhood $\D \times K$ is replaced by the annulus $S^1 \times [0,1] \times \{q\}$; 

\smallskip
\item \label{m2} 
$X=0$ on $K$ and $d\Theta (X) = 1$ on $M \setminus K$.
\end{enumerate}
Then the time-$2\pi$ map of the flow of~$X$ is a diffeomorphism 
$$
\phi \colon F\rightarrow F
$$ 
of the $0$-page $F := \Theta^{-1}(0)\cup K$, which fixes $K$ pointwise. 
The isotopy class $[\phi]$ of~$\phi$ among the diffeomorphisms of~$F$ that fix~$K$ pointwise
is called the \emph{monodromy}\/ of the open book. It turns out that the open book is characterized 
by the pair $(F,[\phi])$. Namely, given the pair $(F,\phi)$, one defines the mapping torus 
$$
\MT(F,\phi) := ([0,2\pi] \times F ) \big/\sim  \quad \mbox{ where } (2\pi,p) \sim (0,\phi(p)) .
$$
This is a manifold with boundary. One has the natural fibration 
$$
\widehat{\Theta} \colon \MT(F,\phi)\rightarrow S^1 
$$
with fibres diffeomorphic to $F$, and there is a natural parametrization of 
the fibre $\widehat{\Theta}^{-1}(0)$ via the restriction of the above quotient map to $\{0\} \times F$. 
For every $\phi'\in [\phi]$ there is a diffeomorphism between $\MT(F,\phi)$ and~$\MT(F,\phi')$ 
that respects the fibrations over~$S^1$ and the natural parametrizations of the $0$-fibres. 
Now, given $\MT(F,\phi)$ one collapses its boundary, which is diffeomorphic to $S^1 \times K$, 
to~$K$ and obtains the so-called \emph{abstract open book} $\OB(F,\phi)$. 
In fact, the closed manifold~$\OB(F,\phi)$ admits an open book given by the pair $(K,\Theta)$, 
where $\Theta$ is induced by~$\widehat{\Theta}$. 
Moreover, for $\phi' \in [\phi]$, 
the diffeomorphism between $\MT(F,\phi)$ and~$\MT(F,\phi')$ descends to a diffeomorphism between the corresponding abstract open books. In particular, $M$ and $\OB(F,\phi)$ may be identified 
together with their open book structures. 
We note that one can choose the spinning vector field~$X$ smooth on~$M$
and such that its flow is $2\pi$-periodic near~$K$. 
However, not every representative of the monodromy class can be obtained 
via a smooth spinning vector field, see Remark~12 in~\cite{Gir17}. 
To obtain all representatives of the monodromy class, 
one needs to use the whole affine space of spinning vector fields. 

Open books meet with contact topology via the following definition. 
Let $M$ be a compact manifold with a co-oriented contact structure~$\xi$. 
We say that $\xi$ \emph{is supported}\/ by an open book~$(K,\Theta)$ on~$M$
or that the open book $(K,\Theta)$ \emph{supports}~$\xi$ 
if there exists a contact form~$\alpha$ on~$(M,\xi)$, that is $\xi=\ker \alpha$, 
such that 
\begin{itemize}
\item $\alpha$ restricts to a positive contact form on $K$;
\item $d\alpha$ restricts to a positive symplectic form on each fibre of $\Theta$.
\end{itemize}
It turns out that given a closed manifold $M$, the isotopy classes of co-oriented contact structures on~$M$ are in one-to-one correspondence with equivalence classes of supporting open books. This statement is 
a very rough summary of what is called the Giroux correspondence. We will recall certain pieces of this celebrated statement in detail.

\begin{theorem}[Theorem~10 in~\cite{Gir02}] \label{thm10} 
Any contact structure on a compact manifold is supported by an open book with Weinstein pages.   
\end{theorem}

This result is the core of the correspondence between supporting open books and contact structures. 
The existence statement of the opposite direction of the correspondence
is relatively easy to achieve, especially in dimension three. 
Namely, given an open book in a 3-dimensional compact manifold, it is not hard to construct a contact form on the corresponding abstract open book whose kernel is supported. 
In higher dimensions, however, one needs that the pages are exact symplectic and that the monodromy is symplectic 
in order to construct a contact form on an abstract open book whose kernel is supported,  
see Proposition~9 in~\cite{Gir02} and Proposition~17 in \cite{Gir17}. 
We will carry out such a construction in Sections~\ref{s:cfaway} and~\ref{s:cfnear}. 
Concerning the uniqueness features of the Giroux correspondence, we are mainly interested in one side, 
namely the ``uniqueness'' of supported contact structures. 
This result is again more involved in higher dimensions. 
Heuristically, given an open book, the symplectic geometry of the pages determines the supported contact structures. In dimension three, any two symplectic structures on a page are isotopic since they are just 
area forms on a given surface, but in higher dimensions this is not the case.  

In \cite{Gir17} Giroux introduced the notion of a Liouville open book, which clears out the technicalities
to which we pointed above. 

A \emph{Liouville open book} (LOB) in a closed manifold $M$ is a triple $(K,\Theta,(\omega_\theta)_{\theta\in S^1})$ where
\begin{enumerate}[label=\mbox{(lob\arabic*)}]
\item \label{lob1}
$(K,\Theta)$ is an open book on $M$ with pages $F_\theta=\Theta^{-1}(\theta)\cup K$,
$\theta\in S^1$;

\smallskip
\item \label{lob2}
$(F_\theta,\omega_\theta)$ is an ILD for all $\theta\in S^1$ and the following holds:  there is a defining function $h \colon M \rightarrow \mathbb{C}$ for $(K,\Theta)$ and a $1$-form $\beta$ on~$M$ such that the restriction of $d(\beta/|h|)$ to each page is an~ILF. More precisely, 
$$
\omega_\theta=d(\beta/|h|) |_{TF_\theta^\circ}
$$
for all $\theta \in S^1$.
\end{enumerate} 
The 1-form $\beta$ is called a \emph{binding 1-form} associated to~$h$. 
If $h'$ is another defining function for $(K,\Theta)$, then  $h'=\kappa \1 h$ 
for a positive function~$\kappa$ on~$M$, and $\beta' := \kappa \1 \beta$ is a binding 1-form associated 
to~$h'$. We also note that for a fixed defining function, the set of associated binding 1-forms is 
an affine space. 
                       
Similar to the case of classical open books, LOBs are characterized by their monodromy, which now 
has to be symplectic: 
One considers \emph{symplectically spinning vector fields}, namely vector fields~$X$ 
satisfying~\ref{m1} and~\ref{m2} and generating the kernel of a closed 2-form on $M \setminus K$ 
which restricts to~$\omega_\theta$ for all $\theta \in S^1$. 
Given such a vector field, the time-$2\pi$ map of its flow, 
say~$\phi$, is a diffeomorphism of $F := F_0$ which fixes~$K$ and preserves $\omega := \omega_0$. 
The isotopy class $[\phi]$, among the symplectic diffeomorphisms that fix~$K$, is called the 
\emph{symplectic monodromy}\/ and characterizes the given~LOB. 
For the construction of a LOB in the abstract open book $\OB(F,\phi)$, where $\phi^*\omega=\omega$, 
we refer to Proposition~17 in~\cite{Gir17} and to our construction at the end of the next section. 

Again, symplectically spinning vector fields form an affine space, 
and all representatives of the symplectic monodromy can be obtained by sweeping out this affine space. 
The obvious choice of a symplectically spinning vector field~$X$ is smooth, 
and by suitably modifying a given binding 1-form 
without affecting its restriction to the kernel of~$d\Theta$
one can arrange that the flow of~$X$ is $2\pi$-periodic near the binding. 

\begin{lem}[Lemma 15 in \cite{Gir17}]  \label{lem15}
Let $(K, \Theta,(\omega_\theta)_{\theta \in S^1})$ be a LOB in a closed manifold~$M$, and let 
$h \colon M \rightarrow \mathbb{C}$ be a defining function for~$(K,\Theta)$. 
Then for every binding 1-form $\beta$, the vector field~$X$ on~$M \setminus K$ spanning the kernel 
of $d(\beta/|h|)$ and satisfying $d \Theta (X) = 1$ extends to a smooth vector field on~$M$ 
which is zero along~$K$.
Furthermore, $\beta$ can be chosen such that the flow of~$X$ is $2\pi$-periodic near~$K$.
\end{lem}

Natural sources of LOBs are contact manifolds: 

\begin{prop}[Proposition 18 in \cite{Gir17}]  \label{prop18}
Let $(M, \xi)$ be a closed contact manifold, and let $(K, \Theta)$ be a supporting open book with defining function $h \colon M \rightarrow \mathbb{C}$.
Then the contact forms $\alpha$ on~$(M,\xi)$ such that $d(\alpha/|h|)$ induces an ideal Liouville structure
on each page form a non-empty convex cone.
\end{prop}

Let $(K,\Theta,(\omega_\theta)_{\theta\in S^1})$ be a LOB in a closed manifold $M$ with a defining function~$h$. 
A co-oriented contact structure~$\xi$ on~$M$ is said to be \emph{symplectically supported}\, by 
$(K,\Theta,(\omega_\theta)_{\theta \in S^1})$ if there exists a contact form~$\alpha$ on~$(M,\xi)$ 
such that $\alpha$ is a binding 1-form of the LOB associated to~$h$. 

By our remark following the definition of a binding 1-form, the property of being symplectically supported 
is independent of the given defining function. But the crucial fact is that once a defining function is fixed, a contact binding 1-form is unique whenever it exists, see Remark~20 in~\cite{Gir17}. 
Hence, once a defining function~$h$ is fixed, there is a one-to-one correspondence between contact structures supported by $(K,\Theta,(\omega_\theta)_{\theta\in S^1})$ and contact binding 1-forms associated to~$h$.

Given two contact structures $\xi_0$ and~$\xi_1$ supported by $(K,\Theta,(\omega_\theta)_{\theta\in S^1})$, 
after fixing $h$ we therefore have unique contact binding 1-forms $\alpha_0$ and~$\alpha_1$, respectively. Since the set of binding 1-forms associated to~$h$ is affine, there is a path $(\beta_t)_{t\in [0,1]}$ of binding 1-forms such that $\beta_0 = \alpha_0$ and $\beta_1 = \alpha_1$. 
Now it is not hard to explicitely deform the forms $\beta_t$ 
without affecting their restrictions to~$\ker d\Theta$
in such a way that

\begin{itemize}
\item
for all $s \geq 0$ and $t \in [0,1]$, $\beta^s_t$ is a binding 1-form for $(K,\Theta,(\omega_\theta)_{\theta\in S^1})$ associated to~$h$ 
(since the deformation of $\beta_t$ leaves unchanged the restriction to the pages);

\smallskip
\item
$\beta_t^s$ is a contact form for $s$ large enough, uniformly in $t \in [0,1]$;

\smallskip
\item 
if $\beta_t$ is already a contact form, then $\beta_t^s$ is a contact form for all $s \geq 0$.
\end{itemize}  

\noindent
By the first property of these deformations and by the uniqueness discussed above, 
whenever $\beta^s_t$ is a contact form then $\ker \beta^s_t$ is symplectically supported by 
$\left(K,\Theta,(\omega_\theta)_{\theta\in S^1}\right)$ and $\beta_t^s$ is the unique contact binding 1-form 
associated to~$h$. 
Together with the other two properties we see that there exists $c>0$ such that the concatenation 
of the paths $(\ker \beta_0^s)_{s\in[0,c]}$, $(\ker \beta^c_t)_{t\in [0,1]}$ and 
$(\ker \beta^{c-s}_1)_{s\in [0,c]}$ gives an isotopy between $\xi_0$ and~$\xi_1$ along contact structures that are symplectically supported by $(K,\Theta,(\omega_\theta)_{\theta\in S^1})$.

\begin{figure}[h]   
 \begin{center}
  \psfrag{t}{$t$} \psfrag{s}{$s$} \psfrag{a0}{$\alpha_0$} \psfrag{a1}{$\alpha_1$} \psfrag{c}{$c$}   
  \psfrag{bt}{$\beta_t$} \psfrag{btc}{$\beta_t^c$} \psfrag{b0s}{$\beta_0^s$} \psfrag{b1c}{$\beta_1^{c-s}$} 
  \leavevmode\includegraphics{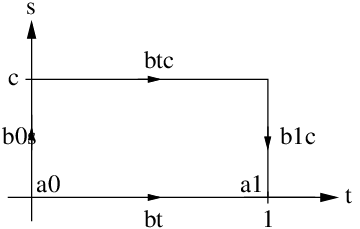}
 \end{center}
 \caption{}   
\label{betas.fig}
\end{figure}

In fact the following more general statement holds. 

\begin{prop}[Proposition 21 in \cite{Gir17}]  \label{prop21} 
On a closed manifold, contact structures supported by a given Liouville open book
form a non-empty and weakly contractible subset in the space of all contact structures.
\end{prop}

\section{Entropy collapse for Reeb flows in dimension $\geq 3$} \label{s:collapsegeq3}

This section is devoted to the proof of the following main result.

\begin{theorem} \label{t:mainn}
Let $(M,\xi)$ be a closed co-orientable contact manifold.
Then for every real number $\gve >0$ there exists 
a contact form~$\alpha$ for~$\xi$
such that $\widehat \Gamma(\alpha) \leq \gve$.
\end{theorem}

We prove the statement by induction on~$n$, where $\dim M=2n+1$. 
The initial case $n=0$ is clear: Then $M$ is a circle, and 
the Reeb flow generated by the vector field~$\partial_\theta$ has vanishing norm growth.
It may be interesting to read the subsequent proof for~$n=0$
and to compare the line of argument with the one of Section~\ref{s:collapse3}. 

We now assume by induction that Theorem~\ref{t:mainn} holds for $n-1 \geq 0$
and fix a contact manifold~$(M,\xi)$ of dimension~$2n+1$. 
By Theorem~\ref{thm10} there exists an open book $(K,\Theta)$ in~$M$ supporting~$\xi$. 
Let $F_\theta := \Theta^{-1}(\theta) \cup K$, $\theta \in S^1=\R/2\pi \Z$, denote the pages of the open book 
and let $h \colon M \rightarrow \mathbb{C}$ be a defining function for~$(K,\Theta)$.
We wish to construct a contact form on the abstract open book defined via the $0$-page
\begin{eqnarray}\label{0thpage}
F := F_0 = \Theta^{-1}(0) \cup K.
\end{eqnarray}

By Proposition~\ref{prop18}, there exists a contact form~$\alpha$ on~$(M,\xi)$ such that 
$(K,\Theta, d(\alpha/|h|) |_{TF_\theta^\circ})$ is a~LOB which supports $\xi$ symplectically. 
By Lemma~\ref{lem15}, we can modify the contact binding 1-form~$\alpha$ without affecting its restriction to the kernel of~$d\Theta$, 
to obtain a binding 1-form~$\widehat \alpha$, not necessarily contact, 
such that the flow of the associated symplectically spinning vector field~$X$ 
is $2\pi$-periodic near~$K$. 
Hence the time-$2\pi$ map of the flow of~$X$ gives us a diffeomorphism 
$\psi \colon F \rightarrow F$ such that  
\begin{equation} \label{monodromy}
\psi^*(d\lambda) = d\lambda
\end{equation}
where $\lambda \in \Omega^1(F^\circ)$ is the ILF 
\begin{eqnarray} \label{lambda}
\lambda := (\widehat \alpha /|h|) |_{TF^\circ}=(\alpha/|h|) |_{TF^\circ}
\end{eqnarray}
and $\psi={\rm id}$ on some neighbourhood of $K$ in~$F$. 
Now our aim is to recover~$M$ as the abstract open book induced by 
the pair~$(F,\psi)$ and to define a contact form on the abstract open book 
with small norm growth. 
We first consider the mapping torus 
$$
\MT (F,\psi) := \bigl( [0,2\pi] \times F \bigr) \big/ \bigl( (2\pi,p)\sim (0,\psi(p)) \bigr).
$$ 
Since $\psi={\rm id}$ on some neighbourhood of $K$, the boundary $\partial \MT(F,\psi)$ has an open neighbourhood given as a product of~$K$ with an annulus, in which we collapse the boundary and get the abstract open book~$\OB(F,\psi)$. We postpone the precise collapsing procedure since it will involve choices 
of coordinates, but note that the abstract open book is independent of these choices, and that we can make
the identifications
$$
\MT(F^\circ,\psi) = \MT(F,\psi) \setminus \partial \MT(F,\psi) = \OB(F,\psi)\setminus K.
$$

\subsection{A family of contact forms away from the binding}
\label{s:cfaway} 
On $[0,2 \pi]\times F^\circ$ with $\theta$ the coordinate on $[0,2\pi]$, 
we define the family of 1-forms 
\begin{eqnarray}\label{alpha_s}
\alpha_s = d\theta +s \left(\lambda + \chi (\theta) \2 \lambda_\psi \right) , \qquad s > 0
\end{eqnarray}
where $\lambda_\psi := \psi^*\lambda-\lambda$ and $\chi \colon [0,2\pi] \rightarrow [0,1]$ is 
a smooth function such that $\chi (0)=0$, $\chi (2\pi)=1$ and $\chi'$ has support in~$(0,2\pi)$. 
By the choice of~$\chi$, each 1-form~$\alpha_s$ descends to a 1-form on~$\MT(F^\circ,\psi)$,
that we still denote by~$\alpha_s$.

\begin{lem}\label{contactaway} 
There exists $s_0>0$, depending on $\psi, \lambda, \chi$, such that 
$\alpha_s$ is a contact form on $\MT(F^\circ,\psi)$ for all $s \in (0,s_0]$.
\end{lem}

\begin{proof} 
Since $d\lambda_\psi=0$, we get $d\alpha_s = s \left( \chi' d\theta\wedge \lambda_\psi+d\lambda \right)$
and
\begin{eqnarray*}
\alpha_s \wedge (d\alpha_s)^{n} &=&
\bigl( d\theta+s \left( \lambda + \chi \1 \lambda_\psi \right) \bigr) \wedge s^n 
   \Bigl( n \1 \chi' d\theta \wedge \lambda_\psi \wedge (d\lambda)^{n-1}+(d\lambda)^n \Bigr) \\
&=&
s^n \2 d\theta \wedge \Bigl( (d\lambda)^n - n \1 s \chi' \lambda \wedge \lambda_\psi \wedge (d\lambda)^{n-1} \Bigr).
\end{eqnarray*}
Since $d\theta\wedge (d\lambda)^n$ is a volume form and since with $\lambda_\psi$ also 
$\lambda \wedge d\theta \wedge \lambda_\psi \wedge (d\lambda)^{n-1}$ is compactly supported in~$\MT(F^\circ,\psi)$,  there exists $s_0>0$ such that $\alpha_s \wedge (d\alpha_s)^{n}$
is a positive volume form for all $s \in(0,s_0]$.
\end{proof}

Next, we study the Reeb vector field $R_{\alpha_s}$ of~$\alpha_s$ on~$\MT(F^\circ,\psi)$. 
Let $Y$ be the vector field on~$\MT(F^\circ,\psi)$ that is tangent to $\{\theta\}\times F^\circ$ for each~$\theta$ 
and along each $\{\theta\} \times F^\circ$ satisfies 
$$
\imath_Yd\lambda= -\chi'\lambda_\psi .
$$
Since $\chi'=0$ near $0$ and $\2\pi$, $Y$ is well defined, and since
$\psi$ is compactly supported in $F^\circ$, $Y$ is compactly supported in $\MT(F^\circ,\psi)$. 
We compute 
\begin{eqnarray*}
\imath_{(\partial_\theta+Y)} d \alpha_s
 &=& s \2 \bigl( \imath_{(\partial_\theta+Y)} \chi' d\theta \wedge \lambda_\psi + \imath_{(\partial_\theta+Y)} d \lambda \bigr) \\
 &=& s \, \chi' \bigl( \lambda_\psi - \lambda_\psi (Y) \, d\theta - \lambda_\psi \bigr) \\
 &=& s \, d\lambda (Y,Y)\,   d\theta \,=\, 0 .
\end{eqnarray*}
Hence on $\MT(F^\circ,\psi)$ the Reeb vector field of $\alpha_s$ is
\begin{eqnarray} \label{ReebsonW'comp}
 R_{\alpha_s}  =\frac{\partial_\theta+Y}{\alpha_s(\partial_\theta+Y)}.
\end{eqnarray}
Note that $R_{\alpha_s} = \partial_\theta$ near~$K$. Since the $\partial_\theta$ component of $R_{\alpha_s}$ never vanishes and since $Y$ is tangent to the pages, $R_{\alpha_s}$ is transverse to $F^\circ\times \{\theta\}$ 
for all~$\theta$. Hence $F^\circ$ is a global hypersurface of section for~$R_{\alpha_s}$ on~$\MT(F^\circ,\psi)$.
We have the first return time map 
\begin{eqnarray} \label{returntimeaway}
T_s \colon F^\circ \rightarrow \R, \quad
  T_s(p) = \inf \left\{ t>0 \mid \phi_{R_{\alpha_s}}^t (0,p) \in \{0\} \times F^\circ \right\}
\end{eqnarray}
and the first return map 
\begin{eqnarray} \label{returnmapaway}
\Upsilon \colon F^\circ \rightarrow F^\circ , \quad 
(0,\Upsilon(p)) = \phi_{R_{\alpha_s}}^{T_s(p)}(0,p) \quad \forall \; p \in F^\circ .
\end{eqnarray}

\begin{rem} \label{returnmapindependent}
{\rm 
Since $R_{\alpha_s}$ is a multiple of the vector field $\partial_\theta+Y$ that does not depend on~$s$, 
the return map~$\Upsilon$ is independent of~$s$. This justifies the absence of the subscript in~\eqref{returnmapaway}.
}
\end{rem}

We note that for all $s \in(0,s_0]$,
\begin{eqnarray} \label{tau_sh_snearK}
T_s \equiv 2\pi \,\mbox{ and }\, \Upsilon = {\rm id} \; \mbox{ on } F^\circ \setminus {\rm supp}\,\psi .
\end{eqnarray}

Recall that we write $\Gamma (\alpha_s)$ for the norm growth $\Gamma (\phi_{\alpha_s})$ 
of the Reeb flow $\phi_{\alpha_s}^t$. 

\begin{lem} \label{l:return} 
There exists $s_1 \in (0,s_0)$ such that for every $s \in (0,s_1]$,
\begin{equation} \label{e:tau2}
\pi \leq T_s \leq 4\pi \quad \mbox{on } F^\circ
\end{equation}
and such that 
\begin{equation} \label{e:htopE}
\Gamma (\alpha_s) \leq E 
\end{equation}
for some constant $E >0$ that depends only on $\psi, \lambda, \chi$.
\end{lem}

\begin{proof} 
We compute
$$
d\theta \left(R_{\alpha_s}\right) = 
\frac{1}{\alpha_s(\partial_\theta+Y)} = \frac{1}{1 + s \bigl(\lambda(Y) + \chi \2 \lambda_\psi(Y) \bigr)}    .
$$
Since $Y$ is compactly supported, we find $s_1 \in (0,s_0)$ such that for every $s \in (0,s_1]$,
\begin{equation} \label{e:122}
\frac 12 \leq \frac{1}{\alpha_s(\partial_\theta+Y)} \leq 2  \quad \mbox{on } \MT (F^\circ,\psi) .
\end{equation}
The inequality~\eqref{e:tau2} follows.
For the second claim,
we apply Proposition~\ref{p:ER} to the vector field
$\partial_\theta + Y$ and the positive function $\frac{1}{\alpha_s(\partial_\theta+Y)}$,
and in view of~\eqref{ReebsonW'comp} and~\eqref{e:122} find that
$\Gamma (\alpha_s) \leq 2 \, \Gamma (\phi_{\partial_\theta +Y}) =:E$.
\end{proof}

\subsection{A family of contact forms near the binding}
\label{s:cfnear}
Let $E>0$ be the constant from Lemma~\ref{l:return}. 
By our inductive hypothesis, 
for any $\gve>0$ there exists a contact form~$\sigma_\gve$ 
on~$(K, \xi |_K)$ such that 
\begin{eqnarray} \label{sigma_eps}
\vol_{\sigma_\gve} (K) = \gve \quad \mbox{ and } \quad \Gamma (\sigma_\gve) \leq E.
\end{eqnarray}
Indeed, there is a contact form $\alpha_0$ on $(K, \xi |_K)$ such that 
$$
\vol_{\alpha_0} (K) = 1 \quad \mbox{ and } \quad \Gamma (\alpha_0) \leq \gve^{1/n} E .
$$ 
We can thus take $\sigma_\gve := \gve^{1/n} \, \alpha_0$. 

Applying Lemma~\ref{le:nearK} to $\sigma_\gve$ we obtain an embedding 
\begin{eqnarray}\label{lambdanearK}
\imath_\gve \colon [0,+\infty) \times K \hookrightarrow F \,\mbox{ such that }\,
         \imath_\gve^* \lambda = \frac{1}{r} \2 \sigma_\gve ,
\end{eqnarray}
and $\imath_\gve(0,q)=q$ for every $q \in K$. 
This embedding induces the smooth coordinate $r \in [0,+\infty)$ on a neighborhood 
of $K=\partial F$ in~$F$. 
There exists $r_{\gve} > 0$ that depends only on $\psi$ and~$\sigma_\gve$ such that 
\begin{equation} \label{psi=id}
\imath_\gve \left([0,r_\gve] \times K \right) \cap \, \supp \psi = \emptyset .
\end{equation}
We define
\begin{eqnarray}\label{F_eepsilon}
F_\gve := F \setminus \bigl( \imath_\gve ([0,r_\gve) \times K) \bigr)
\end{eqnarray}
and note that near the boundary of $\MT (F_\gve,\psi)$ the 
expression~\eqref{alpha_s} for~$\alpha_s$ reads 
\begin{eqnarray} \label{e:ase}
\alpha_s = d\theta+\frac{s}{r} \, \sigma_\gve.
\end{eqnarray}

\begin{lem} \label{fandg} 
For every $\gve >0$ and $s \in(0,r_\gve/2)$ there exist smooth functions 
$$
f,g \colon [0,r_\gve] \rightarrow \R
$$ 
with the following properties. 
\end{lem}

\begin{enumerate}[label=\mbox{(f\arabic*)}]
\item \label{f1}  {\it $f(r)=s/r$ near $r=r_\gve$ and $f(r)=1$ near $r=0$. }

\smallskip
\item \label{f2}  $f(r_\gve/2)=1/2$.                     

\smallskip
\item \label{f'}  {\it $-2/r_\gve \leq f' \leq 0 $ on $[0,r_\gve]$ and $f'<0$ on $[r_\gve /2, r_\gve ]$. }
\end{enumerate}

\begin{enumerate}[label=\mbox{(g\arabic*)}]
\item \label{g1}   {\it $g=1$ on $[r_\gve/2,r_\gve]$ and $g(r)=r^2/2$ near $r=0$. }

\smallskip
\item \label{g2}  {\it $0\leq g'\leq 4/r_\gve$ on $[0,r_\gve]$ and $0 < g'$ on $(0,r_\gve/2]$.}
\end{enumerate}

\smallskip
\noindent 
The easy proof is left to the reader.  
For later use we note that the function
\[
h:= fg' - f'g \colon [0,r_\gve] \rightarrow \R
\]
is positive on $(0,r_\gve]$, satisfies $h(r)=r$ near $r=0$, and
\begin{equation} \label{e:h}
h \,\leq\, \frac{4}{r_\gve} + \frac{2}{r_\gve} = \frac{6}{r_\gve} \;\mbox{ on $[0,r_\gve]$.} 
\end{equation}
Furthermore,
\begin{equation} \label{le:g'h}
\frac{g'}{h} \leq 2 \;\mbox{ on $[0,r_\gve]$.} 
\end{equation}
Indeed, for $r \in [r_\gve / 2, r_\gve]$ we have $( \frac{g'}{h} )(r) =0$ by~\ref{g1}.
In~$0$ we have $\frac{g'}{h} =0$. Further, for $r \in (0, \frac{r_\gve}{2}]$ property~\ref{g2} shows that $g'(r) >0$,
and hence $-f' \2 \frac{g}{g'} \geq 0$. Since $f(r) \geq \frac 12$ by \ref{f2} and~\ref{f'},
we conclude that
$$
\frac{h}{g'} \,=\, f - f' \2 \frac{g}{g'} \,\geq\, \frac 12, 
$$ 
as claimed.

\begin{figure}[h]   
 \begin{center}
  \psfrag{r}{$r$}  \psfrag{1}{$1$}  \psfrag{12}{$\frac 12$}  
  \psfrag{f}{\textcolor{red}{$f$},\textcolor{blue}{$\2 g$}}  
  \psfrag{sr}{\textcolor{red}{$\frac sr$}} \psfrag{re}{$r_\gve$} \psfrag{re2}{$\frac{r_\gve}{2}$} 
  \psfrag{r2}{\textcolor{blue}{$\frac{r^2}{2}$}} 
  \leavevmode\includegraphics{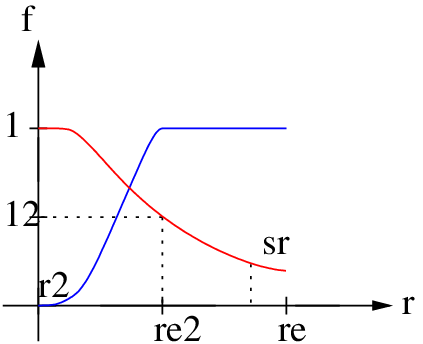}
 \end{center}
 \caption{The functions $f$ and $g$}   \label{fgn.fig}
\end{figure}

Given $\gve>0$ and $s \in(0,r_\gve/2)$, we define the 1-form 
\begin{eqnarray}\label{alpha_sepsnearK}
\alpha_{s,\gve} (\theta,r,q) \,:=\, g(r) \2 d\theta + f(r) \2 \sigma_\gve (q)
\end{eqnarray}
on $S^1 \times [0,r_\gve] \times K$. 
We note that by \ref{f1} and \ref{g1},
$$
\alpha_{s,\gve} (\theta,r,q) \,=\, \frac{r^2}{2} \, d\theta + \sigma_\gve (q)
\quad \mbox{ near $r=0$.} 
$$ 
By considering $\theta$ and $r$ as angular and radial 
coordinates on the disk~$r_\gve \D$, 
we thus see that the 1-form $\alpha_{s,\gve}$ is smooth on~$ r_\gve \D \times K$.

\begin{lem} \label{contactnear} 
For $\gve>0$ and $s \in(0,r_\gve/2)$, $\alpha_{s,\gve}$ is a contact form on $r_\gve \D \times K$.
\end{lem}

\begin{proof} 
We compute
\begin{eqnarray}
\alpha_{s,\gve} \wedge (d\alpha_{s,\gve})^n
&=& (g \, d\theta + f \2 \sigma_\gve) \wedge (g' dr \wedge d\theta + f' dr \wedge \sigma_\gve + f \, d\sigma_\gve)^n 
     \notag \\
&=& n \,h\, f^{n-1} \big( dr \wedge d\theta \wedge \sigma_\gve \wedge (d \sigma_\gve)^{n-1} \big) 
    \label{e:adan}
\end{eqnarray}
where $h =fg'-f'g$.
Since $f^{n-1}h>0$ on $(0,r_\gve]$, it follows that $\alpha_{s,\gve}$ is a contact form away from~$K$. 
Near $K$ we have $h(r)=r$, so that there $\alpha_{s,\gve} \wedge (d\alpha_{s,\gve})^{n}$ reads 
$$
n \1 \bigl( r dr \wedge d\theta \wedge \sigma_\gve \wedge (d \sigma_\gve)^{n-1} \bigr) ,
$$
which is a positive volume form at any point on $K$.
\end{proof}

Away from $K$ the Reeb vector field is
\begin{equation*} %\label{reebnearK}
R_{\alpha_{s,\gve}}(\theta,r,q) = -\frac{f'(r)}{h(r)} \, \partial_\theta + \frac{g'(r)}{h(r)} \, R_{\sigma_\gve}(q)    
\end{equation*}
and has the flow
\begin{equation} \label{reebflownearK}
\phi^t_{\alpha_{s,\gve}} (\theta,r,q) = 
\left( \theta - \frac{f'(r)}{h(r)}\,t,\, r,\, \phi_{\sigma_\gve}^{\frac{g'(r)}{h(r)} \,t}(q)\right),    
\end{equation}
where $\phi_{\sigma_\gve}^t$ is the flow of $R_{\sigma_\gve}$.

\begin{lem} \label{le:S2S}
For $s \in (0,r_{\gve}/2)$ we have $\Gamma  (\alpha_{s,\gve} |_{r_\gve \D \times K}) \leq 2 \2 E$.
\end{lem}

\proof
By continuity it suffices to estimate the differentials of 
$\phi_{\alpha_{s,\gve}}^{\pm n}$ away from~$K$.
We choose a basis $\partial_\theta, \partial_r, \partial_{q_1}, \dots, \partial_{q_{2n-1}}$
of the tangent space at $(\theta,r,q)$.
In view of~\eqref{reebflownearK}, the images of these vectors under $d \phi^{\pm n}_{\alpha_{s,\gve}}(\theta,r,q)$
are
$$
\partial_\theta, \qquad 
\mp \biggl( \frac{f'}{h}\biggr)' \!(r) \,n\, \partial_\theta + \partial_r \pm \biggl( \frac{g'}{h} \biggr)' \!(r) \,n\, R_{\sigma_\gve}, \qquad
d \phi_{\sigma_\gve}^{\pm \bigl( \frac{g'}{h} \bigr)(r) \, n} (\theta,r,q) \, \partial_{q_j} . 
$$
The size of the functions $\bigl( \frac{f'}{h} \bigr)'$ and $\bigl( \frac{g'}{h} \bigr)'$ plays no role when we apply
$\displaystyle \lim_{n \to \infty} \frac 1n \log$, and together with \eqref{le:g'h} we find that
$$
\Gamma (\alpha_{s,\gve} |_{r_\gve \D \times K}) \,=\, \max_r \left| \biggl( \frac{g'}{h} \biggr) (r) \right| \, \Gamma (\sigma_\gve) 
\,\leq\, 2 \2 \Gamma (\sigma_\gve) .
$$
The lemma follows together with assumption \eqref{sigma_eps}.
\proofend

\subsection{A family of contact forms on $\OB(F,\psi)$}
For every $\gve>0$ and $s \in(0,r_\gve/2)$, we define 
\begin{eqnarray} \label{alpha_epseps}
\alpha_{s,\gve} =  \left\{
\begin{array}{ll}
      \alpha_{s} & \textrm{ on } \MT(F_{\gve},\psi) \\ [.2em]
      \alpha_{s,\gve} = g(r)\, d\theta + f(r) \,\sigma_\gve &\textrm{ on } r_\gve\D\times K 
\end{array}      \right.
\end{eqnarray}
on the abstract open book 
$$
\OB(F,\psi) = \MT(F_{\gve},\psi) \cup \left(r_\gve\D\times K\right)
$$ 
where $\alpha_s$ on $\MT(F_{\gve},\psi)$ is defined by \eqref{alpha_s} and $f$ and $g$ 
are given in Lemma~\ref{fandg}. 
By~\eqref{e:ase} and the properties~\ref{f1} and~\ref{g1}, each $\alpha_{s,\gve}$ is a well-defined 
contact form on~$\OB(F,\psi)$.

We first estimate the volume. 
$$
\int_{\OB(F,\psi)}\alpha_{s,\gve}\wedge (d\alpha_{s,\gve})^n \,=\,
   \int_{\MT(F_{\gve},\psi)} \alpha_{s} \wedge (d\alpha_{s})^n + 
   \int_{r_\gve \D \times K} \alpha_{s,\gve} \wedge (d\alpha_{s,\gve})^n .
$$
For $s \in(0,s_1]$, where the positive number $s_1$ is given by Lemma~\ref{l:return}, we have 
$$
\int_{\MT (F_\gve,\psi)} \alpha_{s} \wedge (d \alpha_{s})^n 
\,=\, \int_{F_{\gve}} T_s\, (d\alpha_s|_{\{0\} \times F_\gve})^n 
   \,=\, \int_{F_{\gve}} T_s\, s^n (d\lambda)^n \,\leq\, 
	       2 \2 s^n\int_{F_{\gve}}(d\lambda)^n,
$$
where we used \eqref{e:tau2} in the last inequality. 
For the second term we use \eqref{e:adan}, the assumption~\eqref{sigma_eps}, 
and $f \leq 1$ and \eqref{e:h} to estimate
\begin{eqnarray*}
\int_{r_\gve\D\times K}\alpha_{s,\gve}\wedge (d\alpha_{s,\gve})^n
&=& \int_{r_\gve \D \times K} n \, h \, f^{n-1} 
            \bigl( dr \wedge d\theta \wedge \sigma_\gve \wedge (d \sigma_\gve)^{n-1} \bigr) \\
&=& 2\pi n  \vol_{\sigma_\gve} (K) \int_0^{r_\gve} h \2 f^{n-1} dr \\
&\leq& 12 \2 \pi n \2 \gve .
\end{eqnarray*}
Together we get
\begin{equation}\label{volumebound}
  \int_{\OB(F,\psi)} \alpha_{s,\gve} \wedge (d\alpha_{s,\gve})^n
	\,\leq\, 2 \2 s^n\int_{F_{\gve}} (d\lambda)^n + 12 \2 \pi n \2 \gve. 
\end{equation}

Next we estimate the norm growth~$\Gamma$. 
The bound~\eqref{e:htopE} also applies to the subset $\MT(F_\gve,\psi)$ of~$\MT(F^\circ,\psi)$
since $\Gamma$ is monotone with respect to inclusion of compact invariant 
subsets.
Since $\MT(F_\gve,\psi)$ and $r_\gve \D \times K$ are invariant under the flow,
we conclude together with assertion (3) of Proposition~\ref{p:Gammaele} 
and Lemma~\ref{le:S2S} that
\begin{eqnarray}\label{entropyoverall}
\Gamma (\alpha_{s,\gve}) = 
\max \left\{ \Gamma (\alpha_s |_{\MT(F_\gve,\psi)}) , 
           \Gamma ( \alpha_{s,\gve} |_{r_\gve \D \times K}) \right\} 
\leq 2E.
\end{eqnarray}

Now given any $\varepsilon_0 >0$, we choose $\gve>0$ such that $12 \2 \pi n \2 \gve \leq \varepsilon_0 /2$. 
Once $\gve$ is fixed, so are $r_\gve$ and $\int_{F_{\gve}}(d\lambda)^n$. 
If we choose $s>0$ such that 
$$
s \leq \min \left\{ s_1, \frac{r_\gve}{2},
    \left(\frac{\varepsilon_0}{4\int_{F_{\gve}}(d\lambda)^n} \right)^{\frac{1}{n}}\right\},
$$ 
then the right-hand side of \eqref{volumebound} is $\leq \gve_0$.
Since $n! \2 \omega_n \geq 1$ we thus get a contact form $\alpha_{s,\gve}$ on~$\OB(F,\psi)$ such that 
\begin{equation}  \label{e:2Ee0}
\Gamma (\alpha_{s,\gve}) \leq 2 E \quad \mbox{ and } \quad
\vol_{\alpha_{s,\gve}} \bigl( \OB (F,\psi) \bigr) \leq \varepsilon_0 .
\end{equation}
The last step of the proof consists of pushing the contact form $\alpha_{s,\gve}$ to~$M$ in such a way that 
the contact structure $\ker  \alpha_{s,\gve}$ is mapped to~$\xi$. 
This is possible thanks to the following lemma.

\begin{lem} \label{supporting} 
There exists a diffeomorphism $\rho$ of $M$ such that $\rho_* (\xi) = \ker \alpha_{s,\gve}$. 
\end{lem}

Postponing the proof, we use the lemma to complete the proof of Theorem~\ref{t:mainn}.
Thanks to the lemma, the 1-form
\begin{equation} \label{e:tauvr}
\tau \,:=\, 
\left( \vol_{\alpha_{s,\gve}} \bigl( \OB (F,\psi) \bigr) \right)^{-1/(n+1)} \, \rho^* \alpha_{s,\gve}
\end{equation}
is a contact form on $(M,\xi)$, and $\vol_\tau(M) = 1$.
By the elementary properties (1) and (4) in Proposition~\ref{p:Gammaele}
and by~\eqref{e:2Ee0}, 
$$
\widehat \Gamma (\tau) \,=\, \Gamma (\tau) \,=\, 
\left( \vol_{\alpha_{s,\gve}} \bigl( \OB (F,\psi) \bigr) \right)^{1/ (n+1) } \, \Gamma (\alpha_{s,\gve}) 
\,\leq\, (2E) \, (\gve_0)^{1/ (n+1) } .
$$
Since the constant $E$ from Lemma~\ref{l:return} depends only on $\psi, \lambda, \chi$, 
which are fixed data associated with~$(M,\xi)$,
and since $\gve_0 >0$ is arbitrarily small, 
we obtain $\tau$ with $\widehat \Gamma (\tau)$ as small as we like.

\subsection{Proof of Lemma \ref{supporting}.}
We first show that the obvious open book structure on $\OB(F,\psi)$ is a Liouville open book 
with contact binding form $\alpha_{s,\gve}$. 
Let 
$$
\widetilde \Theta \colon \OB(F,\psi)\setminus K \rightarrow S^1
$$
be the fibration induced by the projection $\MT(F,\psi) \rightarrow S^1$. 
We construct a defining function~$\tilde h$ as follows. 
As before, we consider the variable $r \in [0,+\infty)$ on a neighborhood of $K=\partial F$ in~$F$ 
that is induced by the embedding~\eqref{lambdanearK}.
Let
$$
\tilde u \colon F \rightarrow [0,\infty)
$$
be a smooth function and $d>0$ and $\delta>0$ be constants such that

\smallskip
\begin{enumerate}[label=\mbox{(df\arabic*)}]
\item \label{df1}  $\tilde u (r,q)=r$ for $(r,q) \in [0,r_\gve] \times K$, 

\smallskip
\item \label{df2}  $\tilde u \equiv d$ on $([0,r_\gve+\delta) \times K)^c$ \,and\, 
                          $\supp \psi \subset ([0,r_\gve+\delta] \times K)^c$,
				
\smallskip
\item \label{df3}  $\tilde{u}$ depends only on $r$ and 
                          $\partial_r \1 \tilde u \geq 0$ on $[0,r_\gve+\delta] \times K$.
\end{enumerate}

\begin{figure}[h]   
 \begin{center}
  \psfrag{u}{$\tilde u$} \psfrag{d}{$d$} 
  \psfrag{re}{$r_\gve$} 
  \psfrag{red}{$r_{\gve} \!+\! \delta$}
  \psfrag{rK}{$r \times K$}  
  \leavevmode\includegraphics{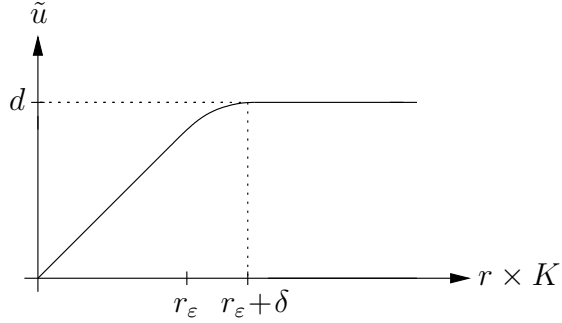}
 \end{center}
 \caption{The function $\tilde u$, schematically}   \label{utilde}
\end{figure}

\smallskip \noindent
Since $\tilde u$ is constant on $\supp \psi$, the $S^1$-invariant extension of~$\tilde u$ 
is a well-defined smooth function on~$\MT(F,\psi)$, which constitutes the function $|\tilde h|$. 
Pairing $|\tilde h|$ with $\widetilde \Theta$ leads to a well-defined defining function $\tilde h$ 
for the open book $(K,\widetilde \Theta)$ on~$\OB(F,\psi)$. 
Note that on $r_\gve \D \times K$, $\tilde h$ is simply the projection to the disk~$r_\gve \D$, 
which is smooth. 

\medskip \noindent
\textbf{Claim 1:} 
\textit{$d \bigl( \alpha_{s,\gve}/|\tilde{h}| \bigr)$ induces an ideal Liouville structure on each fibre 
of $\widetilde{\Theta}$.}

\begin{proof} 
For each $\theta \in S^1$ we abbreviate
\begin{eqnarray} \label{lambdatilde}
\tilde{\lambda}_\theta := \bigl( \alpha_{s,\gve}/|\tilde{h}| \bigr) |_{T(\{\theta\} \times F^\circ)} .
\end{eqnarray} 
We need to check that $d\tilde \lambda_\theta$ is a positive symplectic form on 
$\{\theta\} \times F^\circ = \widetilde{\Theta}^{-1}(\theta)$. 
We do this on the three regions separately:

\medskip \noindent
\underline{On $\{\theta\} \times (0,r_\gve] \times K$}: 
By \eqref{alpha_epseps} and~\ref{df1} we have 
\begin{eqnarray} \label{er_eps}
\tilde{\lambda}_\theta = \frac{f(r)}{r} \, \sigma_{\gve}
\end{eqnarray}
and so 
$$
(d\tilde{\lambda}_\theta)^n = n \2 f^{n-1} \2 \frac{f'r-f}{r^{n+1}}\,dr \wedge \sigma_{\gve} \wedge (d\sigma_{\gve})^{n-1}.
$$
In view of the parametrization \eqref{lambdanearK},
$d (\frac 1r \sigma_{\gve})$ is a positive symplectic form on $\{\theta\} \times F^\circ$,
whence $(d(\frac 1r \sigma_{\gve}))^n = - \frac{n}{r^{n+1}} \, dr \wedge \sigma_{\gve} \wedge (d\sigma_{\gve})^{n-1}$
is a positive volume form. 
Since $f'r-f<0$ by~\ref{f1} and~\ref{f'}, it follows that
$(d \tilde{\lambda}_\theta)^n$ is a positive volume form, i.e., $d \tilde{\lambda}_\theta$ is a 
positive symplectic form.

\medskip \noindent
\underline{On $\{\theta\} \times [r_\gve,r_\gve+\delta]\times K$}: 
By~\ref{df2} we have $\psi = \id$ on this set. 
Hence, by~\eqref{e:ase},
\begin{eqnarray} \label{err_eps}
\tilde{\lambda}_\theta = \frac{s}{r \tilde u} \, \sigma_{\gve}.
\end{eqnarray}
By \ref{df3}, $\tilde u$ depends only on $r$. We thus obtain
$$
(d\tilde{\lambda}_\theta)^n = 
- s \2 n \2 \frac{\tilde{u} + r \2 \partial_r \tilde{u}}{r^{n+1} \tilde u^{n+1}} \, 
                  dr \wedge \sigma_{\gve} \wedge (d\sigma_{\gve})^{n-1}.
$$
Also by \ref{df3}, $\tilde{u} + r \2 \partial_r \tilde{u}>0$, and the claim follows as in the previous case.

\medskip \noindent
\underline{On $\{\theta\} \times ([0,r_\gve+\delta)\times K)^c$}: 
By \ref{df2}, $\tilde u \equiv d$ and 
\begin{eqnarray} \label{faraway}
\tilde{\lambda}_\theta = \frac{s}{d} \left(\lambda + \chi(\theta) \lambda_\psi\right) .
\end{eqnarray}
Hence $d \tilde{\lambda}_\theta=\frac{s}{d} d\lambda$, which is a positive symplectic form.
\end{proof}

Now we are in the following situation. On $\OB(F,\psi)$ we have the Liouville open book
\begin{eqnarray} \label{LOB1}
\left( K,\widetilde{\Theta}, d(\alpha/|h|)|_{T(\{\theta\} \times F^\circ)} \right),
\end{eqnarray}
which symplectically supports the contact structure $\xi = \ker \alpha$. 
Here $\alpha$, $\xi$, and~$h$ stand for the objects induced by the correspondence between 
$M$ and~$\OB(F,\psi)$ given by the symplectically spinning vector field~$X$ on~$M$. 
Moreover, by Claim~1 we have a second Liouville open book
\begin{eqnarray} \label{LOB2}
\left( K, \widetilde{\Theta}, d \bigl( \alpha_{s,\gve}/|\tilde{h}| \bigr)|_{T(\{\theta\} \times F^\circ)} \right),
\end{eqnarray}
which symplectically supports the contact structure $\ker \alpha_{s,\gve}$. 
By the identities \eqref{er_eps}, \eqref{err_eps}, and \eqref{faraway}, 
the ideal Liouville structures 
$\bigl( d \bigl( \alpha_{s,\gve}/|\tilde{h}| \bigr) |_{T(\{\theta\} \times F^\circ)} \bigr)_{\theta \in S^1}$ 
are invariant under the flow of the vector field~$\partial_\theta$. 
Although $\tilde \lambda_\theta$ in~\eqref{faraway} is not invariant under this flow, the symplectic form
$d \tilde \lambda_\theta = \frac s d \, d\lambda$ is.
The vector field~$\partial_\theta$ is therefore a symplectically spinning vector field on the LOB~\eqref{LOB2}.
Note that the symplectically spinning vector field~$X$ on~$M$ also reads $\partial_\theta$ on the LOB~\eqref{LOB1}.
                                
\medskip \noindent
\textbf{Claim 2.} 
\textit{There exists a diffeomorphism   
\begin{eqnarray}\label{Phi} 
\Phi \colon \OB(F,\psi) \rightarrow \OB(F,\psi) 
\end{eqnarray} 
such that $\Phi \circ \widetilde \Theta = \widetilde \Theta \circ \Phi$ and the restriction of~$\Phi$  
to each fibre is symplectic, that is,  
$$ 
\Phi^*d \bigl( \alpha_{s,\gve}/|\tilde{h}| \bigr)|_{T(\{\theta\} \times F^\circ)} =  
           d \bigl( \alpha/|h| \bigr)_{T(\{\theta\} \times F^\circ)}, 
\quad \forall \,\theta \in S^1 . 
$$ 
}

If such a diffeomorphism exists, then  $\ker \Phi^*\alpha_{s,\gve}$ and $\ker \alpha$ are two contact structures on $\OB(F,\psi)$ which symplectically support the Liouville open book~\eqref{LOB1}. 
Hence they are isotopic by Proposition~\ref{prop21}. 
By Gray's stability theorem we then find a diffeomorphism~$\hat \rho$ of~$M$ such that 
$\hat \rho_* (\ker \alpha) = \ker \Phi^* \alpha_{s,\gve}$. 
Set $\rho = \Phi \circ \hat \rho$.
Since $\ker \Phi^* \alpha_{s,\gve} = \Phi_*^{-1} (\ker \alpha_{s,\gve})$, we conclude 
that $\rho_* (\ker \alpha) = \ker \alpha_{s,\gve}$, as claimed in Lemma~\ref{supporting}.

\medskip \noindent
{\it Proof of Claim 2.} 
We have the following ideal Liouville structures on the $0$-page:
\begin{eqnarray}
\widetilde \omega &:=& d \bigl( \alpha_{s,\gve}/|\tilde{h}| \bigr)|_{T(\{0\} \times F^\circ)}, \label{omega_1} \\ [0.5em]
\omega            &:=& d \bigl( \alpha/|h| \bigr)|_{T(\{0\}\times F^\circ)}=d\lambda. \label{omega}
\end{eqnarray}
We first show that 
$$
\omega_t := (1-t) \2 \omega + t \2 \widetilde \omega
$$
is symplectic on $F^\circ$ for all $t\in [0,1]$. In fact, we claim that
\begin{eqnarray}\label{lambda_t}
\lambda_t = (1-t) \2 \lambda + t \2 \tilde{\lambda}
\end{eqnarray}
is a Liouville form on $F^\circ$ for all $t \in [0,1]$, 
where $\lambda$ is the primitive of~$\omega$ given by~\eqref{lambda}
and $\tilde \lambda$ is the primitive of $\widetilde \omega$ given by~\eqref{lambdatilde}.
Again, we compute $d\lambda_t$ on different subsets of~$F^\circ$:

\medskip \noindent
\underline{On $\{\theta\} \times (0,r_\gve]\times K$}: 
By \eqref{er_eps} we have 
$$
\lambda_t = (1-t) \,\frac{1}{r}\, \sigma_\gve + t \,\frac{f(r)}{r}\, \sigma_\gve = \frac{\kappa(r)}{r} \,\sigma_\gve
$$
where $\kappa=(1-t)+tf$. We have $\kappa>0$ and $\kappa'<0$, so that $\kappa'r-\kappa<0$. 
Hence the claim follows as in the first case of Claim~1.

\medskip \noindent
\underline{On $\{\theta\}\times [r_\gve,r_\gve+\delta]\times K$}: 
By \eqref{err_eps} we have
$$
\lambda_t = (1-t) \,\frac{1}{r}\, \sigma_\gve + t \,\frac{s}{r \tilde u}\, \sigma_\gve =
  \frac{\kappa(r)}{r} \, \sigma_\gve
$$
where $\kappa = (1-t) + t s/ \tilde u$. We have $\kappa>0$ and $\kappa' \leq 0$, so that $\kappa'r-\kappa<0$. 
The claim follows as above.

\medskip \noindent
\underline{On $\{\theta\} \times ([0,r_\gve+\delta)\times K)^c$}: 
By \eqref{faraway} we have
$$
d \lambda_t = (1-t) \2 d\lambda + t \,\frac{s}{d}\, d\lambda = \left( (1-t) + t \2 \frac sd \right) \, d \lambda.
$$
Hence $\omega_t = d\lambda_t$ is symplectic on $F^\circ$ for all $t \in [0,1]$.

%\medskip
%Denote by $u \colon F \rightarrow [0,+\infty)$ the restriction of $|h|$ to~$F$, 
%where $h \colon M \rightarrow \C$ is the initial defining function for the open book~$(K,\Theta)$. 
%Then the $1$-form $u \1  \lambda_t$ extends to~$F$ for each $t \in [0,1]$, the extension being the smooth $1$-form 
%$$
%(1-t) \,\alpha|_{TF} + t \, \frac{u}{\tilde u} \, \alpha_{s,\gve} |_{TF}.
%$$
%Note that the function $u / \tilde u$ is smooth since $0$ is a regular value of~$h$ and $\tilde u = r$ near~$K$.

%We have a smooth path of ideal Liouville structures $(\omega_t)_{t \in [0,1]}$ 
%in the sense of Proposition~\ref{p:6} such that $\omega_0 = \omega$ and $\omega_{1} = \widetilde \omega$.
%
%It is also clear that for each $t \in [0,1]$ the boundary contact structure associated to~$\omega_t$ 
%is $\ker \alpha$. Hence by Proposition~\ref{p:6} there exists a smooth isotopy 
%$(\psi_t)_{t\in [0,1]}$ of~$F$ such that
 
\medskip
Recall that by \eqref{lambdanearK} and \ref{f1},
$$
\omega = \widetilde{\omega} = d \left( \frac{1}{r} \,\sigma_\varepsilon \right) 
\quad \mbox{on a deleted neighbourhood of $K$.}
$$
Hence this identity holds on the same deleted neighbourhood of~$K$ 
for all the symplectic forms $\omega_t$, $t \in [0,1]$.
Applying the standard Moser argument to the path $\omega_t$, 
we obtain a smooth isotopy $(\psi_t)_{t\in [0,1]}$ of~$F$ such that

\smallskip
\begin{enumerate}
[label=\mbox{($\Psi$\arabic*)}]   
\item   \label{psi1}   $\psi_0 = \id$; 

\smallskip
\item   \label{psi2}   $\psi_t = {\rm id}$ near $K$ for all $t \in [0,1]$; 

\smallskip
\item   \label{psi3}   $\psi_t^* \, \omega_t = \omega_0 = \omega$ for all $t \in [0,1]$.
\end{enumerate}

\smallskip \noindent
Now we define $\Phi \colon [0,2\pi] \times F \rightarrow [0,2\pi] \times F$ by
\begin{eqnarray}\label{isotopy}
\Phi(\theta,p) := \left( \theta,\psi_{1} \circ \psi^{-1}_{\frac{\theta}{2\pi}} \circ \psi^{-1} \circ \psi_{\frac{\theta}{2\pi}} (p) \right)
\end{eqnarray}
where $\psi$ is the monodromy that we fixed at the outset of the proof. 
We note that 
$$
\Phi(2\pi,p) = \bigl (2\pi, \psi^{-1} \circ \psi_{1}(p) \bigr),
$$
and by~\ref{psi1}, 
$$
\Phi(0,\psi(p)) 
= \left(0, \psi_{1}(p) \right) = \bigl( 0, \psi(\psi^{-1} \circ \psi_{1}(p)) \bigr).
$$
Hence $\Phi$ descends to a diffeomorphism on $\MT(F,\psi)$. Since $\psi = \id$ near~$K$
and $\psi_t = \id$ near~$K$ for each~$t$ by~\ref{psi2}, we have that $\Phi = \id$ on 
a neighbourhood of~$\partial \MT(F,\psi)$. 
Hence $\Phi$ descends to a diffeomorphism on~$\OB(F,\psi)$. 
By definition, $\Phi$ commutes with~$\widetilde {\Theta}$. 

Now recall that $\partial_{\theta}$ is a symplectically spinning vector field for both LOBs~\eqref{LOB1} and~\eqref{LOB2}.
In view of~\eqref{omega_1} and~\eqref{omega} and identifying $\{\theta\} \times F^\circ$ with $\{0\} \times F^\circ$ via the flow
of~$\partial_\theta$, we can therefore identify
\begin{eqnarray*}
d \bigl( \alpha / |h| \bigr) |_{T (\{\theta\} \times F^\circ)} & \mbox{ with } & 
                                            \omega |_{T (\{\theta\} \times F^\circ)} := \omega , \\
d \bigl( \alpha_{s,\gve} / |\tilde h| \bigr) |_{T (\{\theta\} \times F^\circ)} & \mbox{ with } & 
                         \widetilde \omega |_{T (\{\theta\} \times F^\circ)} := \widetilde \omega . 
\end{eqnarray*}
Also recall that $\psi^* \omega = \omega$. Since $\partial_\theta$ generates the monodromy~$\psi$
and preserves~$\widetilde \omega$, we also have $\psi^* \widetilde \omega = \widetilde \omega$.
Therefore, $\psi^*\omega_t = \omega_t$ for all~$t \in [0,1]$. 
Inserting \eqref{isotopy} and using~\ref{psi3} we obtain, 
with the abbreviation $F_\theta^\circ = T (\{\theta\} \times F^\circ)$,
\begin{eqnarray*}
\Phi^* d \bigl( \alpha_{s,\gve}/|\tilde{h}| \bigr) |_{F_\theta^\circ}
&=&\Phi^* \widetilde \omega  |_{F_\theta^\circ} \\
&=&\big(\psi_{1}\circ\psi^{-1}_{\frac{\theta}{2\pi}} \circ \psi^{-1}\circ \psi_{\frac{\theta}{2\pi}} \big)^* \widetilde \omega |_{F_\theta^\circ} \\
&=&\psi_{\frac{\theta}{2\pi}}^* \, (\psi^{-1})^* \, (\psi^{-1}_{\frac{\theta}{2\pi}})^* \, \psi_{1}^* \omega_{1} |_{F_\theta ^\circ} \\
&=&\psi_{\frac{\theta}{2\pi}}^* \, (\psi^{-1})^* \, (\psi^{-1}_{\frac{\theta}{2\pi}} )^* \, \omega_0 |_{F_\theta^\circ} \\
&=&\psi_{\frac{\theta}{2\pi}}^* \, (\psi^{-1})^* \omega_{\frac{\theta}{2\pi}} |_{F_\theta^\circ} \\
&=&\psi_{\frac{\theta}{2\pi}}^* \, \omega_{\frac{\theta}{2\pi}} |_{F_\theta^\circ} \\
&=&\omega_0 |_{F_\theta^\circ} = \omega |_{F_\theta^\circ}  = d(\alpha/|h|) |_{F_\theta^\circ} .
\end{eqnarray*} 
This concludes the proof of Lemma \ref{supporting},  
and hence of Theorem~\ref{t:mainn}.

\section{Full entropy spectrum} \label{s:spectrum}

By Theorem~\ref{t:mainn} every closed contact manifold $(M,\xi)$
admits normalized contact forms with arbitrarily small topological entropy.
On the other hand, one can always find normalized contact forms~$\alpha$
on $(M,\xi)$ with arbitrarily large topological entropy.
We first sketch a direct proof of this fact. 
A proof of a stronger statement relying on our previous construction is given thereon.

\smallskip \noindent
{\it Existence of contact forms with large entropy.}
Choose a transverse knot $\gamma$, that is, 
a simple closed curve $\gamma \colon S^1 \to M$ such that $\dot \gamma$ is everywhere
transverse to~$\xi$.
Let $B^{2n}$ be the closed ball in~$\R^{2n}$ of radius~$1$.
By the normal form theorem for transverse knots \cite[Example~2.5.16]{Gei08},
we find a full torus $\ct = B^{2n} \times S^1$ around~$\gamma$ with coordinates
$(\theta_1, \dots, \theta_n, r_1, \dots, r_n, q)$ such that $\gamma$ is parametrized by $r_j = 0$ and~$q$
and such that on~$\ct$ the contact structure~$\xi$ is the kernel of $\tau := d q + \sum_j r_j \2 d \theta_j$.

Next, take a second such full torus $(\ct_2, \tau_2)$ that is disjoint from~$\ct$,
perturb $\tau_2$ in the interior of~$\ct_2$ to a contact form~$\tau_2'$
with positive topological entropy, and take a contact form~$\alpha$ on~$(M,\xi)$ that agrees with 
$\tau$ on~$\ct$ and with $\tau_2'$ on~$\ct_2$.
Now for $\delta >0$ let $F \colon M \to \R$ be a positive smooth function 
that is equal to~$\delta$ outside~$\ct$,
on~$\ct$ depends only on the coordinates~$r_1, \dots, r_n$, 
and is such that $\vol_{F\tau} (\ct) = 1$. 
Take the smooth contact form $\alpha_{\delta, F}$ on~$M$ that is equal to~$F\tau$ on~$\ct$
and to $\delta \alpha$ on~$M \setminus \ct$.
Then $h_{\top} (\alpha_{\delta, F}|_\ct) = 0$ whence
$$
h_{\top}(\alpha_{\delta, F}) \,=\, h_{\top}(\alpha_{\delta, F} |_{M\setminus \ct})
\,=\, h_{\top}(\delta \alpha |_{M\setminus \ct}) \,=\,
\delta^{-1} \, h_{\top}(\alpha) .
$$
With this one then readily finds
$$
\left( \widehat h_{\top}(\alpha_{\delta, F}) \right)^{n+1} \,=\,
\left( \widehat h_{\top}(\alpha |_{M \setminus \ct}) \right)^{n+1} +   
     \left( \frac{h_{\top}(\alpha)}{\delta} \right)^{n+1} .
$$
Varying $\delta$ on $(0,\infty)$ we obtain a normalized contact form with 
topological entropy~$c$ for every $c > \widehat h_{\top}(\alpha |_{M \setminus \ct})$. 
\proofend

\begin{rem}
{\rm
For spherizations $S^*Q_k$ of closed orientable surfaces of genus $k \geq 2$
this result has been obtained in~\cite{EK19} inside the much smaller class of
geodesic flows of negatively curved Riemannian metrics:
For every $c \geq 2\pi \sqrt{2(k-1)}$ there exists a negatively curved Riemannian metric~$g$
on~$Q_k$ such that $\widehat h_{\top}(g) =c$.
In the class of all Riemannian metrics, 
geodesic flows with arbitrarily large $\widehat h_{\top}$
were constructed on all closed manifolds of dimension at least~two
already in~\cite{Man81}. 
}
\end{rem}

We shall now combine the above argument with the construction in the proof of Theorem~\ref{t:mainn} 
to prove the following more precise result.

\begin{prop} \label{p:spec}
Let $(M,\xi)$ be a closed co-orientable contact manifold of dimension $2n+1 \geq 3$. 
Then for every $c>0$ there exist normalized contact forms~$\alpha$ and~$\alpha'$ on~$(M,\xi)$
such that $h_{\top}(\alpha) = c$ and $\Gamma (\alpha') =c$.
\end{prop}

\proof
We give the proof for the topological entropy. The proof for the norm growth~$\Gamma$ is similar.
Fix $c>0$ as in the proposition. 
As in~\eqref{e:tauvr} let
$$
\tau \,=\, v^{-1/(n+1)} \, \rho^* \alpha_{s,\gve}
$$
be a contact form with $h_{\top}(\tau) \leq c$ and 
$\vol_\tau (M) = 1$.
Here we abbreviated $v := \vol_{\alpha_{s,\gve}} \bigl( \OB(F,\psi) \bigr)$.
In view of~\eqref{e:2Ee0} we can assume that
\begin{equation} \label{e:volase1}
v \leq 1.
\end{equation}
Recall that $\alpha_{s,\gve}$ on $r_\gve \D \times K$ was constructed
recursively, starting from the circle~$S^1$.
For $j=1, \dots, n$ let $f_j,g_j,h_j$ be the functions from Lemma~\ref{fandg}
that we used in~\eqref{alpha_sepsnearK} to construct $\alpha_{s,\gve_j}$ on $r_{\gve_j} \D \times K_j$.
Set $r = \min \{ r_{\gve_1}, \dots, r_{\gve_n} \} >0$.
Let $B^{2n}(r)$ be the closed $r$-ball in~$\R^{2n}$ with polar coordinates
$(\otheta, \ovr) \equiv (\theta_1, \dots, \theta_n, r_1, \dots, r_n)$,
and abbreviate $\ct_r = B^{2n}(r) \times S^1$.

For $\delta >0$ choose a positive smooth function $F \colon \OB(F,\psi) \to \R$ 
with the following properties:

\begin{enumerate}[label=\mbox{(F\arabic*)}]
\item \label{F1}
$F = \delta$ on $\ct_r^c := \OB(F,\psi) \setminus \ct_r$.

\smallskip
\item \label{F2}
$F$ only depends on $\ovr$ on $\ct_r$.

\smallskip
\item \label{F3}
$\vol_{F \alpha_{s,\gve}} (\ct_r) = 1$.
\end{enumerate}

\begin{lem} \label{le:F123}
$h_{\top} ( F \alpha_{s,\gve} |_{\ct_r}) =0$.
\end{lem}

\proof
The full torus $\ct_r$ is foliated by the tori
$$
\T_{\ovr} \,:=\, (S^1)^k \times \left\{ \ovr = (r_1, \dots, r_n) \right\} \times S^1
$$
with $r_1, \dots, r_n \geq 0$ constant, of dimension $k+1 \leq n+1$.
By \ref{f1} and~\ref{g1},
\begin{equation} \label{e:fjgj}
\frac{f_j'(r_j)}{h_j(r_j)} = -2r_j \quad \mbox{and} \quad \frac{g_j'(r_j)}{h_j(r_j)} = 1
\quad \mbox{for $r_j$ near $0$.}  
\end{equation}
Writing out \eqref{reebflownearK} recursively and using \eqref{e:fjgj} we see that
the Reeb flow of~$\alpha_{s,\gve}$ leaves the tori~$\T_{\ovr}$ invariant 
and on each $\T_{\ovr}$ is a Kronecker flow.

Applying now \eqref{alpha_sepsnearK} recursively we see that on $\ct_r$,
$$
\alpha_{s,\gve} (\otheta, \ovr, q) \,=\, 
\rho_1(\ovr) \2 d\theta_1 + \dots + \rho_n(\ovr) \2 d\theta_n
+\rho(\ovr) \2 dq
$$
with smooth functions $\rho_j, \rho$.
The Reeb flow of any $1$-form of this form leaves the tori $\T_\ovr$ invariant and there restricts to a Kronecker flow.
This is clear at~$\ovr$ if the Jacobian determinant of 
$\left( \frac{\partial \rho_i}{\partial r_j} (\ovr) \right)$ does not vanish, and in general follows by approximation.
Therefore, for each~$\ovr$ the Reeb flow of $F \alpha_{s,\gve}$ is a Kronecker flow on~$\T_\ovr$,
and hence $h_{\top} (F \alpha_{s,\gve} |_{\T_{\ovr}}) =0$.
The variational principle for topological entropy now implies that
$$
h_{\top} (F \alpha_{s,\gve}) \,=\, \sup_{\ovr} h_{\top} (F \alpha_{s,\gve} |_{\T_{\ovr}}) = 0 ,
$$
as claimed.
\proofend

By Lemma \ref{le:F123} the topological entropy of $\alpha_{s,\gve}$ and $F \alpha_{s,\gve}$
on~$\ct_r$ vanishes.
Together with~\ref{F1} we obtain
\begin{equation} \label{e:hd}
h_{\top}(F \alpha_{s,\gve}) \,=\, h_{\top}(F \alpha_{s,\gve} |_{\ct_r^c})
\,=\, \tfrac{1}{\delta} \, h_{\top}(\alpha_{s,\gve} |_{\ct_r^c}) 
\,=\, \tfrac{1}{\delta} \, h_{\top}(\alpha_{s,\gve}) . 
\end{equation}
Now consider the contact form
\begin{equation} \label{e:defFrho}
(F \circ \rho) \2 \tau \,=\, v^{-1/(n+1)} \, \rho^* (F \alpha_{s,\gve})
\end{equation}
on $(M,\xi)$.
By \eqref{e:defFrho} and \eqref{e:hd},
\begin{eqnarray*}
\widehat h_{\top} \bigl( (F \circ \rho) \2 \tau \bigr)^{n+1} 
&=& \vol_{(F \circ \rho) \2 \tau} (M) \; 
                         h_{\top} \bigl( (F \circ \rho) \2 \tau \bigr)^{n+1} \\
&=& 
v^{-1}  \vol_{F \alpha_{s,\gve}} \bigl( \OB(F,\psi) \bigr) \; 
                  v \, h_{\top} (F \alpha_{s,\gve})^{n+1} \\
&=& \vol_{F \alpha_{s,\gve}} \bigl( \OB(F,\psi) \bigr) \, 
                                 \delta^{-(n+1)} \, h_{\top} (\alpha_{s,\gve})^{n+1} .
\end{eqnarray*}
By \eqref{e:volase1}, $\underline{v} := \vol_{\alpha_{s,\gve}} (\ct_r^c) \in (0,1)$,
and by \ref{F1} and~\ref{F3},
$$
\vol_{F\alpha_{s,\gve}} \bigl( \OB(F,\psi) \bigr) \,=\, 1+\delta^{n+1} \underline{v} ,
$$
whence
$$
\widehat h_{\top} \bigl( (F \circ \rho) \2 \tau \bigr)^{n+1} \,=\, 
\left( \underline{v} + \delta^{-(n+1)} \right) 
               h_{\top} (\alpha_{s,\gve})^{n+1} \,=:\, (f(\delta))^{n+1} .
$$
Assume first that $h_{\top} (\alpha_{s,\gve}) >0$. 
Then the range of the function $f \colon (0,\infty) \to \R$ is $\left( \underline{v}^{1/(n+1)} \, h_{\top} (\alpha_{s,\gve}), \infty \right)$.
Since $\underline{v}<1$ and $h_{\top} (\alpha_{s,\gve}) \leq c$, we in particular find $\delta$
such that $\widehat h_{\top} \bigl( (F \circ \rho) \2 \tau \bigr) =c$.
If $h_{\top} (\alpha_{s,\gve}) =0$, 
Proposition~\ref{p:spec} follows from the following result.

\begin{lem}
We can assume that $h_{\top}(\alpha_{s,\gve} |_{\MT (F_\gve, \psi)}) >0$.
\end{lem}

\proof
Assume that $h_{\top}(\alpha_{s,\gve} |_{ \MT (F_\gve, \psi)}) =0$.
By Theorem~6.2 in \cite{New77} there exists a contact form 
$\alpha_{s,\gve}'$ on~$\MT (F_\gve,\psi)$ that 
is $C^1$-close to~$\alpha_{s,\gve}$ and equal to~$\alpha_{s,\gve}$ 
near the boundary, and whose Reeb flow has a generic 1-elliptic periodic orbit.
Hence this flow contains a hyperbolic basic set, and therefore $h_{\top} (\alpha_{s,\gve}') >0$.
Further, the $C^1$-closeness of $\alpha_{s,\gve}$ and~$\alpha_{s,\gve}'$
implies that all the 1-forms 
$$
(1-s) \2 \alpha_{s,\gve} + s \2 \alpha_{s,\gve}', \quad s \in [0,1],
$$
are contact forms.
Gray's stability theorem therefore shows that there exists a diffeomorphism~$\zeta$ 
of~$\MT (F_\gve, \psi)$ that is the identity near the boundary
such that the kernel of $\zeta^* \alpha_{s,\gve}'$ is~$\xi$.
The contact form on $(M,\xi)$ that agrees with $\alpha_{s,\gve}$ 
on~$M \setminus \MT (F_\gve, \psi)$ and with $\zeta^* \alpha_{s,\gve}'$ 
on~$\MT (F_\gve, \psi)$ is the contact form we were looking for.

Newhouse's full Theorem~6.2 starts with the $C^1$-closing lemma, and holds in all dimensions.
We do not need to appeal to the closing lemma in our situation, 
and we only need the easier 3-dimensional result:
Assume that $\dim M = 3$. 
By~\eqref{tau_sh_snearK} we can choose a flow-invariant neighbourhood $U \subset \MT(F_\gve, \psi)$
of the boundary of~$\MT (F_\gve, \psi)$ such that all orbits in~$U$ are closed.
Let $\gamma$ be one of these orbits that is not on the boundary of~$\MT(F_\gve,\psi)$,
and choose a flow-invariant open neighbourhood $N(\gamma)$ whose closure is also disjoint 
from the boundary of~$\MT(F_\gve,\psi)$.
Since $N(\gamma)$ is foliated by closed orbits, $\gamma$ is elliptic.
Using the Birkhoff normal form theorem and the KAM theorem, 
one can $C^1$-perturb $\alpha_{s,\gve}$ to a contact form~$\alpha_{s,\gve}'$ 
that agrees with $\alpha_{s,\gve}$ outside~$N(\gamma)$
and whose Reeb flow has a transverse homoclinic connection near~$\gamma$,
see~\cite{Zeh73}.
This Reeb flow therefore contains a horse-shoe and thus has positive topological entropy.
As above we can isotope $\alpha_{s,\gve}'$ without changing it outside $N(\gamma)$ 
to a contact form~$\alpha_{s,\gve}''$ for~$(M^3,\xi)$. 
Proceeding from this contact form with positive topological entropy, 
our inductive construction in Section~\ref{s:cfnear}
shows that $\sigma_\gve$ on $M^{2n+1}$ also has positive topological entropy.
Hence this also holds true for~$\alpha_{s,\gve}$ on the boundary of $\MT(F_\gve, \psi)$
and therefore, by the monotonicity of topological entropy, also on~$\MT (F_\gve, \psi)$.
\proofend

%%%%%%%%%%%%%%%%%%%%%%%%%%%%%%%%%%%%%%%%%%%%%%%%%%%%%%%%%%%%%%%%%%%%%%%%%%

\section{Collapsing the growth rate of symplectic invariants} \label{s:collapseSH}
Theorem~\ref{t:mainintro} on the collapse of topological entropy of Reeb flows implies the collapse 
of the growth rate of two symplectic invariants: 
symplectic homology and wrapped Floer homology.

\subsection{Liouville domains and fillings}
Recall that a Liouville domain is a compact exact symplectic manifold 
$\boldsymbol{W}=(W,\boldsymbol{\omega},\boldsymbol{\lambda})$ 
with boundary $\Sigma = \partial W$ and a primitive $\boldsymbol{\lambda}$ 
of~$\boldsymbol{\omega}$ such that $\boldsymbol{\alpha}_W = \boldsymbol{\lambda}|_{\Sigma}$ 
is a contact form on~$\Sigma$. 
The Liouville $1$-form $\boldsymbol{\lambda}$ also induces the contact structure 
$\xi_W = \ker \boldsymbol{\alpha}_W$ on~$\Sigma$.
The Liouville domain $\boldsymbol{W}$ is called an exact symplectic filling for the contact form $\boldsymbol{\alpha}_W$, and we say that the contact form~$\boldsymbol{\alpha}_W$ is exactly filled by~$\boldsymbol{W}$.

A standard construction (see for example \cite[Section 2.2.1]{AM19}) shows that 
if a contact form~$\alpha$ on a contact manifold~$(\Sigma,\xi)$ is exactly filled by a Liouville domain $\boldsymbol{W}_{\m2 \alpha}$, then we can construct for any other 
contact form~$\alpha'$ on~$(\Sigma,\xi)$ a Liouville domain $\boldsymbol{W}_{\m2 \alpha'}$ 
that fills~$\alpha'$. It therefore makes sense to say that a contact manifold is fillable 
by Liouville domains.
%We stress that once we have a filling $\boldsymbol{W}_{\m2 \alpha}$ for $\alpha$ the construction gives a unique filling $\boldsymbol{W}_{\alpha'}$ for each contact form $\alpha'$ on $(\Sigma,\xi)$. 
%In this way we obtain a family of Liouville domains 
%$$\mathcal{W}=\{\boldsymbol{W}_\alpha \ | \ \alpha \mbox{ is a contact form on } (\Sigma,\xi).\}$$
%One important property of this family of Liouville domains is that it has no preferred generator. 
%If we take any element $\boldsymbol{W}_\alpha}$ of $\mathcal{W}$ and perform the construction to generate Liouville domains filling contact forms on $(\Sigma,\xi)$ the family we obtain is exactly $\mathcal{W}$. We will call the family of Liouville domains obtained in this way a coherent family of Liouville domains.

\subsection{Symplectic homology and collapse of its exponential growth}

Let $\boldsymbol{W}_{\m2 \alpha}$ be a Liouville domain filling the contact form~$\alpha$ on the contact manifold $(\Sigma,\xi)$.
The symplectic homology $\SH (\boldsymbol{W}_{\m2 \alpha})$ is a homology theory associated 
to~$\boldsymbol{W}_{\m2 \alpha}$. While there are various versions of symplectic homology, 
we here consider the one originally developed by Viterbo~\cite{Vit99}.

Geometrically, one can think of the chain complex associated to $\SH (\boldsymbol{W}_{\m2 \alpha})$ as the $\mathbb{Z}_2$-vector space generated by the periodic orbits of the Reeb flow of~$\alpha$ and 
by the critical points of a $C^2$-small non-positive Morse--Smale function 
$f \colon \boldsymbol{W}_{\m2 \alpha} \to \mathbb{R}$ such that $f^{-1}(0) = \Sigma$ 
is a regular energy level. The differential of $\SH(\boldsymbol{W}_{\m2 \alpha})$ counts Floer cylinders connecting generators. 
We refer the reader to~\cite{BO09, Oan04} for details. 

There is a filtration of $\SH (\boldsymbol{W}_{\m2 \alpha})$ by the action of its generators. 
For the Reeb orbits the action equals the period. 
For each real number~$a>0$ let $\SH^{<a}(\boldsymbol{W}_{\m2 \alpha})$ be the homology of the subcomplex generated by 
the Reeb orbits of action~$<a$ and the critical points of~$f$.
The inclusion of this subcomplex induces the homomorphism of $\Z_2$-vector spaces
$$
\Psi^a_{\alpha} \colon \SH^{<a} (\boldsymbol{W}_{\m2 \alpha}) \to \SH (\boldsymbol{W}_{\m2 \alpha}). 
$$
We define the exponential growth rate of $\SH (\boldsymbol{W}_{\m2 \alpha})$ by 
$$
\Gamma ( \SH (\boldsymbol{W}_{\m2 \alpha})) := 
  \limsup_{a \to +\infty} \frac{\log \left( \rank (\Psi^a_{\alpha}) \right)}{a} .
$$ 
The following remarkable result is due to Meiwes \cite{Meiwes}.

\begin{theorem*}
If $\boldsymbol{W}_{\m2 \alpha}$ is a Liouville domain filling a contact form $\alpha$ on~$(\Sigma,\xi)$, then 
$$
h_{\top}(\phi_\alpha) \,\geq\, \Gamma( \SH (\boldsymbol{W}_{\m2 \alpha})).
$$
\end{theorem*}

Together with Theorem \ref{t:mainintro} we obtain the following result.

\begin{corollary} \label{co:sympcollapse}
Let $(\Sigma,\xi)$ be a contact manifold fillable by Liouville domains. 
Then for every $\varepsilon >0 $ there exists a contact form~$\alpha$ with $\vol_\alpha (\Sigma) = 1$ 
such that
$$
\Gamma( \SH( \boldsymbol{W}_{\m2 \alpha})) \leq \varepsilon,
$$
for any Liouville filling $\boldsymbol{W}_{\m2 \alpha}$  of $\alpha$.
\end{corollary}

It follows that one cannot, in general, recover the volume of a contact form~$\alpha$ 
from the exponential growth rate of $\SH(\boldsymbol{W}_{\m2 \alpha})$ 
of a Liouville filling~$\boldsymbol{W}_{\m2 \alpha}$. 
To obtain a better geometric formulation of the corollary, we notice that 
if $\boldsymbol{W}_{\m2 \alpha}=(W^{2n}_\alpha, \boldsymbol{\omega}_{\alpha}, \boldsymbol{\lambda}_\alpha)$ is a Liouville filling of~$\alpha$, 
then the symplectic volume 
$
        \int_{W_{\m2 \alpha}} (\boldsymbol{\omega}_\alpha)^n  
$
equals the contact volume of~$\vol_\alpha (\Sigma)$. 
Corollary~\ref{co:sympcollapse} thus says that 
every Liouville fillable contact manifold admits fillings by Liouville domains of symplectic volume~$1$
and arbitrarily small growth of symplectic homology. 

\smallskip
In the opposite direction, one can ask if for a fixed contact manifold~$(\Sigma,\xi)$ 
there exists a constant~$\mathrm{K}_{\Sigma,\xi}$ such that 
$$ 
\Gamma( \SH (\boldsymbol{W})) \,\leq\, \mathrm{K}_{\Sigma,\xi}
$$ 
for every Liouville domain $\boldsymbol{W}$ that fills some normalized contact form on~$(\Sigma,\xi)$. 
A partial negative answer to this question is given by the following result.
Recall that the spherization of a closed manifold~$Q$ is 
the contact manifold~$(S^*Q,\xi_{\can})$ whose Reeb flows comprise the co-geodesic flows 
of Riemannian metrics on~$Q$.

\begin{lem} \label{le:EK}
Let $Q_k$ be the closed orientable surface of genus $k \geq 2$. 
Then for every real number~$c \geq 2 \pi \sqrt{2(k-1)}$  
there exists a contact form~$\alpha$ of volume~$1$ 
on $(S^*Q_k, \xi_{\can})$ and a Liouville domain~$\boldsymbol{W}_{\m2 \alpha}$ filling~$\alpha$ such that 
$$ 
\Gamma( \SH (\boldsymbol{W}_{\m2 \alpha})) = c.
$$
\end{lem}

\proof 
It follows from \cite[Theorem A]{EK19} that for any 
real number~$c \geq 2 \pi \sqrt{2(k-1)}$ 
there exists a negatively curved Riemannian metric~$g$ with area~$1/(2\pi)$
such that the topological entropy of the geodesic flow~$\phi_g$ is equal to~$c$.
Let $\alpha$ be the contact form on $(S^*Q_k ,\xi_{\can})$ whose Reeb flow is the co-geodesic flow of~$g$. 
Then $\vol_\alpha (S^*Q_k) = 1$ and $h_{\top} (\phi_\alpha) = h_{\top}(\phi_g) = c$. 

Let $\boldsymbol{W}_{\m2 \alpha} = D^*(g) \subset (T^* Q_k, \lambda_{\can})$ be the unit co-disk bundle associated to the Riemannian metric~$g$, where $\lambda_{\can}$ is 
the Liouville form on the cotangent bundle~$T^*Q_k$. 
Then $\boldsymbol{W}_{\m2 \alpha}$ is a Liouville domain filling~$\alpha$. 

Since the Riemannian metric $g$ is negatively curved, 
a theorem of Margulis~\cite{Mar69} shows that 
%the Reeb flow $\phi_\alpha$ is an Anosov flow. Since $\phi_\alpha$ is an Anosov flow 
%on a $3$-dimensional manifold we know by the work of Bowen that 
$$
h_{\top}(\phi_\alpha) = \lim_{t \to +\infty} \frac{\log \left( P^t(\phi_\alpha) \right)}{t}, 
$$ 
where $P^t(\phi_\alpha)$ denotes the number of periodic orbits of the flow~$\phi_\alpha$ 
of length~$<t$,
see also~\cite{Bow71}.
Since all periodic Reeb orbits have Morse index zero and are non-contractible, 
there are no Floer cylinders starting or ending at these orbits.
Hence there is a bijection between the Reeb orbits of action~$<a$ and the generators 
of $\SH^a (\boldsymbol{W}_{\m2 \alpha})$, 
up to a finite error coming from the finitely many critical points of the function~$f$. 
It follows that 
$$ 
\Gamma( \SH (\boldsymbol{W}_{\alpha})) \,=\, 
       \lim_{t \to +\infty} \frac{\log \left( P^t(\phi_\alpha) \right)}{t}
$$
see \cite{MP12, A1} for details.
Combining these two equalities we get 
$$
\Gamma( \SH (\boldsymbol{W}_{\alpha})) \,=\, h_{\top}(\phi_\alpha) \,=\, c,
$$  
while as noted above $\vol_\alpha (S^*Q_k) = 1$.
\proofend

\subsection{Wrapped Floer homology and collapse of its exponential growth}
In a similar way we obtain a collapse result for the exponential growth of another symplectic invariant called wrapped Floer homology. The wrapped Floer homology~$\WH (\boldsymbol{W},L)$ 
is an invariant associated to a Liouville domain~$\boldsymbol{W}$ and an asymptotically conical exact Lagrangian submanifold~$L$ 
of~$\boldsymbol{W}$. 
One of several references giving the precise definition of~$\WH (\boldsymbol{W},L)$ 
is~\cite{AM19}.

With $\alpha$ the contact form on the boundary of 
a Liouville domain~$\boldsymbol{W}_{\m2 \alpha}$, 
one can think of the chain complex associated to~$\WH (\boldsymbol{W}_{\m2 \alpha},L)$ 
as the $\mathbb{Z}_2$-vector space generated 
by the Reeb chords of~$\alpha$ that start and end on~$\partial L$, 
and by the intersection points of~$L$ and a $C^2$-small perturbation of~$L$.
The differential of $\WH (\boldsymbol{W}_{\m2 \alpha},L)$ counts Floer strips connecting generators.

As in the case of symplectic homology there is a filtration of $\WH (\boldsymbol{W}_{\m2 \alpha},L)$ 
by the action of the generators, and again the action of Reeb chords is equal to their time. 
For each real number $a>0$ let $\WH^{<a}(\boldsymbol{W}_{\m2 \alpha},L)$ be the homology of the subcomplex generated by Reeb chords and intersection points of action~$<a$.
Again there are natural homomorphisms 
$$
\Psi^a_{\alpha} \colon \WH^{<a}(\boldsymbol{W}_{\m2 \alpha},L) \to \WH(\boldsymbol{W}_{\m2 \alpha},L) , 
$$
and we define the exponential growth rate of $\WH(\boldsymbol{W}_{\m2 \alpha},L)$ by 
$$
\Gamma (\WH (\boldsymbol{W}_{\m2 \alpha},L)) := 
        \limsup_{a \to +\infty} \frac{\log \left( \rank (\Psi^a_{\alpha}) \right)}{a}.
$$

The following result was obtained in~\cite{AM19}.

\begin{theorem*} 
Let $\boldsymbol{W}_{\m2 \alpha}$ be a Liouville domain filling a contact form~$\alpha$ 
on~$(\Sigma,\xi)$, and let $L$ be an asymptotically conical exact Lagrangian submanifold 
of~$\boldsymbol{W}_{\m2 \alpha}$ whose intersection with~$\partial \boldsymbol{W}_{\m2 \alpha}$ 
is a sphere.\footnote{The assumption that $L \cap \boldsymbol{W}_{\m2 \alpha}$ is a sphere has been removed in~\cite{Fender-Lee-Sohn} using the techniques introduced in~\cite{Cineli-Ginzburg-Gurel}. 
This can also be achieved with the methods developed by Meiwes in~\cite{Meiwes}.
}
Then
$$
h_{\top}(\phi_\alpha) \,\geq\, \Gamma( \WH(\boldsymbol{W}_{\m2 \alpha}), L).
$$
\end{theorem*}

Together with Theorem \ref{t:mainintro} we obtain

\begin{corollary} \label{co:wrappedcollapse}
Let $(\Sigma,\xi)$ be a contact manifold fillable by Liouville domains. 
Then for every $\varepsilon >0$ there exists a contact form~$\alpha$ with 
$\vol_\alpha (\Sigma) = 1$ such that
$$
\Gamma( \WH (\boldsymbol{W}_{\m2 \alpha},L)) \,\leq\, \varepsilon ,
$$
for any Liouville filling $\boldsymbol{W}_{\alpha}$ of $\alpha$ 
and any asymptotically conical Lagrangian submanifold~$L$ of~$\boldsymbol{W}_{\alpha}$ 
whose intersection with~$\partial \boldsymbol{W}_{\alpha}$ is a sphere.
\end{corollary}

Classical examples of pairs $(\boldsymbol{W}_{\m2 \alpha}, L)$
are the unit co-disk bundles~$D^*(g)$ over a closed Riemannian manifold~$Q$, 
with $L$ a co-disk $D_q^*(g)$ over a point~$q \in Q$.
Here is the analogue of Lemma~\ref{le:EK}.

\begin{lem}
Let $Q_k$ be the closed orientable surface of genus $k \geq 2$. 
Then for every real number~$c \geq 2 \pi \sqrt{2(k-1)}$ 
there exists a Riemannian metric~$g$ on~$Q_k$ 
of area~$1/(2\pi)$ such that 
$$ 
\Gamma( \WH (D^*(g), D_q^*(g)) \,=\, c.
$$
\end{lem}

\proof
As in the proof of Lemma~\ref{le:EK} we appeal to~\cite{EK19}
and take a negatively curved Riemannian metric~$g$ on~$Q_k$ of area~$1/(2\pi)$
whose geodesic flow has topological entropy 
$$
h_{\top}(g) \,=\, c .
$$
Since $g$ is negatively curved, Manning's inequality in Theorem~\ref{t:manning}
is an equality, 
$$
h_{\top} (g) \,=\, h_{\vol}(g).
$$
On the other hand, it is clear that
$$
h_{\vol} (g) \,=\, \lim_{t \to \infty} \frac{\log C^t(\phi_g,q)}{t} 
          \quad \forall \, q \in Q_k,
$$
where $C^t(\phi_g,q)$ denotes the number of geodesics from $q$~to~$q$ of length~$<t$.
Furthermore, since the Morse indices of all geodesics vanish, we have as 
for the symplectic homology that
$$
\Gamma( \WH (D^*(g), D_q^*(g)) \,=\, 
       \lim_{t \to \infty} \frac{\log \left( C^t(\phi_g,q) \right)}{t} . 
$$
These four identities prove the lemma.
\proofend

%%%%%%%%%%%%%%%%%%%%%%%%%%%%%%%%%%%%%%%%%%%%%%%%%%%%%%%%%%%%%%%%%%%%%%%%%%%

\appendix

\section{Volume entropy and Manning's inequality for Finsler metrics}
\label{a:Manning}

In this appendix we first give an elementary proof of Manning's inequality for regular Finsler geodesic flows.
We then use Yomdin's theorem to prove a generalization of Manning's inequality to arbitrary $C^\infty$-smooth
flows on compact fiber bundles whose fibers are of dimension one less than the base.

\subsection{The volume entropy of a Finsler metric} 
Let $F$ be a Finsler metric on the closed $n$-dimensional connected manifold~$Q$: 
$F$ is a continuous real function on the tangent bundle~$TQ$ 
which is positive away from the zero section, fiberwise positively one-homogeneous and fiberwise convex. The Finsler metric~$F$ induces the length functional
\[
\ell_F(\gamma) \,:=\, \int_a^b F(\dot{\gamma}(t))\, dt
\]
on the space of Lipschitz curves $\gamma \colon [a,b] \rightarrow Q$ 
and the function $d_F \colon Q \times Q \rightarrow [0,+\infty)$,
\[
d_F(x,y) \,:=\, 
\inf \left\{ \ell_F(\gamma) \mid \gamma \colon [0,1] \rightarrow Q 
        \mbox{ Lipschitz curve with } \gamma(0)=x, \; \gamma(1)= y \right\}.
\]
By the Arzel\`a--Ascoli theorem, the above infimum is actually a minimum: 
There is a Lipschitz curve~$\gamma$ from $x$ to~$y$ such that 
$\ell_F(\gamma) = d_F(x,y)$. 
The function $d_F$ is positive away from the diagonal and satisfies the 
triangle inequality
$$
d_F(x,z) \leq d_F(x,y) + d_F(y,z) \qquad \forall \, x,y,z \in Q.
$$
In general, $d_F$ is not symmetric because we are not assuming $F$ to be reversible, 
i.e.\ to satisfy $F(-v)=F(v)$ for every $v \in TQ$. 
By the compactness of $Q$, the irreversibility ratio of $F$, 
i.e.\ the number
$$
\theta \,:=\, \max_{\substack{v\in TQ \\ F(v)=1}} F(-v) \,\in\, [1,+\infty),
$$
is well-defined and we have
\begin{equation} \label{irre}
d_F(y,x) \leq \theta\,  d_F(x,y) \qquad \forall \, x,y \in Q.
\end{equation}
The Finsler metric $F$ lifts to a Finsler metric $\widetilde F$ 
on the universal cover~$\widetilde{Q}$ of~$Q$. 
We denote by $\ell_{\widetilde F}$ and $d_{\widetilde F}$ the induced length functional and asymmetric distance on~$\widetilde{Q}$. 
Note that the lifted Finsler metric has the same irreversibility ratio~$\theta$ 
and \eqref{irre} holds also for the lifted asymmetric distance $d_{\widetilde{F}}$. The closed forward $R$-ball centered at $x \in \widetilde{Q}$ is the set
$$
B_x(\widetilde{F},R) \,:=\, 
   \{ y \in \widetilde{Q} \mid d_{\widetilde F}(x,y) \leq R\},
$$
which is easily seen to be compact also when $\widetilde Q$ is not compact. 
The compactness of the forward balls and the Arzel\`a--Ascoli theorem imply the existence of a Lipschitz curve~$\gamma$ of $\widetilde F$-length 
$\ell_{\widetilde F}(\gamma) = d_{\widetilde{F}}(x,y)$ joining two arbitrary 
points $x$ and~$y$ on~$\widetilde Q$. 
This in turn implies the following characterization of forward balls:
\begin{equation} \label{caraball}
B_x(\widetilde{F},R) \,=\, 
\left\{ \gamma(R) \mid \gamma \colon [0,R] \rightarrow \widetilde{Q} 
  \mbox{ Lipschitz curve with } \gamma(0)=x \mbox{ and } 
	\widetilde F \circ \dot\gamma \leq 1 \mbox{ a.e.} \right\}.
\end{equation}

We fix an arbitrary Riemannian metric on $Q$, lift it to~$\widetilde{Q}$, 
and denote by $\mathrm{Vol}$ the induced volume 
(i.e.\ $n$-dimensional Hausdorff measure) of Borel subsets of~$\widetilde{Q}$. 
The volume entropy of~$F$ is the number
\begin{equation} \label{thelimit}
h_{\vol}(F) \,:=\, \lim_{R\rightarrow \infty} 
   \frac{1}{R} \log \Vol \bigl( B_x(\widetilde{F},R) \bigr) \,\in\, [0,+\infty) .
\end{equation}

\begin{prop} \label{p:manning}
The limit \eqref{thelimit} exists, is finite, and is independent 
of the point $x \in \widetilde Q$ and of the choice of the Riemannian metric on~$Q$. 
\end{prop}

\proof
The independence of the choice of the Riemannian metric on $Q$ is clear because by the compactness of $Q$ the volume $\mathrm{Vol}'$ induced by another Riemannian metric satisfies 
\[
\frac{1}{c} \mathrm{Vol} \,\leq\, \mathrm{Vol}' \,\leq\, c \; \mathrm{Vol} 
\]
for a suitable positive number $c$. Given $x\in \widetilde{Q}$ and $R\geq 0$ we abbreviate
\[
B_x(R) := B_x(\widetilde{F},R) \qquad \mbox{and} \qquad V_x(R):= \Vol (B_x(R)).
\]
Choose a closed fundamental domain $N$ in $\widetilde Q$ and set
\[
a := \max \bigl\{ d_{\widetilde F}(y,z) \mid y,z \in N \bigr\}.
\] 
By the triangle inequality we have
\begin{equation*} 
B_{x'}(R) \subset B_x(R+a) \qquad \forall \, R \geq 0, \quad \forall \,x,x' \in N ,
\end{equation*}
and hence
\begin{equation} \label{e:Vxr}
V_{x'}(R) \,\leq\, V_x(R+a) \qquad \forall \, R\geq 0, \quad \forall \,x,x' \in N .
\end{equation}
Let $\tau$ be a deck transformation of $\widetilde Q$.
Since $\tau$ preserves the $\widetilde F$-length of oriented curves,
$$
B_{\tau (z)}(R) = \tau (B_z(R))  \qquad \forall \, R\geq 0, \qquad \forall \,z \in \widetilde Q.
$$
Since $\tau$ also preserves $\mathrm{Vol}$,  we have
$$
V_{\tau (z)}(R) = V_z(R)  \qquad \forall \, R\geq 0, \qquad \forall \,z \in \widetilde Q.
$$
Applying this to deck transformations that bring points into $N$ we can upgrade the inequalities~\eqref{e:Vxr} 
to
\begin{equation} \label{e:Vxrall}
V_{x'}(R) \,\leq\, V_x(R+a) \qquad \forall \, R\geq 0, \quad \forall \,x,x' \in \widetilde Q .
\end{equation}
This implies that the limit \eqref{thelimit} is independent of $x$, if it exists.
Set 
\[
v \,:=\, \inf_{z \in \widetilde Q} V_z(1) \,=\, \min_{z \in N} V_z(1) \,>\,0.
\]
Fix $R>0$, and let $Y$ be a maximal subset of $B_x(R)$ such that the balls $B_y(1)$, $y\in Y$, are pairwise disjoint. If $z$ belongs to~$B_y(1)$ for some $y\in Y$ then
\[
d_{\widetilde F}(x,z) \,\leq\, d_{\widetilde F}(x,y) + d_{\widetilde F}(y,z) \,\leq\, R + 1.
\]
Therefore,
\[
\bigcup_{y\in Y} B_y(1) \,\subset\, B_x(R+1),
\]
and from the fact that the balls on the left-hand side are pairwise disjoint and from the definition of~$v$ 
we obtain
$$
(\card Y) \, v \,\leq\, \sum_{y \in Y} V_y(1) \,\leq\, V_x(R+1) 
$$
and so
\begin{equation} \label{e:cardY}
\card Y \,\leq\, \frac 1v \,V_x(R+1).
\end{equation}
From the maximality property of $Y$ we deduce that for every $z\in B_x(R)$ there is a point $y\in Y$ such that $B_z(1)$ intersects $B_y(1)$. If $w$ is a point in this non-empty intersection, we find
\[
d_{\widetilde F}(y,z) \leq d_{\widetilde F}(y,w) + d_{\widetilde F}(w,z) \leq d_{\widetilde F}(y,w) + \theta \, d_{\widetilde F} (z,w) \leq 1 + \theta,
\]
where $\theta$ is the irreversibility ratio of $F$ and we have used \eqref{irre}. This proves the inclusion
\begin{equation} \label{incl1}
B_x(R) \,\subset\, \bigcup_{y\in Y} B_y(1+\theta).
\end{equation}
Let $S\geq 0$ and $z$ be a point in $B_x(R+S)$. By \eqref{caraball}, 
$z=\gamma(R+S)$ where $\gamma \colon [0,R+S] \rightarrow \widetilde{Q}$ is a Lipschitz curve satisfying 
$\gamma(0)=x$ and $\widetilde F\circ \dot{\gamma}\leq 1$ a.e.. 
Then $\gamma(R)$ belongs to~$B_x(R)$, and from the above inclusion we find $y \in Y$ such that 
$d_{\widetilde F}(y,\gamma(R)) \leq 1+\theta$. With this,
\[
d_{\widetilde F}(y,z) = d_{\widetilde F}(y,\gamma(R+S)) \leq d_{\widetilde F}(y,\gamma(R)) + d_{\widetilde F}(\gamma(R),\gamma(R+S)) \leq 1 + \theta + S,
\]
and hence \eqref{incl1} can be upgraded to
\[
B_x(R+S) \,\subset\, \bigcup_{y\in Y} B_y(S+1+\theta) \qquad \forall \, S\geq 0.
\]
From this, together with \eqref{e:cardY} and \eqref{e:Vxr}, we obtain
$$
V_x(R+S) \,\leq\, (\card Y ) \, V_y ( S+1+\theta)  \,\leq\, 
  \frac 1v\, V_x(R+1) \,V_x ( S+1+\theta+a) \qquad \forall\, S\geq0.
$$
Abbreviating $b := 1+\theta +a$, this implies
$$
V_x(R+S) \,\leq\, \frac 1v\, V_x(R+b) \, V_x(S+b) \qquad \forall \, R,S\geq 0.
$$
Taking logarithms, we find
$$
\log V_x(R+S) \,\leq\,  \log V_x(R+b) + \log V_x(S+b) -\log v \quad \qquad \forall \, R,S> 0.
$$
This inequality implies that the function
\[
f \colon [0,+\infty) \rightarrow \R, \qquad 
f(R) := \log V_x(R+2b) - \log v,
\]
is subadditive, i.e.\ $f(R+S)\leq f(R) + f(S)$ for every $R,S\geq 0$. Therefore, $f(R)/R$ converges to 
its infimum for $R \rightarrow \infty$, which is a non-negative finite number because $f$ 
is monotonically increasing. From the identity
\[
\frac{\log V_x(R)}{R}  \,=\, \left( \frac{f(R-2b)}{R-2b}+\frac{\log v}{R-2b} \right) \cdot \frac{R-2b}{R}
\]
we deduce that the limit \eqref{thelimit} exists and is a non-negative finite number. 
Proposition~\ref{p:manning} is proven.
\proofend

\subsection{The topological entropy} \label{ss:htop} 
We shall use the following definition of the topological entropy 
$h_{\top}(\phi)$ of a continuous flow~$\phi^t$ on a compact metrizable space~$X$:
Choose a metric~$d$ which generates the topology of~$X$.
For $T >0$ define a new metric $d_T$ on~$X$ by
\begin{equation} \label{e:dynmet}
d_T(x,x') \,=\, \max_{0 \leq t \leq T} d \left( \phi^t(x), \phi^t(x') \right) .
\end{equation}
For $\delta >0$,
a subset $Y$ of~$X$ is called $(T,\delta)$-separated if 
$d_T(y,y') \geq \delta$ for all $y \neq y'$ in~$Y$.
Abbreviate
$$
\nu (T,\delta) \,:=\, \mbox{maximum cardinality of a $(T,\delta)$-separated subset of $X$} .
$$
Then the function $\delta \mapsto \nu(T,\delta)$ is monotonically decreasing for every $T>0$,
and one possible definition of $h_{\top}(\phi)$ is
\begin{equation} \label{def:htop}
h_{\top}(\phi) \,:=\, 
\lim_{\delta\rightarrow 0} \limsup_{T \to \infty} \frac 1T\, \log \nu(T,\delta)  \,=\,
\sup_{\delta >0} \, \limsup_{T \to \infty} \frac 1T\, \log \nu(T,\delta) \,\in\, [0, + \infty].
\end{equation}
This number represents the exponential growth rate of the number of orbit segments 
that can be distinguished with arbitrary fine but finite precision.
As the notation suggests, one obtains the same number if one starts with 
another metric which generates the topology of~$X$.
We refer to \cite[\S 3.1]{HK95} and~\cite[\S 7]{Wal82}
for this fact and for much information on topological entropy. 
We here only mention that for a Lipschitz-continuous flow on a compact manifold,
$h_{\top} (\phi)$ is finite, see~\cite[Theorem~3.2.9]{HK95}.

\subsection{Manning's inequality} 
We now assume the Finsler metric $F$ to be regular: $F$ is of class~$C^2$ away from the zero section 
and the second fiberwise differential of~$F^2$ is positive definite at every non-zero tangent vector. 
Under these assumptions, the Euler--Lagrange equations for the energy functional
\[
E_F(\gamma) \,:=\, \frac{1}{2} \int_0^1 F(\dot\gamma(t))^2\, dt
\]
define a well-posed second order Cauchy problem on $Q$. Solutions of this Cauchy problem are $C^2$-curves 
defined on the whole~$\R$ and are called Finsler $F$-geodesics. The function $F\circ \dot{\gamma}$ 
is constant for every Finsler $F$-geodesic~$\gamma$.  After an orientation preserving reparametrization 
making $F\circ \dot{\gamma}$ constant, any minimizer~$\gamma$ of~$\ell_F$ among Lipschitz curves 
from $x$ to~$y$ is an arc of a Finsler $F$-geodesic.

The Cauchy problem for Finsler $F$-geodesics defines a $C^1$-flow $\phi^t$ on the unit sphere bundle
\[
S Q \,:=\, \{v\in TQ \mid F(v)=1\}.
\]
Denoting by $\pi \colon S Q \rightarrow Q$ the footpoint projection, $\gamma(t):= \pi\circ \phi^t(v)$ 
is the unique Finsler $F$-geodesic starting at~$\pi(v)$ with velocity~$v$ for $t=0$, 
and $\phi^t(v) = \dot\gamma(t)$. 

We denote by $h_{\top}(F)$ the topological entropy of the flow $\phi^t$ defined above. 
Manning's celebrated inequality from~\cite{Man79} states that for any Riemannian metric~$F=\sqrt{g}$
on the closed manifold~$Q$, the topological entropy is bounded from below by the volume growth, 
$$
h_{\top}(g) \,\geq\, h_{\vol}(g),
$$
where we are writing the argument of both $h_{\top}$ and $h_{\vol}$ as $g$ instead of~$\sqrt{g}$, 
as this is the standard notation when dealing with Riemannian geodesic flows. 
It is well-known that this inequality persists to hold for Finsler metrics:

\begin{theorem} \label{t:manning}
Let $F$ be a regular Finsler metric on a closed connected manifold~$Q$.
Then the topological entropy of the Finsler geodesic flow of~$F$ is at least the volume entropy of~$F$:
$$
h_{\top} (F) \geq h_{\vol} (F) .
$$
\end{theorem}

The above theorem is stated for instance in \cite[Theorem~15]{Bar17} as well as 
\cite[Theorem~6.1]{Egl97} and \cite[Remarque~3.3]{Ver99}. 
Since there does not seem to be a proof in the literature, we give one here, 
by a straightforward adaptation of Manning's proof. 

We will use the following standard property of functions having exponential growth:

\begin{lem} \label{le:Ri}
Assume that the function $f \colon (0,+\infty) \rightarrow (0,+\infty)$ satisfies
\[
\lim_{R\rightarrow +\infty} \frac{1}{R} \log f(R) \,=\, h
\]
for some $h\in (0,+\infty)$. 
Then for every $\gve>0$ and every $\delta>0$ there exists a sequence $(R_n)\subset (0,+\infty)$  tending to $+\infty$
such that 
\[
\lim_{n\rightarrow \infty} \frac{f(R_n+\delta) - f(R_n)}{e^{(h-\gve)R_n}} \,=\, +\infty.
\]
\end{lem}

\begin{proof}
The assumption on $f$ is equivalent to
\begin{equation}
\label{ide} 
f(R) = e^{R(h+\sigma(R))} \qquad \forall \, R>0,
\end{equation}
where the function $R\mapsto \sigma(R)$ tends to zero for $R\rightarrow +\infty$.  We claim that
\[
L \,:=\, \limsup_{R \rightarrow +\infty} R \bigl( \sigma(R+\delta)-\sigma(R) \bigr) 
\]
belongs to $[0,+\infty]$. Indeed, if by contradiction $L$ is negative or $-\infty$, then we can find $R_0>0$ and $\mu>0$ such that
\[
R \bigl( \sigma(R+\delta)-\sigma(R) \bigr) \,\leq\, - \mu \qquad \forall \, R \geq R_0.
\]
This inequality can be rewritten as
\[
\sigma(R+\delta) \,\leq\, \sigma(R) - \frac{\mu}{R} \qquad \forall \, R \geq R_0,
\]
which implies
\[
\sigma(R_0+ n\delta) \,\leq\, \sigma(R_0) - \sum_{k=0}^{n-1} \frac{\mu}{R_0+ k \delta} \qquad \forall \, n \in \N.
\]
But then the divergence of the harmonic series implies that $\sigma(R_0+ n\delta)$ tends to $-\infty$, contradicting the fact that this sequence is infinitesimal.

Let $(R_n)$ be a sequence of positive numbers that diverges to $+\infty$ and such that
\[
\lim_{n\rightarrow \infty} R_n \bigl( \sigma(R_n+\delta)-\sigma(R_n) \bigr) \,=\, L.
\]
Then we have
\begin{equation} \label{versoeL}
\lim_{n\rightarrow \infty} e^{R_n \bigl( \sigma(R_n+\delta)-\sigma(R_n) \bigr)} \,=\, e^L ,
\end{equation}
where we are using the notation $e^{+\infty}:=+\infty$ in the case $L=+\infty$. 
The identity~\eqref{ide} and a simple algebraic manipulation produce the identity
\begin{equation} \label{newide}
\frac{f(R_n+\delta)-f(R_n)}{e^{(h-\gve) R_n}} \,=\, 
e^{R_n \bigl( \gve + \sigma(R_n) \bigr)} 
\left( e^{\delta \bigl( h+\sigma(R_n+\delta) \bigr)} 
           e^{R_n \bigl( \sigma(R_n+\delta) - \sigma(R_n) \bigr)} - 1 \right).
\end{equation}
By \eqref{versoeL} and the fact that $\sigma(R_n+\delta)$ is infinitesimal, the expression in brackets on the right-hand side of this identity tends to
\[
e^{\delta h} e^L - 1,
\]
which belongs to $(0,+\infty]$ because $\delta h>0$ and $L\in [0,+\infty]$. This, together with the positivity of $\gve$ and the fact that $\sigma(R_n)$ is infinitesimal, implies that the right-hand side of~\eqref{newide} tends to $+\infty$.
\end{proof}

\noindent
{\it Proof of Theorem~\ref{t:manning}.} 
We can assume that $h := h_{\vol}(F) > 0$.
Fix $\gve \in (0,h)$.
We must show that
\begin{equation*} %\label{e:hh}
h_{\top}(F) \,\geq\, h -\gve.
\end{equation*}
We use the following notation from the proof of Proposition~\ref{p:manning}: 
Given $x$ in the universal cover $\widetilde{Q}$ of~$Q$ and $R\geq 0$, $B_x(R):= B_x(R,\widetilde F)$ 
denotes the closed forward ball of radius~$R$ centered at~$x$ in~$\widetilde{Q}$, see~\eqref{caraball}, 
and $V_x(R):= \Vol (B_x(R))$ denotes its volume with respect to a Riemannian metric on~$Q$ which 
has been lifted 
to~$\widetilde{Q}$.

Fix $x_0 \in \widetilde Q$. By Proposition~\ref{p:manning}, there exists $R_0(\gve)$
such that
\begin{equation} \label{e:heR}
e^{(h+\gve)R} \,\geq\, V_{x_0}(R) \,\geq\, e^{(h-\gve)R} \qquad \forall\, R \geq R_0(\gve) .
\end{equation}
Fix some $\delta>0$. By Lemma \ref{le:Ri}, we find a diverging sequence of positive 
numbers $(R_n) \subset [R_0(\gve),+\infty)$ that satisfies
\begin{equation} \label{conclA3}
V_{x_0}(R_n+\delta/2)-V_{x_0}(R_n) \,\geq\, e^{(h-\gve) R_n} \qquad \forall \, n \in \N.
\end{equation}

The $F$-distance function $d_{\widetilde F}$ on~$\widetilde{Q}$ is not necessarily symmetric, because 
we are not assuming $F$ to be reversible, but we can symmetrize it and obtain the genuine distance function 
\[
\widetilde d \colon \widetilde{Q} \times \widetilde{Q} \rightarrow [0,+\infty), 
        \qquad \widetilde d(x,y):= d_{\widetilde F}(x,y) + d_{\widetilde F}(y,x).
\]
If $\theta\in [1,+\infty)$ denotes the irreversibility ratio of~$F$, then we have
\begin{equation}
\label{e:dzz}
\widetilde d(x,y) \,\leq\, (1+\theta) \, d_{\widetilde F}(x,y) \qquad \forall \, x,y \in \widetilde{Q}.
\end{equation}

For every $n \in \N$, we take a maximal subset $X_n$ of the annulus 
$B_{x_0}(R_n+ \delta/2) \setminus B_{x_0}(R_n)$
such that
\begin{equation} \label{e:dqq}
\widetilde d(x,x') \,>\, (2\theta +1)\, \delta  \qquad \forall\, x,x' \in X_n, \; x \neq x'.
\end{equation}
Since $X_n$ is maximal, for every point $z \in B_{x_0}(R_n+ \delta/2) \setminus B_{x_0}(R_n)$ 
there exists $x \in X_n$ such that $\widetilde d(x,z) \leq(2\theta+1)\, \delta$. 
A fortiori, $d_{\widetilde F}(x,z) \leq (2\theta+1) \2 \delta$ and hence 
$$
B_{x_0}(R_n+ \delta/2) \setminus B_{x_0}(R_n) \,\subset\, \bigcup_{x \in X_n} B_x \bigl( \delta (2\theta +1) \bigr).
$$
If we set 
\[
w \,:=\, \sup_{z \in \widetilde Q} V_z \bigr(\delta (2\theta +1) \bigr) > 0,
\]
we therefore find
\[
V_{x_0}(R_n+\delta/2) - V_{x_0}(R_n) \,=\, 
\Vol \bigl( B_{x_0}(R_n+ \delta/2) \setminus B_{x_0}(R_n) \bigr)  \leq\, (\card X_n)\, w
\]
and hence, by \eqref{conclA3},
\begin{equation} \label{e:Qna}
\card X_n \,\geq\, \frac 1w \, e^{(h-\gve) R_n}.
\end{equation}
Given $x\in X_n$, we denote by 
\[
\gamma_x \colon [0,d_{\widetilde F}(x_0,x)] \to \widetilde Q
\]
an $F$-geodesic segment from $x_0$ to $x$ realizing the minimal distance $d_{\widetilde F}(x_0,x)$ 
and such that $\widetilde F\circ \dot\gamma_x=1$.  
Since $X_n \subset B_{x_0}(R_n+\delta/2) \setminus B_{x_0}(R_n)$, we have 
\begin{equation} \label{e:Rnd}
R_n \,\leq\, d_{\widetilde F}(x_0,x) \,\leq\, R_n+\frac{\delta}{2}.
\end{equation}

Now comes the crux of the proof:
Consider the subset 
$$
Y_n := \left\{ \dot \gamma_x(0) \mid x \in X_n \right\} \subset S_{x_0} \widetilde Q
$$ 
of the $F$-unit sphere~$S_{x_0} \widetilde Q \subset T_{x_0} \widetilde Q$.
Choose any metric $\widetilde d_S$ on $S \widetilde Q$ such that the footpoint projection $\widetilde \pi \colon S \widetilde Q \to \widetilde Q$
is distance decreasing:
$$
\widetilde d_S \bigl( v,v' \bigr) \,\geq\, \widetilde d \2 \bigl( \widetilde \pi(v), \widetilde \pi(v') \bigr)  
               \qquad \forall \, v,v'\in S \widetilde{Q}.
$$
For instance, one can take 
$\widetilde d_S \bigl( v,v' \bigr) := 
       \widetilde d \2 \bigl( \widetilde \pi(v) ,\widetilde \pi(v') \bigr) + \overline{d} \2 \bigl( v,v')$,
where $\overline{d}$ is any metric on~$S \widetilde Q$.

\begin{lem} \label{le:Y}
The set $Y_n$ is $(R_n,\delta)$-separated for the Finsler geodesic flow $\widetilde\phi^t$ of $F$ on the metric space $(S \widetilde Q,\widetilde d_S)$,
and $\card Y_n = \card X_n$.
\end{lem}

\begin{proof}
We first show that the surjective map $X_n \to Y_n$, $x \mapsto \dot \gamma_x (0)$, that defines 
the set~$Y_n$, is a bijection.
Assume that $x,x' \in X_n$ give the same point $\dot \gamma_x(0) = \dot \gamma_{x'}(0)$ of~$Y_n$. We can assume that $d_{\widetilde F}(x_0,x') \leq d_{\widetilde F}(x_0,x)$. Then the fact that the vectors $\gamma_x'(0)$ and $\gamma_{x'}'(0)$ coincide implies that $\gamma_{x'}$ is the restriction of $\gamma_x$ to the interval $[0,d_{\widetilde F}(x_0,x')]$. The restriction of~$\gamma_x$ to the interval $[d_{\widetilde F}(x_0,x'),d_{\widetilde F}(x_0,x)]$ is a curve of  $\widetilde F$-length $d_{\widetilde F}(x_0,x) - d_{\widetilde F}(x_0,x')$ from $x'$ to~$x$, and by~\eqref{e:Rnd} we have
\[
d_{\widetilde F}(x,x') \,\leq\, d_{\widetilde F}(x_0,x) - d_{\widetilde F}(x_0,x') \,\leq\, \frac{\delta}{2}.
\]
Therefore, \eqref{e:dzz} implies 
$$
\widetilde d \2 (x',x) \,\leq\, (\theta +1)\, d_{\widetilde F}(x',x) \,\leq\, \frac{1}{2} \, (\theta +1)\, \delta .
$$
Our choice \eqref{e:dqq} now implies that $x=x'$.

By the definition of $(R_n,\delta)$-separated, for the first assertion it suffices to show that 
$$
\widetilde d_S \bigl( \widetilde\phi^{R_n}(y), \widetilde\phi^{R_n}(y') \bigr) \,\geq\, 
           \delta \qquad \forall\, y \neq y' \in Y_n.
$$
Let $x,x'$ be the points in $X_n$ with $y=\dot \gamma_x(0)$ and $y'=\dot \gamma_{x'}(0)$.
Then $x = \gamma_x (d_{\widetilde F}(x_0,x))$ and 
$\widetilde \pi \bigl( \widetilde\phi^{R_n}(y) \bigr) = \gamma_x(R_n)$, 
and so, by~\eqref{e:Rnd} again, 
\[
d_{\widetilde F}\bigl(\widetilde \pi (\widetilde \phi^{R_n}(y)) ,x\bigr) \,=\, d_{\widetilde F}(x_0,x) - R_n \,\leq\, \frac{\delta}{2}.
\]
In the same way, $d_{\widetilde F} \bigl( \widetilde \pi (\widetilde\phi^{R_n}(y')), x' \bigr) \leq \delta /2$.
Together with~\eqref{e:dqq}, we can now estimate
\begin{eqnarray*}
(2\theta +1)\,\delta \,\leq\, \widetilde d \2 (x,x') &\leq& \widetilde d \2 
\bigl( x, \widetilde{\pi} (\widetilde{\phi}^{R_n}(y)) \bigr) + \widetilde d \2 \bigl( \widetilde{\pi} (\widetilde{\phi}^{R_n}(y)),\widetilde{\pi}( \widetilde{\phi}^{R_n}(y')) \bigr) + 
                         \widetilde d \2 \bigl( \widetilde{\pi} (\widetilde{\phi}^{R_n}(y')),x' \bigr) \\
&\leq& \widetilde d \2 \bigl( \widetilde{\pi} ( \widetilde{\phi}^{R_n}(y)),\widetilde{\pi} (\widetilde{\phi}^{R_n}(y')) \bigr) + (\theta +1) \2 \delta .
\end{eqnarray*}
Hence 
\[
\widetilde d_S \bigl( \widetilde\phi^{R_n}(y), \widetilde\phi^{R_n}(y') \bigr) \,\geq\, 
\widetilde d \2 \bigl( \widetilde \pi (\widetilde\phi^{R_n}(y)),\widetilde \pi (\widetilde\phi^{R_n}(y')) \bigr)
\,\geq\, \theta \, \delta \,\geq\, \delta.
\]
The proof of Lemma~\ref{le:Y} is complete.
\end{proof}

We now consider the projection $\pr_S \colon S \widetilde Q \to S Q$ induced by the covering map
$\pr \colon \widetilde Q \to Q$ and we look at the set~$\pr_S (Y_n)$.
Since $Y_n \subset S_{x_0} \widetilde Q$, we have $\card \pr_S (Y_n) = \card Y_n$,
and so by the previous lemma and by~\eqref{e:Qna},
\begin{equation} \label{e:cardpr}
\card \pr_S (Y_n) \,\geq\, \frac 1w \, e^{(h-\gve) R_n}.
\end{equation}

\begin{figure}[h]   
 \begin{center}
  \psfrag{Q}{$X_n \subset \widetilde Q$} \psfrag{Md}{$(Q,d)$} 
  \psfrag{tMd}{$\bigl( \widetilde Q,\widetilde d \bigr)$} 
  \psfrag{TMd}{$(S Q, d_S) \supset \pr_S (Y_n)$}
  \psfrag{TtMd}{$\bigl( S\widetilde Q,\widetilde d_S \bigr)$}  \psfrag{pi}{$\widetilde \pi$}
  \psfrag{upi}{$\pi$} \psfrag{q}{$x$} 
  \psfrag{q'}{$x'$} \psfrag{x}{$x_0$}  \psfrag{gq}{$\dot \gamma_{x_0}(0)$}  
  \psfrag{y}{$y$} \psfrag{y'}{$y'$}   
  \psfrag{Y}{$Y_n \subset S_{x_0} \widetilde Q$}  
  \psfrag{pr}{$\pr$}  \psfrag{pr1}{$\pr_S$}  \psfrag{Phi}{$\Phi^t$} 
  \psfrag{phi}{$\phi^t$}  \psfrag{c}{{\huge $\curvearrowright$}} \psfrag{ci}{$\circ$}
  \leavevmode\includegraphics{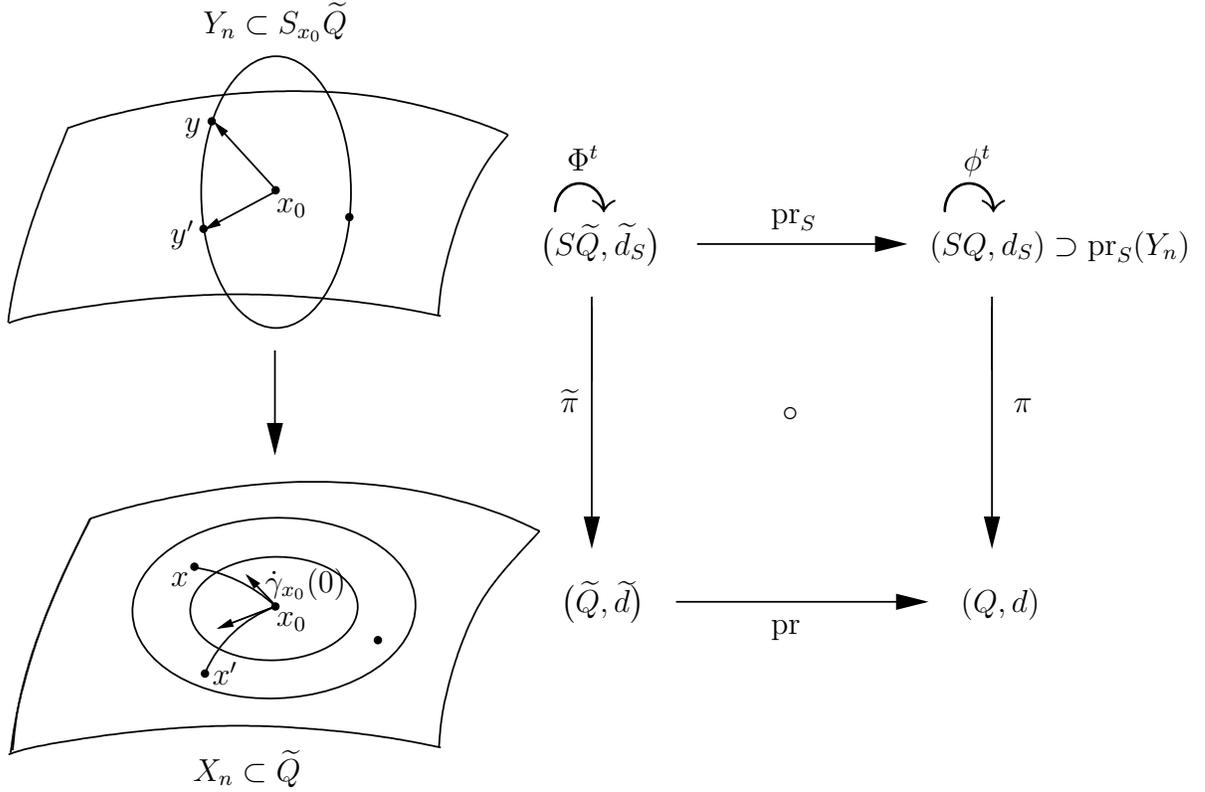}
 \end{center}
 \caption{Manning's construction}   \label{manning}
\end{figure}

As for $\widetilde Q$, we symmetrize the asymmetric distance $d_F$ on~$Q$ and obtain the genuine metric
$$
d(x,y) \,:=\, d_F(x,y) + d_F(y,x),
$$
and we choose a metric $d_S$ on $SQ$ such that the projection 
$\pi \colon (SQ, d_S) \to (Q,d)$
is distance decreasing. 
Note that the covering map $\pr \colon \bigl( \widetilde Q, \widetilde d \2 \bigr) \rightarrow (Q,d)$ is a local isometry. 
Therefore, there exists $\delta_1>0$ such that for all $z,z' \in \widetilde{Q}$ with 
$\widetilde d \2 (z,z') \leq \delta_1$ 
we have $d \bigl( \pr(z),\pr(z') \bigr) = \widetilde d(z,z')$.

\begin{lem} \label{l:deltasep}
For every $\delta >0$ smaller than $\delta_1$ the set $\pr_S (Y_n)$ is $(R_n,\delta)$-separated 
for the Finsler geodesic flow $\phi^t$ of~$F$ on~the metric space $(S Q,d_S)$.
\end{lem}

\begin{proof}
Let $v = \pr_S (y)$ and  $v' = \pr_S (y')$ be two different points of $\pr_S (Y_n)$.
Recall from the previous proof that 
$\widetilde d \2 \bigl( \widetilde \pi ( \widetilde \phi^{R_n}(y)),\widetilde \pi(\widetilde \phi^{R_n}(y')) \bigr) 
\geq \delta$.
This does not imply that the same lower bound holds for
$d \bigl( \pi (\phi^{R_n}(v)), \pi (\phi^{R_n}(v')) \bigr)$,
because $\pr$ is only a local isometry.

Since 
$\widetilde d \bigl( \widetilde \pi (\widetilde \phi^{R_n}(y)),\widetilde \pi (\widetilde \phi^{R_n}(y')) \bigr) \geq \delta$, we can look at
$$
t_\delta \,:=\, \min \left\{ t \in [0,R_n] \mid \widetilde d \2 \bigl( \widetilde \pi ( \widetilde \phi^t(y)) ,\widetilde \pi (\widetilde \phi^t(y')) \bigr) = \delta \right\} .
$$
Since $\pr_S \circ \widetilde \phi^t = \phi^t \circ \pr_S$ and $\pr \circ \widetilde \pi = \pi \circ \pr_S$, 
we have
$\pr \circ \widetilde \pi \circ \widetilde \phi^t = \pi \circ \phi^t \circ \pr_S$.
Hence, by the fact that $\delta<\delta_1$ and by the local isometry property of $\pr$,
$$
\delta =
\widetilde d \2 \bigl(\widetilde \pi (\widetilde \phi^{t_\delta}(y)), \widetilde \pi (\widetilde \phi^{t_\delta}(y')) \bigr) =
d \bigl(\pr \circ \widetilde \pi \circ \widetilde \phi^{t_\delta}(y), \pr \circ \widetilde \pi \circ \widetilde \phi^{t_\delta}(y') \bigr) =
d \bigl( \pi (\phi^{t_\delta}(v)), \pi (\phi^{t_\delta}(v')) \bigr) .
$$
For the dynamical metric $(d_S)_{R_n}$ on~$SQ$ defined in~\eqref{e:dynmet} we therefore find
$$
(d_S)_{R_n} (v, v') \,\geq\, 
d_S \bigl( (\phi^{t_\delta} (v), \phi^{t_\delta}(v') \bigr)) \,\geq\,
d \bigl( \pi (\phi^{t_\delta}(v)), \pi (\phi^{t_\delta}(v')) \bigr) = \delta .
$$
Hence $\pr_S (Y_n)$ is $(R_n,\delta)$-separated.
\end{proof}

Using a $\delta_0 < \delta_1$ as in Lemma~\ref{l:deltasep}, 
the sequence $R_n \to \infty$ satisfying~\eqref{conclA3} for $\delta=\delta_0$, 
and the estimate~\eqref{e:cardpr},
and noting that $w = \sup_{z \in \widetilde Q} V_z(\delta_0(2\theta +1))$ depends only on~$\delta_0$ 
but not on~$R_n$,
we can conclude:
\begin{eqnarray*}
h_{\top}(\phi^t) &=& \sup_{\delta >0} \, \limsup_{T \to \infty} \frac 1T \, \log \nu (T,\delta) \\
&\geq& \limsup_{T \to \infty} \frac 1T \, \log \nu (T,\delta_0) \\
&\geq& \limsup_{R_n \to \infty} \frac{1}{R_n} \log \left( \card \pr_S (Y_n) \right) \\
&\geq& h-\gve ,                       
\end{eqnarray*}
as we wished to show.
\finedim

\bigskip
\begin{rems}
{\rm
Manning's inequality $h_{\top}(g) \geq h_{\vol}(g)$ for Riemannian geodesic flows has several improvements:

\medskip
(1)
Let $\cm_g \subset SQ$ be the set of minimal vectors, namely those vectors $v$ for which the lifts to~$\widetilde{Q}$ of the geodesic 
determined by $v$ are shortest paths between any of their points. 
The set~$\cm_g$ is invariant under the geodesic flow.
Then $h_{\top} (\phi_g |_{\cm_g}) \geq h_{\vol}(g)$, 
see~\cite[Theorem~9.6.7]{HK95} and~\cite[Theorem~1.1]{GKOS14}.
This follows from a modification of Manning's proof of $h_{\top}(g) \geq h_{\vol}(g)$,
and the same modification of our proof above shows that this stronger inequality generalizes to 
regular Finsler geodesic flows:
$$
h_{\top} (\phi_F |_{\cm_F}) \geq h_{\vol}(F) .
$$

(2)
Manning also showed in~\cite{Man79} that $h_{\top}(g) = h_{\vol}(g)$ if $g$ has non-positive sectional curvature.
This hypothesis was weakened by Freire and Ma$\tilde{\text n}$\'e~\cite{FrMa82},
who only need to assume that the geodesic flow has no conjugate points.
Manning's equality extends to Finsler geodesic flows if $F$ has non-positive flag curvature, see~\cite[Theorem~15]{Bar17}.
On the other hand, it is not known if the improvement of Freire--Ma$\tilde{\text n}$\'e extends to Finsler geodesic flows. 
We thank Thomas Barthelm\'e for pointing out this open problem.
}
\end{rems}

\subsection{A proof via Yomdin's theorem} 
\label{ss:Yomdin}
Given a $C^1$ flow $\phi^t$ on a closed manifold~$M$, 
define the volume growth of~$\phi$ by
$$
v (\phi) \,:=\, 
\sup_S \limsup_{T \to \infty} \frac{1}{T} \log \mathcal{H} (\phi^T(S)) .
$$
Here the supremum is taken over all compact smooth
submanifolds~$S$ of~$M$ of any dimension and $\mathcal{H}(S)$ denotes the Riemannian volume of the submanifold $S$ with respect to the restriction 
of a fixed Riemannian metric on~$M$ or, equivalently, the $k$-dimensional Hausdorff measure $\mathcal{H}^k(S)$ of $S$ in the metric space $M$, where $k=\dim S$.
The Riemannian metric is not specified, since $v(\phi)$ does not depend on its choice. 
Yomdin proved in~\cite{Yom87} that if the flow~$\phi \colon \R \times M \rightarrow M$ is smooth, i.e.\ $C^{\infty}$, then
\begin{equation} \label{e:Yom}
h_{\top}(\phi) \,\geq\, v(\phi) . 
\end{equation}
Actually, equality holds in \eqref{e:Yom},
since by a result of Newhouse~\cite{New88},
$h_{\top}(\phi) \leq v(\phi)$ for $C^{1+\gve}$ flows.
Hence the volume growth $v(\phi)$ is another way to think of the topological entropy of smooth flows.

Gabriel Paternain noticed in \cite[p.\ 72]{Pat99} that Yomdin's theorem yields a quick proof of Mannings' inequality
for smooth Riemannian geodesic flows. We conclude this appendix by showing that a variant of his argument applies to 
general smooth flows on the total space of a fiber bundle and in particular implies Theorem~\ref{t:manning} 
in the case of smooth Finsler metrics.

Consider the following setting: $\phi^t$ is an arbitrary smooth flow on the total space~$E$ of a smooth fiber bundle 
\[
\pi \colon E \rightarrow Q
\]
over the closed $n$-dimensional manifold $Q$ with typical fiber a closed manifold of dimension $k<n$. 
The universal cover of~$Q$ is denoted by
\[
\mathrm{pr} \colon \widetilde{Q} \rightarrow Q,
\]
and the pull-back bundle of $E$ by the map $\mathrm{pr}$ is denoted by
\[
\widetilde \pi \colon \widetilde E \rightarrow \widetilde Q.
\]
Therefore, we have the commutative diagram of smooth maps
\begin{equation*} 
\xymatrixcolsep{3pc}
\xymatrix{ 
\widetilde E \ar[d]_{\widetilde \pi} \ar[r]^{\pr_{E}} &
E \ar[d]^{\pi} 
\\
\widetilde Q \ar[r]^-{\pr}  & Q 
}
\end{equation*}
where $\mathrm{pr}_E$ is a covering map. We choose Riemannian metrics on $Q$, $\widetilde Q$, $E$ and~$\widetilde E$
such that the horizontal maps are local isometries and the vertical projections
are Riemannian submersions. 
The $d$-dimensional Hausdorff measures induced by these metrics are denoted~$\ch^d$.
The flow $\phi^t$ on $E$ lifts to a smooth flow $\widetilde \phi^t$ on $\widetilde E$ such that
\[
\mathrm{pr}_E \circ \widetilde \phi^t \,=\, \phi^t \circ \mathrm{pr}_E.
\]
We shall prove the following result:

\begin{prop}
\label{p:proiezioni}
For any $\widetilde x_0$ in $\widetilde Q$ let 
$\widetilde E_{\widetilde x_0} := \widetilde{\pi}^{-1} ( \widetilde x_0)$ be the fiber of~$\widetilde E$ 
at~$\widetilde{x}_0$ and $E_{x_0}:= \pi^{-1}(x_0)$ the fiber of~$E$ 
at~$x_0:= \widetilde{\pi}(  \widetilde x_0 )$. 
Then 
\begin{equation} \label{2limsup}
\limsup_{T\rightarrow \infty} \frac{1}{T} 
  \log \mathcal{H}^{k+1} \bigl(\widetilde{\pi} (\widetilde{\phi} ([0,T] \times \widetilde{E}_{\widetilde x_0} ))\bigr) \,\leq\, 
\max \left\{ 0, \limsup_{T\rightarrow \infty} \frac{1}{T} \log  \mathcal{H}^k\bigl(\phi^T(E_{x_0}) \bigr) \right\},
\end{equation}
where on the left-hand side we are using the convention $\log 0:= - \infty$.
\end{prop}

Postponing the proof, we specialize to the case $k+1 = n$,
and for $x_0 \in Q$ define the volume entropy of~$\phi$ by the left-hand side of~\eqref{2limsup}:
\begin{equation} \label{def:hvolgen}
h_{\vol} (\phi;x_0) \,:=\, \limsup_{T\rightarrow \infty} \frac{1}{T} 
  \log \mathcal{H}^n \bigl(\widetilde{\pi} (\widetilde{\phi} ([0,T] \times \widetilde{E}_{\widetilde x_0} ))\bigr) .
\end{equation}
Note that the right-hand side indeed does not dependent on the lift $\widetilde x_0$ of~$x_0$
nor on the Riemannian metric used to define the Hausforff measure.
The set $\widetilde{\pi} (\widetilde{\phi} ([0,T] \times \widetilde{E}_{\widetilde x_0} ))$
is the set of points $\widetilde x \in \widetilde Q$
for which the fiber $\widetilde E_{\widetilde x}$ can be reached in time~$\leq T$
by a $\widetilde \phi$-flow line starting at $\widetilde E_{\widetilde x_0}$.
In the case of a Finsler geodesic flow, when $E$ is an $S^{n-1}$-bundle, this set is the forward ball
$B_{\widetilde x_0}(\widetilde F, T)$.
Hence \eqref{def:hvolgen} generalizes \eqref{def:hvol}. 

The right-hand side of~\eqref{2limsup} is, by definition, 
a lower bound for the volume growth~$v(\phi)$ of~$\phi$. 
Together with Yomdin's inequality~\eqref{e:Yom} we obtain
\[
 h_{\vol}(\phi;x_0) \,\leq\, v(\phi) \,\leq\, h_{\top}(\phi).
\]
We have shown the following result.

\begin{theorem} \label{t:manninggen}
Assume that $\phi$ is a $C^\infty$-smooth flow on the compact fiber bundle~$E$ over~$Q$
with fibers of dimension $\dim Q -1$.
Then 
$$
h_{\top} (\phi) \,\geq\, h_{\vol}(\phi;x) \qquad \forall \, x \in Q .
$$
\end{theorem}

In the case of Finsler geodesic flows, Theorem \ref{t:manninggen} together with Proposition~\ref{p:manning}
imply Theorem~\ref{t:manning}.
This proof is less satisfactory than the one in the previous paragraphs, however:
It needs the Finsler geodesic flow to be $C^\infty$-smooth and it is less elementary, 
as it relies on Yomdin's theorem whose proof is highly non-trivial. 
On the other hand, this proof shows that the special features of Finsler geodesic flows 
which are given by the underlying length functional and the triangle inequality 
are needed only to guarantee that the limit defining 
the volume entropy exists and is independent of the center of the balls, 
whereas for a general flow on the total space of a fiber bundle 
the limit superior in definition~\eqref{def:hvolgen} 
cannot be replaced by a limit and may depend on the choice of~$x_0$. 
These special features of a Finsler geodesic flow are instead not needed for proving that the volume entropy  
does not exceed the topological entropy.

Theorem~\ref{t:manninggen} in particular applies to Reeb flows on spherizations~$S^*Q$
and to magnetic flows on~$SQ$.
In the latter situation, Theorem~\ref{t:manninggen} improves the first statement in~\cite[Theorem~D]{BP02}.

\begin{proof}[Proof of Proposition~\ref{p:proiezioni}]
As it intertwines the flows $\widetilde \phi^t$ and~$\phi^t$, 
the projection~$\pr_E$ maps 
$\widetilde \phi^T(\widetilde E_{\widetilde x_0})$ bijectively to $\phi^T (E_{x_0})$
for every~$T$, and since $\pr_E$ is a local isometry we have
\begin{eqnarray} 
\label{e:TT} 
\mathcal{H}^k\bigl( \widetilde \phi^T(\widetilde E_{\widetilde x_0}) \bigr) 
\,=\, 
\mathcal{H}^k \bigl( \phi^T (E_{x_0}) \bigr) \qquad \forall \, T \in \R.
\end{eqnarray}
By the area formula, we have
\begin{equation} 
\label{area}
\mathcal{H}^{k+1} 
\bigl(\widetilde{\phi} ([0,T] \times \widetilde{E}_{\widetilde x_0} )\bigr) 
\,\leq\, 
\int_{[0,T] \times \widetilde{E}_{\widetilde x_0}} J\psi(t,\xi) \, dt \, d\xi
\end{equation}
where 
$J\psi(t,\xi) := \det \bigl( d\psi(t,\xi)^T d \psi (t,\xi) \bigr)^{\frac{1}{2}}$
is the Jacobian of the map
\[
\psi \colon [0,T] \times \widetilde{E}_{\widetilde x_0} \rightarrow \widetilde{E}, \qquad \psi(t,\xi) := \widetilde \phi^t(\xi),
\]
and $d\xi$ denotes the Riemannian volume form on the compact manifold 
$\widetilde E_{\widetilde x_0}$. If $X$ denotes the vector field on~$\widetilde{E}$ generating the flow~$\widetilde \phi^t$ and $\psi^t$ is the map
\[
\psi^t \colon \widetilde{E}_{\widetilde x_0} \rightarrow \widetilde{E}, 
                \qquad \psi^t(\xi) := \psi(t,\xi),
\]
we easily see that the symmetric linear endomorphism
\[
d\psi(t,\xi)^T d \psi (t,\xi) \colon 
     \R \times T_{\xi} \widetilde{E}_{\widetilde x_0} \rightarrow \R \times T_{\xi}   \widetilde{E}_{\widetilde x_0}
\]
has the form
\[
d\psi(t,\xi)^T d \psi (t,\xi) \,=\, 
\left( \begin{array}{cc} |X(\psi(t,\xi))|^2 & * \\ * & d\psi^t(\xi)^T d\psi^t(\xi) \end{array} \right).
\]
By the Schur determinant identity, this implies the bound
\[
J\psi(t,\xi) \,=\, \det \bigl( d\psi(t,\xi)^T d \psi (t,\xi) \bigr)^{\frac{1}{2}} 
\,\leq\, 
|X(\psi(t,\xi))| \det \bigl( d\psi^t(\xi)^T d \psi^t(\xi) \bigr)^{\frac{1}{2}} 
\,\leq\, c\, J\psi^t(\xi),
\]
where $c$ is the supremum norm of the vector field~$X$, 
which is bounded being the lift of a vector field on the compact manifold~$E$. 
Together with~\eqref{area} and using the area formula for injective maps, 
we find
\[
\mathcal{H}^{k+1} \bigl( \widetilde{\phi} ([0,T] \times \widetilde{E}_{\widetilde x_0} ) \bigr) \,\leq\, 
c \int_0^T \int_{\widetilde{E}_{\widetilde x_0}} J\psi^t(\xi)\, d\xi \, dt \,=\, 
c \int_0^T \mathcal{H}^k \bigl( \widetilde \phi^t ( \widetilde E_{\widetilde x_0}) \bigr)\, dt,
\]
and hence
\[
\limsup_{T\rightarrow \infty} \frac{1}{T} \log \mathcal{H}^{k+1} \bigl(\widetilde{\phi} ([0,T] \times \widetilde{E}_{\widetilde x_0} )\bigr) \,\leq\, \limsup_{T\rightarrow \infty} \frac{1}{T} \log \int_0^T \mathcal{H}^k\bigl(\widetilde \phi^t ( \widetilde E_{\widetilde x_0}) \bigr)\, dt.
\]
It is easy to see that if $f$ is a positive function on $[0,+\infty)$ then
\[
\limsup_{T\rightarrow \infty} \frac{1}{T} \log \int_0^T f(t)\, dt  \,\leq\, 
\max \left\{ 0, \limsup_{T\rightarrow \infty} \frac{1}{T} \log f(T)\right\},
\]
see \cite[Lemma 3.24.2]{Pat99}, so using also~\eqref{e:TT} we obtain the bound
\[
\limsup_{T\rightarrow \infty} \frac{1}{T} \log \mathcal{H}^{k+1} \bigl(\widetilde{\phi} ([0,T] \times \widetilde{E}_{\widetilde x_0} )\bigr) \,\leq\, 
\max \left\{ 0, \limsup_{T\rightarrow \infty} \frac{1}{T} \log  \mathcal{H}^k\bigl(\phi^T ( E_{x_0}) \bigr)\right\}.
\]
The desired bound \eqref{2limsup} now follows from the fact that the projection 
$\widetilde \pi$ is $1$-Lipschitz, by our choice of the Riemannian metrics.
\end{proof}

\begin{rem}
{\rm
If one replaces the inequality \eqref{area} in the above proof by the equality
$$
\int_{\widetilde E} \mathcal{H}^0(\psi^{-1}(\{\xi\}) \, d\mathcal{H}^{k+1}(\xi) \,=\,  
\int_{[0,T] \times \widetilde{E}_{\widetilde x_0}} J\psi(t,\xi) \, dt \, d\xi
$$
which is given by the area formula, one gets the more precise bound
$$
\limsup_{T\rightarrow \infty} \frac{1}{T} 
      \log \int_{\widetilde E} \mathcal{H}^0(\psi^{-1}(\{\xi\}) \, d\mathcal{H}^{k+1}(\xi) 
			\,\leq\, 
\max \left\{ 0, \limsup_{T\rightarrow \infty} \frac{1}{T} \log  \mathcal{H}^k\bigl(\phi^T(E_{x_0}) \bigr) \right\}.
$$

For $k+1=n$, this inequality and the argument leading to Theorem~\ref{t:manninggen},  
applied to~$Q$ instead of the universal cover~$\widetilde Q$, yield the bound
\begin{equation} \label{e:chords}
\limsup_{T\rightarrow \infty} \frac{1}{T}\log  \int_Q n_T(x_0,x)\, d\mathcal{H}^n(x) \,\leq\, h_{\top}(\phi),
\end{equation}
where $n_T(x_0,x)$ denotes the number of $\phi$-flow lines of time at most~$T$ from $E_{x_0}$ to~$E_x$. 
For Finsler geodesic flows, $n_T(x_0,x)$ is the number of $F$-geodesics arcs from $x_0$ to~$x$, and \eqref{e:chords} 
is a Finsler generalization of~\cite[Corollary 3.28]{Pat99}.
}
\end{rem}

\section{From Riemannian geodesic flows to Reeb flows} \label{s:GeotoReeb}

In this appendix we first recall from a historical and geometric perspective how Finsler and Reeb flows
are successive generalizations of Riemannian geodesic flows.
We then give for each of the four circles below at least two results on Riemannian geodesic flows that stop to hold
exactly at 
\raisebox{.5pt}{\textcircled{\raisebox{-.9pt} {1}}} or at
\raisebox{.5pt}{\textcircled{\raisebox{-.9pt} {2}}} or at
\raisebox{.5pt}{\textcircled{\raisebox{-.9pt} {3}}},
or extends all the way to Reeb flows (\raisebox{.5pt}{\textcircled{\raisebox{-.9pt} {4}}}).
\begin{equation} \label{e:1234}
\stackrel{\raisebox{.5pt}{\tiny \textcircled{\raisebox{-.9pt} {4}}}}{\phantom{\supsetneq}}\,
\{\mbox{Reeb}\}
\,\stackrel{\raisebox{.5pt}{\tiny \textcircled{\raisebox{-.9pt} {3}}}}{\supsetneq}\,
\{\mbox{irreversible Finsler}\}
 \,\stackrel{\raisebox{.5pt}{\tiny \textcircled{\raisebox{-.9pt} {2}}}}{\supsetneq}\,
\{\mbox{reversible Finsler}\}
\,\stackrel{\raisebox{.5pt}{\tiny \textcircled{\raisebox{-.9pt} {1}}}}{\supsetneq}\,
\{\mbox{geodesic}\}
\end{equation}

\subsection{Reeb flows on spherizations as a generalization of Riemannian geodesic flows}
\label{ss:RSG}

Fix an $n$-dimensional manifold $Q$.
Consider a smooth star field along~$Q$: 
At every point $q \in Q$ there is a set $S_qQ \subset T_qQ$
that is the smooth boundary of a subset~$D_qQ$ of~$T_qQ$ that is strictly starshaped 
with respect to the origin of~$T_qQ$, and $S_qQ$ varies smoothly with~$q$.
The star field $\{S_qQ\}_{q \in Q}$ can be used to define the length of oriented curves in~$Q$:
For a smooth curve $\gamma \colon [a,b] \to Q$ with non-vanishing derivative, set
$$
\length (\gamma) \,:=\, \int_a^b \ell (\dot \gamma (t)) \,dt
$$
where $\ell (\dot \gamma (t)) = s$ if $\frac 1s \1 \dot \gamma (t) \in S_{\gamma (t)}Q$.
The number $\length (\gamma)$ does not change under orientation preserving reparametrisations of~$\gamma$.
Given $q,q' \in Q$ set 
$$
d_S(q,q') \,:=\, \inf \left\{ \length (\gamma) \right\}
$$
where the infimum is taken over all curves as above from $q$ to~$q'$.
The function $d_S$ is non-degenerate in the sense that $d_S (q,q') = 0$
if and only if $q=q'$.
Furthermore, $d_S$ satisfies the ordered triangle inequality
$$
d_S(q,q'') \,\leq\, d_S(q,q') + d_S(q',q'') \quad \forall \, q,q',q'' \in Q. 
$$
The function $d_S$ is symmetric if and only if each star~$S_qQ$ is symmetric, that is $-S_qQ=S_qQ$ for all $q \in Q$. 
If each $S_qQ$ is strictly convex, then there are unique shortest curves between sufficiently nearby points, 
see e.g.\ \cite[\S 6.3]{BCS00}.

A star field $\{S_qQ\}$ as above is called a Finsler structure if each $D_q Q$ is strictly convex,
and a Finsler structure is called reversible if each~$S_q Q$ is symmetric, and irreversible otherwise.
A reversible Finsler structure is a Riemannian structure if each~$S_q Q$ is an ellipsoid, i.e., 
the level set of an inner product on~$T_qQ$.

Riemannian structures were introduced in 1854 by Riemann in his Habilitationsvortrag~\cite{Rie}.
Both Berger~\cite[p.\ 708]{Ber03} and Chern~\cite{Che96}
pointed out that what Riemann really had in mind are Finsler structures.
However, from his text\footnote{In \S II.1 of his text, page 259 of~\cite{Rie}, he writes:
``Unter diesen Annahmen wird das Linienelement eine beliebige homogene Function
ersten Grades der Gr\"ossen $dx$ sein k\"onnen, welche unge\"andert bleibt, 
wenn s\"ammtliche Gr\"ossen~$dx$ ihr Zeichen \"andern''.
He then readily specializes to the convex (Finsler) case, mentioning just an example, 
and then restricts to Riemannian metrics, since  
``die Untersuchung dieser allgemeinern Gattung
w\"urde zwar keine wesentlich andere Principien erfordern, aber ziemlich zeitraubend sein''.}
and given that Riemann's main conceptual point was to do intrinsic measurements, 
that were all based
on length measurements, one may at least as well argue that what he meant is
``reversible star field geometry'', 
see also Spivak~\cite[p.\ 167 and p.\ 202]{Spi79}.
In contrast to Riemannian geometry, 
Finsler geometry developed only slowly, see~\cite{BCS00},
and ``star field geometry'' 
%did not develop at all. 
developed only in the setting of Lorentzian and semi-Riemannian geometry 
(and their Finsler generalizations), 
in which however the stars~$S_q Q$ are not compact.

If each $D_qQ$ is strictly convex, one can pass from $TQ$ to~$T^*Q$ by the Legendre transform. 
In geometric terms, each convex body $D_q Q \subset T_qQ$ is replaced by its dual body
$$
D_q^*Q \,:=\, \left\{ (q,p) \in T_q^*Q \mid p(v) \leq 1 \mbox{ for all } v \in D_qQ \right\}
$$
in $T_q^*Q$.
The dual body $D_q^*Q$ is strictly convex, or is symmetric, or is an ellipsoid, if and only if 
$D_qQ$ has this property, cf.\ Figure~\ref{dual.fig}.

\begin{figure}[h]   
 \begin{center}
  \psfrag{L}{$\mathcal{L}$} \psfrag{N}{\mbox{NO !}}        
  \leavevmode\includegraphics{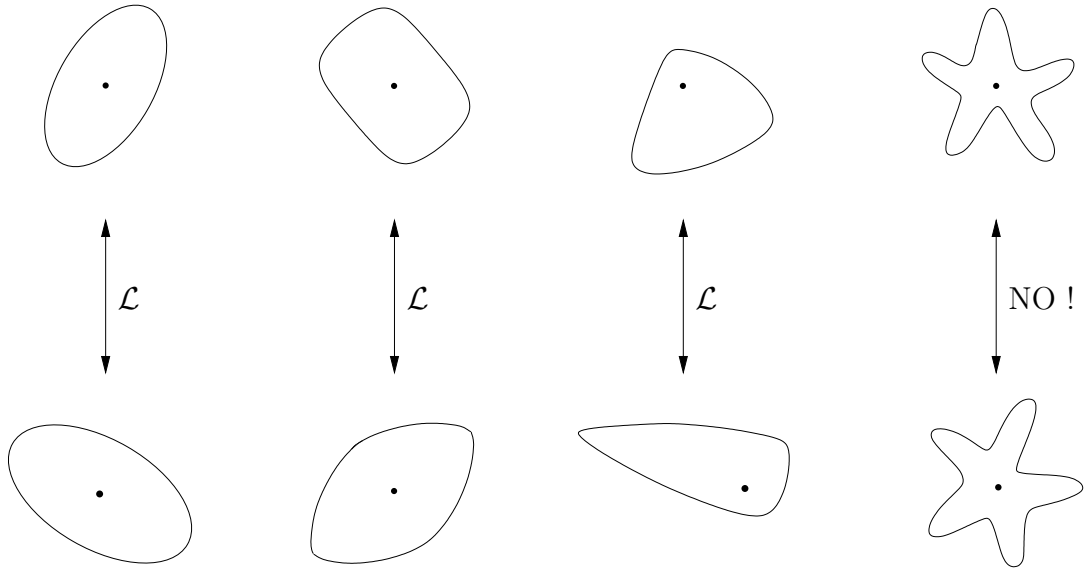}
 \end{center}
 \caption{Stars in $T_qQ$ and $T_q^*Q$}   \label{dual.fig}
\end{figure}

In more dynamical terms, we associate to the field of strictly convex disks~$\{D_qQ\}$
and its dual field $\{D_q^*Q\}$ the functions 
$$
F \colon TQ \to \R, \qquad F^* \colon T^*Q \to \R
$$
that are homogenous of degree one in each fiber and satisfy
$$
F^{-1}(1) = SQ, \qquad (F^*)^{-1} (1) = S^*Q .
$$
The Legendre transform
$$
\mathcal{L} \colon TQ \to T^*Q, \quad (q,v) \mapsto 
\left( q,  \partial_v \bigl( \tfrac 12 F^2(q,v) \bigr) \right)
$$
is a diffeomorphism that maps fibers to fibers and $D_qQ$ to $D_q^*Q$.
This diffeomorphism conjugates the Finsler geodesic flow of~$F$ on~$SQ$
with the Hamiltonian flow of $F^*$ restricted to~$S^*Q$, 
$$
\mathcal{L} \circ \phi_F^t (q,v) \,=\, \phi_{F^*}^t \circ \mathcal{L} (q,v) \quad 
                  \forall \, t \in \R, \; \forall \, (q,v) \in SQ .
$$
For strictly convex starfields, one can therefore freely switch between 
Finsler geodesic flows on tangent
bundles and co-Finsler geodesic flows on cotangent bundles.

For non-convex star fields $\{ S_qQ \}$ along $Q$ there is no Legendre transform. 
However, a smooth star field~$\{ S^*_qQ \}$ in~$T^*Q$ is the same thing as a Reeb flow on 
the spherization~$S^*Q$ of~$Q$, as we shall recall below.
We conclude that while Riemann's concept of a star field geometry in~$TQ$ led to nothing, 
the same picture in~$T^*Q$ describes the main example of contact geometry and Reeb flows,
a by now huge and thriving theory!
This is one more instance for the fact that it is always worthwhile and often crucial
to work in~$T^*Q$ instead of~$TQ$.

\bigskip \noindent
{\bf The flows $\phi_H$ on $S^*(H)$ are Reeb flows.}
For every function $H \colon T^* Q \to \R$ that is fiberwise homogenous of degree one
and smooth and positive off the zero-section we consider
the regular hypersurface $S^*(H) = H^{-1}(1)$ in~$T^*Q$
and the restriction $\phi_H^t$ of the Hamiltonian flow of~$H$ to~$S^*(H)$.
The flows $\phi_H^t$ live on different spaces. 
To have a class of flows on one manifold, 
we consider the spherization, or positive projectivization, of the cotangent bundle
$$
S^*Q \,:=\, \left( T^*Q \setminus Q \right) / \sim \quad 
              \mbox{ where $(q,p) \sim (q,sp)$ for $s>0$.}
$$
While the 1-form $\lambda = \sum_j p_j \1 dq_j$ on $T^*Q$ does not descend to this quotient, 
the kernel of~$\lambda$ does descend. 
The resulting hyperplane field $\xi_{\can}$ is the canonical contact structure on~$S^*Q$.
If for every function $H$ as above we abbreviate $\lambda_H = \lambda |_{S^* (H)}$ and $\xi_H = \ker (\lambda_H)$, 
we have that $(S^*(H),\xi_H)$ is diffeomorphic to $(S^*Q, \xi_{\can})$ under the map $(q,p) \mapsto [(q,p)]$.

We wish to show that the set of flows $\phi_H^t$ is in bijection with the set of Reeb flows
on $(S^*Q, \xi_{\can})$.
For this we first recall that $\phi_H^t$ is the Reeb flow of the contact form~$\lambda_H$ on~$(S^*(H), \xi_H)$,
see \cite[Lemma~4.2]{FLS13} for the short proof.
To identify the Reeb flows of the contact forms~$\lambda_H$ with the Reeb flows on~$(S^*Q, \xi_{\can})$,
we fix a representative $(S^*(H_0), \xi_{H_0})$ of $(S^*Q, \xi_{\can})$,
and for an arbitrary $H$ consider the diffeomorphism $\Psi_H \colon S^*(H) \to S^*(H_0)$
given by fiberwise radial projection, 
$$
\Psi_H (q,p) \,=\, \left( q, \frac{H(q,p)}{H_0(q,p)} \,p \right) .
$$
A computation shows that the differential of $\Psi_H$ takes the Reeb vector field of~$\lambda_H$
to the Reeb vector field of $\frac{H_0}{H} \2 \lambda_{H_0}$.
Therefore, the map $\Psi_H$ conjugates the Reeb flows of $\lambda_H$ and of $\frac{H_0}{H} \2 \lambda_{H_0}$.
In conclusion, each Reeb flow on $(S^*Q,\xi_{\can})$ corresponds to the Reeb flow of $f  \lambda_{H_0}$ 
on $(S^*(H_0), \xi_{H_0})$
for a positive smooth function~$f$ on $S^*(H_0)$,
and this Reeb flow is conjugate to the Hamiltonian flow $\phi_H^t$ on~$S^*(H)$, where $H = H_0 / \tilde f$
and $\tilde f$ is the extension of $f$ to~$T^*Q \setminus Q$ that is fiberwise homogenous of degree zero. 
 
For a Hamiltonian $H$ as above, the Holmes--Thompson volume is defined as
$$
\vol_H^{\HT}(Q) \,=\, \frac{1}{n! \, \omega_n} \int_{D^*(H)} \omega^n ,
$$
and for a contact form $\alpha$ on $S^*Q$ we defined the contact volume 
in \S \ref{ss:collapse-intro} as
$$
\vol_\alpha (S^*Q) \,=\, \int_{S^*Q} \alpha \wedge (d\alpha)^{n-1} .
$$
Under the above identification of $(S^*(H), \lambda_H)$ with $(S^*Q, \alpha)$
we obtain, using $\lambda_H = \lambda |_{S^*(H)}$ and $d \lambda = \omega$ and Stokes' theorem, 
$$
\int_{S^*Q} \alpha \wedge (d \alpha)^{n-1} \,=\, \int_{S^*(H)} \lambda_H \wedge (d\lambda_H)^{n-1}
\,=\, \int_{D^*(H)} \omega^n .
$$
Therefore, 
$$
\vol_\alpha (S^*Q) \,=\, n!\, \omega_n \, \vol_H^{\HT}(Q) .
$$

\subsection{A few generalizations of Riemannian results to Finsler geodesic flows and Reeb flows}

Given a result in Riemannian geometry or dynamics, it is interesting to see whether this result
extends to reversible or even irreversible Finsler metrics, or even to Reeb flows. 
In this way it becomes clear to which geometry the Riemannian result belongs properly. 
In this section we state a few such results from dynamics. 
For (non-)extensions of results from Riemannian geometry, involving for instance curvatures and spectra, 
we refer to \cite{Bar17, BarCol13, ColNewVer09, Ver99}.

\medskip
\raisebox{.5pt}{\textcircled{\raisebox{-.9pt} {1}}}
(i)
Riemannian 2-tori without conjugate points are flat, \cite{Hop48},
but there are many reversible and irreversible Finsler 2-tori 
without conjugate points that are not flat~\cite[\S 33]{Bus55}.

\smallskip
(ii)
There is (up to scaling) only one Riemannian metric on~$\RP^2$ all of whose geodesics 
are simple closed and of the same length, \cite{Gre63},
but there are many such Finsler metrics on~$\RP^2$, among them reversible ones, see~\cite{Sko55} and also~\cite{Bus57}.
%and~\cite[5.61]{Besse}.

\medskip
\raisebox{.5pt}{\textcircled{\raisebox{-.9pt} {2}}}
(i)
Every Riemannian $2$-sphere carries infinitely many geometrically distinct closed geodesics~\cite{Ban93}.
This result extends to reversible Finsler metrics~\cite{DMMS20},
but not to all Finsler metrics: 
Katok~\cite{Kat73} gave a simple example of a Finsler metric on~$S^2$
with only two geometrically distinct geodesics 
(where the reverse geodesic is counted, since it has different period).

\medskip
(ii)
For reversible Finsler geodesic flows on~$S^2$ that are periodic all orbits have the same period~\cite{GroGro81},
but there exist irreversible periodic Finsler geodesic flows on~$S^2$ whose orbits have different minimal periods,
\cite[p.~143]{Zil83}.

\medskip
\raisebox{.5pt}{\textcircled{\raisebox{-.9pt} {3}}}
(i)
For many compact manifolds~$Q$ (namely so-called essential ones and all surfaces),
there exists a constant~$C>0$ such that every normalized Finsler metric~$F$ on~$Q$
has a closed geodesic of $F$-length at most~$C$, cf.\ \S \ref{ss:systolic}.
This was shown by Loewner, Pu, Croke, and Gromov
%~\cite{Gromov83} 
in the Riemannian case and was extended to the Finsler case by 
\'{A}lvarez Paiva--Balacheff--Tzanev~\cite{apbt16}.
The generalization of length to closed orbits of Reeb flows is the period~$\int_\gamma \alpha$.
It is shown in~\cite{ABHS1, Sag18} that every spherization~$S^*Q$ of a compact manifold~$Q$ 
admits for every $C>0$ a normalized contact form~$\alpha$ 
whose Reeb flow has periodic orbits, but none of period~$\leq C$.

\medskip
(ii)
Theorems~\ref{t:Finslerintro} and \ref{t:mainintro} of this paper show that there is a positive
lower bound for the topological entropy of all normalized Finsler geodesic flows on compact surfaces of genus 
at least two, but not for Reeb flows.

\medskip
\raisebox{.5pt}{\textcircled{\raisebox{-.9pt} {4}}}
(i)
Most of the existence and multiplicity results for closed geodesics on compact Riemannian manifolds~$Q$
extend to Reeb flows on spherizations.
For instance, the isomorphism in~\cite{AbSch06} implies that every Reeb flow on~$S^*Q$
has a closed orbit in every component of the free loop space~$\Lambda Q$.
Also the Gromoll--Meyer theorem, according to which every Riemannian metric has infinitely many prime
closed geodesics provided that the Betti numbers of~$\Lambda Q$ 
%(with respect to some coefficient field)
are unbounded~\cite{GroMey69},
extends to Reeb flows on spherizations~\cite{McL12}.

\medskip
(ii) 
The exponential growth of the fundamental group of a compact manifold~$Q$
or of the rank of the homology of the based loop space of a simply connected compact manifold~$Q$ implies that
the topological entropy of every Riemannian geodesic flow on~$Q$ is positive, 
see \cite{Din71} and \cite[\S 5.3]{Pat99}.
This result generalizes to all Reeb flows on the spherization~$S^*Q$, see \cite{FraSch13, MacSch11},
and in fact to time-dependent Reeb flows, namely contact isotopies that are 
everywhere transverse to the contact distribution~$\xi$, see~\cite{Dah18a}.

\medskip
(iii)
According to the Bott--Samelson theorem~\cite{Bott54, Sam63}, 
the cohomology ring of a Riemannian manifold all of whose geodesics are simple closed and of equal length
must be generated by one element. 
Finer versions are proven in~\cite[Chapter~7]{Bes78}.
All these results hold true for Reeb flows on spherizations, even time-dependent ones,
see~\cite{FLS13, Dah18b}.

%%%%%%%%%%%%%%%%%%%%%%%%%%%%%%%%%%%%%%%%%%%%%%%%%%%%%%%%%%%%%%%%%%%%%%%%%%%%

\section{Properties of the norm growth}
\label{app:norm}
Given a $C^1$-diffeomorphism $\phi$ of a compact manifold~$M$, we define the two real numbers 
\begin{eqnarray*}
\Gamma_+(\phi) &:=& \lim_{n \to +\infty} \frac 1n \log \| d \phi^n \|_{\infty} \,, \\
\Gamma(\phi) &:=& \max \left\{ \Gamma_+(\phi), \Gamma_+(\phi^{-1}) \right\}.
\end{eqnarray*}
Here  $\| \cdot \|_{\infty}$ denotes the supremum norm induced by 
the operator norm on endomorphisms of~$TM$ that is determined by any Riemannian metric 
on~$M$.  
The limit defining~$\Gamma_+$ exists because the sequence 
$\left( \log \| d \phi^n \|_{\infty} \right)$ is subadditive. 
Clearly, $\Gamma_+$ and~$\Gamma$ do not depend on the choice of the Riemannian metric on $M$. 

For a $C^1$-flow $\phi=\{\phi^t\}_{t\in \R}$, we set $\Gamma_+(\phi) := \Gamma_+(\phi^1)$ and $\Gamma(\phi) := \Gamma(\phi^1)$, 
or equivalently
\begin{eqnarray*}
\Gamma_+(\phi) &:=& \lim_{t \to +\infty} \frac 1t \log \| d \phi^t \|_{\infty} \,, \\
\Gamma(\phi) &:=& \lim_{t \to +\infty} \frac 1t \log \max \left\{ \| d \phi^t \|_{\infty} , \| d \phi^{-t} \|_{\infty}\right\} .
\end{eqnarray*}
The following properties of $\Gamma_+$ and $\Gamma$ are analogous to those of
the topological entropy, cf.\ \cite[Prop.\ 3.1.7]{HK95},
except for~(5)
(the topological entropy of a product is the sum of the topological entropies 
of the factors).
The proofs are somewhat easier, since in contrast to the case of topological entropy, 
which is defined in metrical terms, we can use differential calculus. 

\begin{prop} \label{p:Gammaele}
Let $\phi$ be a $C^1$-diffeomorphism of the compact manifold~$M$.
The norm growths $\Gamma_+$ and $\Gamma$ have the following properties.

\begin{enumerate} [label=\mbox{{\rm (\arabic*)}}]
\item
{\rm (Conjugacy invariance)} 
If $\psi$ is another $C^1$-diffeomorphism of $M$, then 
$$
\Gamma_+ (\psi^{-1} \phi \psi) = \Gamma_+ (\phi), \qquad
\Gamma (\psi^{-1} \phi \psi) = \Gamma (\phi) .
$$

%\smallskip
\item
{\rm (Monotonicity)}
If $K$ is a compact submanifold of~$M$ that is invariant under~$\phi$, 
then 
$$
\Gamma_+ (\phi |_K) \leq \Gamma_+ (\phi), \qquad
\Gamma (\phi |_K) \leq \Gamma (\phi) .
$$ 

%\smallskip
\item
{\rm (Decomposition)}
If $M = \bigcup_{i=1}^m K_i$, where $K_1, \dots, K_m$ are compact $\phi$-invariant submanifolds, 
then 
$$
\Gamma_+ (\phi) = \max_{1 \leq i \leq m} \Gamma_+ (\phi |_{K_i}) , \qquad
\Gamma (\phi) = \max_{1 \leq i \leq m} \Gamma (\phi |_{K_i}) .
$$

%\smallskip
\item
{\rm (Elementary time change)} 
$\Gamma_+ (\phi^m) = m \2 \Gamma_+ (\phi)$ for all $m\in \N$ and
$\Gamma (\phi^m) = |m| \2 \Gamma (\phi)$ for all $m \in \Z$. 
For a flow, 
$$
\Gamma_+ \bigl( \{\phi^{st}\}_{t\in \R} \bigr) = 
 s \, \Gamma_+ \bigl( \{\phi^t\}_{t \in \R} \bigr) 
\quad \forall \, s \geq 0, \qquad
\Gamma \bigl( \{\phi^{st}\}_{t\in \R} \bigr) = 
|s| \, \Gamma \bigl( \{\phi^t\}_{t \in \R} \bigr) 
\quad \forall \, s \in \R.
$$

\item
{\rm (Product)}
If $\phi = \phi_1 \times \phi_2$ on $M_1 \times M_2$, then 
$\Gamma_+ (\phi) = \max \left\{ \Gamma_+ (\phi_1), \Gamma_+ (\phi_2) \right\}$ and 
$\Gamma (\phi) = \max \left\{ \Gamma (\phi_1), \Gamma (\phi_2) \right\}$.
\end{enumerate}
\end{prop}

\proof
Properties (2) and (4) are clear, and (1) follows from the chain rule.

For (3) it suffices to show that 
$\Gamma_+ (\phi) \leq \max_{1 \leq i \leq m} \Gamma_+ (\phi |_{K_i})$
in view of~(2).
For each $n \in \N$ choose $i_n \in \{1, \dots, m\}$ such that
$$
\| d \phi^n \|_{\infty} 
\,=\, \| (d \phi|_{K_{i_n}})^n \|_{\infty}  .
$$
There exists $j \in \{1, \dots, m \}$ such that $i_n = j$ for infinitely many~$n$.
For such a $j$ it holds that $\Gamma_+ (\phi) = \Gamma_+ (\phi_j)$. The result for $\Gamma$ follows.

For (5), given norms $\| \1 \cdot \1 \|_1$ on $TM_1$ and $\| \1 \cdot \1 \|_2$ on $TM_2$,
we choose the norm $\| (v_1,v_2) \| = \max \{ \|v_1\|_1, \|v_2\|_2 \}$ on~$TM$.
With this choice, 
$$
\|d\phi^n\|_{\infty} \,=\, \max \left\{ \|d\phi_1^n\|_{\infty}, \|d\phi_2^n\|_{\infty} \right\} 
\quad \forall \, n \in \N.
$$
The equality $\Gamma_+ (\phi) = \max \{ \Gamma_+ (\phi_1), \Gamma_+ (\phi_2)\}$ now follows by arguing as in 
the proof of~(3) for $m=2$. The result for $\Gamma$ follows.
\proofend

The following result improves property (4) above. 
It is an analogue of the time change estimate for the topological entropy in~\cite{Ohno80}.

\begin{prop} \label{p:ER}
Let $X$ be a smooth vector field on a compact manifold~$M$,
and let $f \colon M \to \R$ be a positive smooth function. 
Then 
$$
\Gamma_+ (\phi_{fX}) \leq \|f\|_{\infty} \, \Gamma_+ (\phi_X), \qquad \Gamma (\phi_{fX}) \leq \|f\|_{\infty} \, \Gamma (\phi_X).
$$ 
\end{prop}

\proof
Fix some positive number $\gamma>\Gamma_+(\phi_X)$. By the definition of $\Gamma_+(\phi_X)$ and the continuity of $t\mapsto \|\phi_X^t\|_{\infty}$, there exists a positive number $C_{\gamma}$ such that
\begin{equation}
\label{gbound}
\| d\phi^t_X\|_{\infty} \leq C_{\gamma}\2 e^{\gamma t} \qquad \forall \, t \geq 0.
\end{equation}
The smooth function $F \colon \R \times M \rightarrow \R$ that is defined by
\[
\partial_t F(t,p) = f(\phi_{fX}^t(p)), \qquad F(0,p) = 0, \qquad \forall \, (t,p)\in \R \times M,
\]
is the time change function relating the two flows:
\[
\phi_{f X}^t (p) \,=\, \phi_{X}^{F(t,p)}(p) .
\]
We have
\[
F(t,p) \leq \|f\|_{\infty}\, t \qquad \forall \, (t,p) \in [0,+\infty) \times M,
\]
and from \eqref{gbound} we find
\begin{equation} \label{gbound2}
\bigl\| d\phi_X^{F(t,p)}(p) \bigr\| \,\leq\, 
           C_{\gamma}\2 e^{\gamma \2 \|f\|_{\infty} \2 t} \qquad 
					 \forall \,(t,p)\in [0,+\infty) \times M.
\end{equation}
Here and in the following equations, $d$ denotes differentiation with respect to the spatial variables. By differentiating the identity
\[
\partial_t F(t,p) = f \bigl( \phi_{X}^{F(t,p)}(p)\bigr),
\]
we obtain
\begin{equation} \label{dF}
\begin{split}
\partial_t d F(t,p) &= df \bigl( \phi_X^{F(t,p)}(p) \bigr) \circ d\phi_X^{F(t,p)}(p) + 
      df \bigl( \phi_X^{F(t,p)}(p) \bigr) \bigl[ X(\phi_X^{F(t,p)}(p)) \bigr] dF(t,p) 
			\\[0.2em] 
			&= df \bigl( \phi_X^{F(t,p)}(p) \bigr) \circ d\phi_X^{F(t,p)}(p) + 
			df \bigl( \phi_{fX}^{t}(p) \bigr) \bigl[X(\phi_{fX}^{t}(p)) \bigr] dF(t,p) \\ 
			&=  df \bigl( \phi_X^{F(t,p)}(p) \bigr) \circ d\phi_X^{F(t,p)}(p) + 
			\frac{1}{f(\phi_{fX}^t(p))} \2 \frac{d}{dt} f(\phi^t_{fX}(p)) \, dF(t,p).
\end{split}
\end{equation}
Let $v\in T_p M$ be a vector of norm one and set
\[
u(t):=  d F(t,p)[v], \qquad \alpha(t):= f(\phi_{fX}^t(p)).
\]
From \eqref{dF} and \eqref{gbound2} we find
\[
u'(t) - \frac{\alpha'(t)}{\alpha(t)}\2 u(t) \,=\, 
 df \bigl( \phi_X^{F(t,p)}(p) \bigr) \circ d\phi_X^{F(t,p)}(p)[v] 
 \,\leq\, C_{\gamma}\2 \|df\|_{\infty} \,e^{\gamma \2 \|f\|_{\infty} \2t} 
 \qquad \forall \,t \geq 0.
\]
Multiplying both sides of this inequality by $1/\alpha(t)$, which is bounded from above by $1/\min f$, we obtain
\[
\frac{d}{dt} \frac{u(t)}{\alpha(t)} \,\leq\, \frac{C_{\gamma}\2 \|df\|_{\infty}}{\alpha(t)}\, e^{\gamma \2 \|f\|_{\infty}\2 t} \,\leq\, \frac{C_{\gamma}\2 \|df\|_{\infty}}{\min f} \,e^{\gamma \2 \|f\|_{\infty} \2 t} \qquad \forall \, t\geq 0.
\]
Since $u(0)=0$, integration on the interval $[0,t]$ and multiplication by $\alpha(t)$, which is bounded from above by $\|f\|_{\infty}$, yields
\[
u(t) \,\leq\, \alpha(t) \,\frac{C_{\gamma} \2 \|df\|_{\infty}}{\min f}  
       \int_0^t e^{\gamma \2 \|f\|_{\infty} \2 s} \, ds  
			\,\leq\, \frac{C_{\gamma} \2 \|df\|_{\infty}}{\gamma \min f} \,e^{\gamma \2 \|f\|_{\infty}\2 t} \qquad \forall \, t\geq 0.
\]
Recalling that $u(t)=  d F(t,p)[v]$ where $v$ is an arbitrary unit tangent vector at~$p$, we have proven the bound
\begin{equation} \label{gbound3}
\|dF(t,p)\| \,\leq\, C_{\gamma}'\,  e^{\gamma \2 \|f\|_{\infty} \2t} 
   \qquad \forall \, (t,p)\in [0,+\infty) \times M,
\end{equation}
where $C_{\gamma}'= \frac{C_{\gamma} \2 \|df\|_{\infty}}{\gamma \min f}$. 
Thanks to the identity
\[
d\phi_{fX}^t(p) = d\phi_X^{F(t,p)}(p) + X \bigl(\phi_X^{F(t,p)}(p) \bigr) \,dF(t,p),
\]
\eqref{gbound2} and \eqref{gbound3} imply
\[
\|d\phi_{fX}^t \|_{\infty} \,\leq\, 
        \left( C_{\gamma}  + \|X\|_{\infty} \2 C_{\gamma}' \right) e^{\gamma \2 \|f\|_{\infty} \2t} \qquad \forall \, t \geq 0,
\] 
and hence
\[
\Gamma_+(\phi_{fX}) \,=\, \lim_{t\rightarrow \infty} \frac{1}{t} \log \|d\phi_{fX}^t \|_{\infty}  \,\leq\, \gamma \2\|f\|_{\infty}.
\]
Since $\gamma$ is an arbitrary number that is larger than $\Gamma_+(\phi_X)$, we deduce the desired bound
\[
\Gamma_+(\phi_{fX}) \,\leq\, \|f\|_{\infty} \, \Gamma_+(\phi_X).
\]
The analogous bound for $\Gamma$ immediately follows.
\proofend

\section{Dynamically trivial deformations of Finsler metrics} 
\label{app:Finslerisotop}

The aim of this section is to show that it is always possible to deform any given regular Finsler metric on a closed manifold of dimension at least two into a new Finsler metric which is not isometric to the original one but whose geodesic flow is nevertheless conjugate to the one of the original metric by a smooth time-preserving conjugacy.

Let $Q$ be a closed manifold and $\lambda = \sum_j p_j \2 dq_j$ be the canonical Liouville one-form on~$T^*Q$. A diffeomorphism $\varphi \colon T^*Q \rightarrow T^*Q$ is a symplectomorphism, i.e.\ preserves the symplectic form~$d\lambda$, 
if and only if the one-form $\varphi^* \lambda - \lambda$ is closed. 
When this form is exact, $\varphi$ is called an exact symplectomorphism. 

Any diffeomorphism $\theta \colon Q \rightarrow Q$ lifts to an exact symplectomorphism 
$T^*\theta$ of~$T^* Q$ by setting
\[
T^* \theta (q,p) := \bigl(\theta(q), d\theta(q)^{-*} [p] \bigr),
\]
where the symbol $-*$ denotes the inverse of the adjoint. 
Indeed, $T^* \theta$ preserves the Liouville form $\lambda$.
Actually, every diffeomorphism of $T^*Q$ preserving $\lambda$ is the lift of some diffeomorphism of~$Q$,
see for instance Proposition~2.1 and Homework 3.3 in~\cite{Can01}.

Therefore, the diffeomorphism group of $Q$ embeds naturally into the group of exact 
symplectomorphisms of~$T^* Q$. The latter group is much larger, though. 
For instance, any symplectomorphism $\varphi \colon T^* Q \rightarrow T^* Q$ which is compactly supported is exact: 
The exactness of $\varphi^* \lambda - \lambda$ is equivalent to the fact that this closed one-form has vanishing integral on any closed curve in $T^*Q$, and if $\varphi$ is compactly supported this is certainly true since every closed curve in $T^* Q$ is freely homotopic to a closed curve taking values in the complement of any given compact subset of $T^* Q$. Producing non-trivial compactly supported symplectomorphisms 
of~$T^* Q$ is very easy, as one can integrate the vector field that is induced by a compactly supported Hamiltonian function. Actually, the
time-one map of any possibly time-dependent Hamiltonian vector field with globally defined flow is exact even without assuming the support to be compact. Indeed, denoting by $X_t$ the Hamiltonian vector field of the time-dependent function $H_t$ and by $\phi_H^t$ its flow, 
we compute with the help of Cartan's identity
$$
\frac{d}{dt} (\phi_H^t)^* \lambda \,=\, 
(\phi_H^t)^* \left( \imath_{X_t} d\lambda + d \imath_{X_t} \lambda \right) \,=\,
(\phi_H^t)^* d \left( \imath_{X_t} \lambda - H_t \right)
$$
and hence
$$
(\phi_H^1)^* \lambda-\lambda \,=\, 
d \int_0^1 \left( \imath_{X_t} \lambda - H_t \right) \circ \phi_H^t\, dt .
$$
If $Q$ has non-trivial de-Rham cohomology in degree one, then any closed one-form $\eta$ on $Q$ that is not exact induces a symplectomorphism
\[
T^*Q \rightarrow T^* Q, \qquad (q,p) \mapsto (q,p+\eta(q)),
\]
which is symplectically isotopic to the identity but not exact.

Recall from \S \ref{ss:RSG} that to a fiberwise starshaped hypersurface $S \subset T^*Q$
we associate the Reeb flow $\phi_{\alpha}$ of $\alpha := \lambda|_S$, 
namely the flow generated by the vector field $R_\alpha$ defined by
\[
d\alpha (R_\alpha, \cdot) = 0, \qquad \alpha (R_{\alpha}) = 1. 
\]
Equivalently, $\phi_\alpha$ is the restriction to~$S$ of the Hamiltonian flow on~$T^*Q$
of the function $T^*Q \to \R$ that is fiberwise positively homogeneous 
of degree~1 and equal to~$1$ on~$S$.

The next result tells us that the Reeb dynamics on fiberwise starshaped hypersurfaces does not change when we transform them by exact symplectomorphisms. 

\begin{prop} \label{p:Fdef}
Let $Q$ be a closed manifold, $S$ a fiberwise starshaped hypersurface of~$T^*Q$
and $\varphi: T^* Q \rightarrow T^* Q$ an exact symplectomorphism 
such that $S' := \varphi (S)$ is also fiberwise starshaped.
Then there exists a diffeomorphism $\psi \colon S \to S'$
such that
\[
\psi^* \left( \lambda|_{S'} \right) \,=\, \lambda|_S .
\]
In particular, the diffeomorphism $\psi$ is a smooth time-preserving conjugacy between the Reeb flows of $S$ and $S'$. 
\end{prop}

The above proposition is proven at the end of this appendix. We now discuss its consequences concerning Finsler geodesic flows. 
Let $F$ be a regular Finsler metric on~$Q$ and denote by $S^*(F)\subset T^* Q$ the corresponding unit cotangent sphere bundle, which we here consider as the domain 
of the geodesic flow of~$F$. The push-forward $\theta_* F$ of this Finsler metric 
by a diffeomorphism $\theta \colon Q \rightarrow Q$ is another Finsler metric on~$Q$. 
By construction, the Finsler manifolds $(Q,F)$ and $(Q,\theta_* F)$ are isometric, 
and their geodesic flows are smoothly conjugate. 
Indeed, the cotangent lift $T^* \theta$ restricts to a diffeomorphism 
from $S^*(F)$ to~$S^*(\theta_* F)$ conjugating the two geodesic flows. 
Diffeomorphisms of~$Q$ induce ``metrically trivial deformations'' of~$F$.

Now consider a more general exact symplectomorphism 
$\varphi \colon T^* Q \rightarrow T^* Q$. If $\varphi$ is $C^2$-close to the identity, then the image of $S^*(F)$ under~$\varphi$ is still fiberwise strictly convex and hence can be seen as the unit cotangent sphere bundle of another Finsler metric~$F'$:
\[
\varphi(S^*(F)) = S^*(F').
\]
The fact that $\varphi$ is volume preserving implies that the Finsler metrics $F$ and~$F'$ have the same Holmes--Thompson volume. 
By Proposition~\ref{p:Fdef} above, the geodesic flows of 
the Finsler metrics $F$ and~$F'$ are conjugate by a smooth time-preserving conjugacy. Therefore, the exact symplectomorphism~$\varphi$ induces a 
``dynamically trivial deformation'' $F'$ of~$F$. 
However, $F$ and~$F'$ need not be isometric. 
For instance, $F$ could be a Riemannian metric, meaning that $S^*(F)$ is a field of centrally symmetric ellipsoids. In the special case in which $\varphi$ is 
the cotangent lift of a diffeomorphism of~$Q$, $\varphi$ maps fibers into fibers and acts linearly on them, so $\varphi(S^*(F))$ is still a field of centrally symmetric ellipsoids and the metric $F'$ is still Riemannian. But for a more general exact symplectomorphism~$\varphi$, there is no reason why $\varphi(S^*(F))$ should still be a field of ellipsoids, so the new Finsler metric~$F'$ is in general not Riemannian. 

Let $F$ again be an arbitrary regular Finsler metric on~$Q$ and let $q$ be a point in~$Q$. By acting just by diffeomorphisms of~$Q$, we obtain Finsler metrics~$F'$ on~$Q$ 
which are isometric to~$F$ and whose unit sphere $S_q^*(F')$ at~$q$ is the image of 
some $S^*_{q'}(F)$ by some linear isomorphism. 
Therefore, all the possible cotangent unit spheres at~$q$ of Finsler metrics which are constructed in this way belong to the finite dimensional family of sets
\[
\left\{ 
L(S^*_{q'}(F)) \mid q'\in Q, \; L \colon T_{q'} Q \rightarrow T_q Q \mbox{ linear isomorphism}
\right\}.
\]
Instead, by acting by more general exact symplectomorphisms, we can get a new Finsler metric~$F'$ whose geodesic flow is still conjugate to the one of~$F$ 
but such that $S^*_q(F')$ is an arbitrary convex hypersurface which is sufficiently close to~$S^*_q(F)$. When $S^*_q(F')$ does not belong to the above finite dimensional family, the metric~$F'$ cannot be isometric to~$F$.

Let us prove this fact. Let $S'_q$ be a strictly convex hypersurface in~$T^*_q Q$. 
Using a cotangent local chart, we identify $\pi^{-1}(U)$, where $U$ is a neighborhood 
of~$q$ in~$Q$ and $\pi \colon T^* Q \rightarrow Q$ denotes the footpoint projection, 
with $T^* \R^n = \R^n \times (\R^n)^*$ in such a way that $T^*_q Q$ is identified 
with $\{0\} \times (\R^n)^*$.  Then both $S_q^*(F)$ and~$S'_q$ are subsets of 
$\{0\} \times (\R^n)^* \cong (\R^n)^*$. Let $\xi \colon (\R^n)^* \rightarrow (\R^n)^*$ 
be a compactly supported smooth vector field whose flow at time one 
maps $S_q^*(F)$ to~$S'_q$ (one easily achieves this by a radial vector field). 
Consider the Hamiltonian function
\[
H(q,p) \,:=\, \chi(q) \, \langle \xi(p), q \rangle, \qquad (q,p) \in \R^n \times (\R^n)^*,
\]
where $\langle\cdot,\cdot \rangle$ denotes the duality pairing and 
$\chi \colon \R^n \rightarrow \R$ is a smooth compactly supported function taking the value~1 near~0. 
This function extends to a compactly supported smooth function on~$T^*Q$ and the corresponding Hamiltonian flow leaves $T_q^*Q$ invariant and restricts to the flow 
of~$\xi$ on it. Therefore, denoting by $\varphi \colon T^* Q \rightarrow T^* Q$ 
the time-one map of this Hamiltonian flow, we obtain an exact symplectomorphism~$\varphi$ such that $\varphi(S_q^*(F))=S'_q$.  
If $S_q'$ is $C^3$-close to $S_q^*(F)$, meaning that 
\[
S_q' \,=\, \{ e^{f(p)} \2 p \mid p\in S_q^*(F)\}
\]
for some $C^3$-small function $f$ on $S_q^*(F)$, then the vector field~$\xi$ can be chosen to be $C^3$-small, the Hamiltonian vector field of~$H$ is $C^2$-small and hence $\varphi$ is $C^2$-close to the identity. Under this closeness assumption, $\varphi(S^*(F))$ is fiberwise strictly convex and hence is the unit cotangent sphere bundle of a Finsler metric~$F'$ such that $S_q^*(F') = S_q'$. 

Note that if the Finsler metric $F$ is reversible and $S_q'$ is centrally symmetric, 
then the dynamically equivalent Finsler metric~$F'$ as above can be chosen to be reversible by requiring $\xi$ to be an odd map. 

Summarizing: Given any regular (reversible) Finsler metric $F$, 
we obtain a large family of non-isometric (reversible) Finsler metrics~$F'$ 
having the same Holmes--Thompson volume and whose geodesic flows are smoothly 
conjugate to the one of $F$; $F'$ can be chosen to be non-Riemannian 
even if $F$ is Riemannian. 
The topological entropy, 
the volume growth used in Appendix~\ref{ss:Yomdin},
the norm growth from Appendix~\ref{app:norm}, 
the length spectrum and all the dynamical invariants of~$F'$ 
coincide with those of~$F$. 
If one uses exact symplectomorphisms that are isotopic to the identity, 
then also the marked length spectrum is preserved, and 
if one starts with a reversible Finsler metric~$F$ of negative flag curvature 
and deforms it to a reversible Finsler metric~$F'$ that still has negative flag curvature, 
then also the volume entropy of~$F$ and~$F'$ agrees, 
since for Finsler metrics of negative flag curvature 
the volume entropy is equal to the topological entropy 
by the equality case in Manning's theorem, see \cite[Theorem~15]{Bar17}.
One deformation to Finsler metrics 
with equal marked length spectrum and equal volume entropy 
was constructed in~\cite{ColNewVer09} 
for the Riemannian metrics of constant curvature on hyperbolic surfaces.
Our discussion here shows that this is a general phenomenon and is a manifestation of the fact that a neighborhood of the identity on the group of exact symplectomorphisms 
of~$T^*Q$ can transform a Finsler metric in ways in which diffeomorphisms of~$Q$ cannot.

\bigskip \noindent
{\it Proof of Proposition \ref{p:Fdef}.}
The proof is given in \cite[\S 8]{AbMa15} for the case of convex hypersurfaces 
in a Hilbert space. That proof is readily adapted to our situation. We give the proof for the reader's convenience.
 
Since the symplectomorphism $\varphi$ is assumed to be exact, there exists a smooth function $h \colon T^*Q \to \R$
such that
\[
\varphi^* \2 \gl \,=\, \gl + dh.
\]
We abbreviate $\alpha = \lambda |_S$ and $\alpha' = \lambda|_{S'}$. The pull-back of $\ga'$ to~$S$,
that is the 1-form
\begin{equation} \label{e:fa'}
(\varphi|_{S})^* (\ga') \,=\, \alpha + dh |_S ,
\end{equation}
is a contact form on~$S$, and its Reeb vector field is
$$
\varphi^* (R_{\ga'}) \,=\, R_{\ga + dh|_{S}} .
$$
Since $\alpha$ and $\alpha + dh |_S$ have the same differential, we have
\begin{equation} \label{e:Rf}
R_{\ga + dh|_{S}} \,=\, f \1 R_{\ga}
\end{equation}
for a nowhere vanishing function $f$ on $S$. If $Y= p\, \partial_p$ denotes the canonical Liouville vector field on~$T^*Q$, the identity $\imath_Y d\lambda = \lambda$ implies that for every $x\in S$ the one-form $(\imath_{R_{\alpha}} d\lambda) (x)$, 
whose kernel is the tangent space~$T_x S$, is negative on tangent vectors which are pointing outwards, i.e.\ belong to the half-space containing~$Y(x)$. 
Similarly, $\imath_{R_{\alpha'}} d\lambda$ is negative on outward pointing vectors 
based at~$S'$. By the identity
\[
\varphi^* ( \imath_{R_{\alpha'}} d\lambda ) \,=\, \imath_{\varphi^*(R_{\alpha'})} \varphi^*(d\lambda) \,=\,  \imath_{R_{\ga + dh|_{S}}} d\lambda
\]
and by the fact that the differential of $\varphi$ maps outward pointing vectors based 
at~$S$ to outward pointing vectors based at~$S'$, we deduce that the above one-form is negative on outward pointing vectors based at~$S$. 
Since also $\imath_{R_{\alpha}} d\lambda$ is negative on these vectors, 
the function~$f$ appearing in~\eqref{e:Rf} is everywhere positive.

By applying the 1-form $\alpha + dh|_S$ to~\eqref{e:Rf} we obtain
$1 = f \2 \bigl( 1 + dh (R_\alpha) \bigr)$.
Therefore, 
\begin{equation} \label{e:tg0}
1 + t\, dh (R_{\alpha}) \,>\, 0 \qquad \forall \, t \in [0,1].
\end{equation}

We now apply Moser's homotopy method:
We look for a smooth time-dependent vector field $X_t$ on~$S$
of the form $X_t = \chi_t\2 R_\ga$, where $\chi_t$ is a family of functions on~$S$,
such that the flow $\eta_t$ of~$X_t$ satisfies
\begin{equation} \label{e:Moser1}
\eta_t^* \left(\alpha + t\, d h \right) \,=\, \alpha \qquad \forall \, t \in [0,1].
\end{equation}
Since $\eta_0^* \2 \alpha = \alpha$, this identity is equivalent to 
\begin{equation} \label{e:Moser2}
\frac{d}{dt} \bigl( \eta_t^* (\alpha + t\, d h) \bigr) \,=\, 0 
                   \qquad \forall \, t \in [0,1].
\end{equation}
Using Cartan's identity we compute
\begin{eqnarray*}
\frac{d}{dt} \bigl( \eta_t^* \left(\alpha + t\, d h \right) \bigr) &=& 
\eta_t^* \2 \bigl( \cl_{X_t} ( \alpha + t\, d h ) + dh \bigr) \\
&=& \eta_t^* \2 \bigl( \imath_{X_t} d (\alpha + t\, d h)
         + d \imath_{X_t} (\alpha + t\, d h)  + dh \bigr) .
\end{eqnarray*}
Since $\imath_{X_t} d\alpha = \chi_t \, \imath_{R_{\ga}} d\alpha = 0$,
it follows that \eqref{e:Moser2} holds if and only if
\begin{equation} \label{e:diota}
d \imath_{X_t} (\alpha + t\, d h)  + dh \,=\, 0.
\end{equation}
Since $\imath_{X_t} (\alpha + t\, d h) = \chi_t \2 \bigl( 1 + t \,dh (R_{\ga}) \bigr)$,
equation~\eqref{e:diota} is satisfied if we take
$$
\chi_t \,:=\, - \frac{h |_S}{1 + t \,dh (R_{\ga})} , \qquad t \in [0,1].
$$
By \eqref{e:tg0} the functions $\chi_t$ are well-defined.

In view of \eqref{e:fa'} and \eqref{e:Moser1}, 
the diffeomorphism $\psi := \varphi \circ \eta_1 \colon S \to S'$ 
satisfies
$$
\psi^* \alpha' \,=\, \eta_1^* \left(\varphi^* \alpha' \right) \,=\, 
      \eta_1^* \left( \alpha + dh \right) \,=\, \alpha ,
$$
as required.
\proofend

\begin{comment}
1-homoegenous extension ...

\begin{rem}
{\rm
The diffeomorphism $\psi$ in general does not extend to~$T^*Q$.
Indeed, a diffeomorphism of $T^*Q$ that preseves $\lambda$
must be the lift of a diffeomorphism~$\sigma$ of~$Q$,
$$
(q,p) \mapsto \left( \sigma (q),  \bigl( [d\sigma (q)]^* \bigr)^{-1} p \right) ,
$$ 
see~\cite[Homework 3.3]{Can01}.
}
\end{rem}
\end{comment}

%%%%%%%%%%%%%%%%%%%%%%%%%%%%%%%%%%%%%%%%%%%%%%%%

%%%%%%%%%%%%%%%%%%%%%%%%%%%%%%%%%%%%%%%%%%%%%%%%

\end{document}